\documentclass[12pt]{article}

\usepackage{amsfonts,graphicx,amsmath,amssymb,amsthm,geometry,
hyperref,color,nicefrac,enumerate,relsize,titling}
\usepackage[latin1]{inputenc}
\usepackage{bbm} 

\newcommand{\R}{\mathbb{R}}
\newcommand{\N}{\mathbb{N}}
\newcommand{\E}{\mathbb{E}}
\renewcommand{\P}{\mathbb{P}}
\newcommand{\F}{\mathbb{F}}
\newcommand{\Exp}[1]{\E \!\left[ #1 \right]}
\newcommand{\EXp}[1]{\E [ #1 ]}
\newcommand{\EXP}[1]{\E \big[  #1 \big]}
\newcommand{\EXPP}[1]{\E \Big[  #1 \Big]}
\newcommand{\EXPPP}[1]{\E \bigg[  #1  \bigg]}
\newcommand{\EXPPPP}[1]{\E \Bigg[ #1 \Bigg]}

\newcommand{\qandq}{\qquad\text{and}\qquad}
\newcommand{\andq}{\text{and}\qquad}
\newcommand{\norm}[1]{\Vert #1 \Vert}

\newcommand{\abs}[1]{| #1 |}  
\newcommand{\Abs}[1]{\big| #1 \big|}

\renewcommand{\lll}{\langle} 
\newcommand{\rrr}{\rangle}
\newcommand{\llll}{\big\langle}
\newcommand{\rrrr}{\big\rangle}
\newcommand{\lllll}{\Big\langle}
\newcommand{\rrrrr}{\Big\rangle}
\newcommand{\ind}[1]{\mathbbm{1}_{#1}}

\newtheorem{lemma}{Lemma}[section]
\newtheorem{remark}[lemma]{Remark}
\newtheorem{cor}[lemma]{Corollary}

\newtheorem{theorem}[lemma]{Theorem}
\newtheorem{prop}[lemma]{Proposition}
\newtheorem{setting}[lemma]{Setting}

\begin{document}
\setlength{\droptitle}{-6em}
\title{Strong error analysis for stochastic\\ gradient descent optimization algorithms}
\author{Arnulf Jentzen$^1$,
	Benno Kuckuck$^2$,\\
	Ariel Neufeld$^3$, 
	and 
	Philippe von Wurstemberger$^{4}$
	\bigskip
	\\
	\small{$^1$Department of Mathematics, 
		ETH Zurich,}
	\\
		\small{e-mail: 
		arnulf.jentzen@sam.math.ethz.ch}
	\smallskip
	\\
	\small{$^2$Department of Mathematics,  
		Universit\"at D\"usseldorf,}
	\\
	\small{e-mail:
		kuckuck@math.uni-duesseldorf.de}
	\smallskip
	\\
	\small{$^3$Department of Mathematics, 
		ETH Zurich,}
	\\
		\small{e-mail: 
		ariel.neufeld@math.ethz.ch} 
\smallskip
\\
\small{$^4$Department of Mathematics, 
	ETH Zurich,}
\\
\small{e-mail: 
	vwurstep@student.ethz.ch}	
}
\maketitle

\begin{abstract}
\vspace{0.2cm}
Stochastic gradient descent (SGD) optimization algorithms are key ingredients in a series of machine learning applications. In this article we perform a rigorous strong error analysis for SGD optimization algorithms. In particular, we prove for every arbitrarily small $\varepsilon \in (0,\infty)$ and every arbitrarily large $p\in (0,\infty)$ that the considered SGD optimization algorithm converges in the strong $L^p$-sense with order $\nicefrac{1}{2}-\varepsilon$ to the global minimum of the objective function of the considered stochastic approximation problem under standard convexity-type assumptions on the objective function and relaxed assumptions on the moments of the stochastic errors appearing in the employed SGD optimization algorithm. 
%
%
The key ideas in our convergence proof are, first, to employ techniques from the theory of Lyapunov-type functions for dynamical systems to develop a general convergence machinery for SGD optimization algorithms based on such functions, then, to apply this general machinery to concrete Lyapunov-type functions with polynomial structures, and, thereafter, to perform an induction argument along the powers appearing in the Lyapunov-type functions in order to achieve for every arbitrarily large $ p \in (0,\infty) $ strong $ L^p $-convergence rates.
This article also contains an extensive review of results on SGD optimization algorithms in the scientific literature.
\end{abstract}
\newpage
\tableofcontents

\section{Introduction}

%
%
%
%

Stochastic gradient descent (SGD) type optimization algorithms are fundamental tools in many machine and deep learning applications such as object and speech recognition or image analysis (cf., for example, Ruder \cite{Ruder16}). To ensure the performance of such algorithms it is important to analyze their approximation errors and, in particular, to investigate their speeds of convergence. 
A very common approach to study SGD type optimization algorithms is to formulate them as so-called stochastic approximation algorithms (SAAs). SAAs were first introduced in Robbins \& Monro \cite{RobbinsMonro51} and SAAs and SGD type optimization algorithms, respectively, have been widely studied in the scientific literature; 
%
cf., for example,
 \cite{AmariParkFukumizu00,BrosseDurmusMoulinesSabanis17,Fabian68,LeRouxFitzgibbon10,LeRouxManzagolBengio08,Martens10,MokkademPelletier11,NeelakantanETAL15,Nesterov83,PenningtonSocherManning14,Ruppert82,Schraudolph02,Shalev-ShwartzTewari09,VinyalsPovey12} and the references mentioned therein for the derivation and the proposal of SAAs,
cf., for example,  \cite{BordesBottouGallinari09,DauphinDeVriesBengio15,DauphinPascanuGulcehreChoGanguliBengio14,DefossezBach17,Dozat16,DuchiHazanSinger11,KingmaBa14,LangfordLiZhang09,LeRouxSchmidtBach12,McmahanStreeter14,NiuRechtChristopherWright11,Polyak98,PolyakJuditsky92,PolyakTsypkin80,Schraudolph99,SchraudolphYuGunter07,Shalev-ShwartzShingerSrebroCotter11,Sohl-DicksteinPooleGanguli14,Sohl-DicksteinPooleGanguli14,Zeiler12,ZhangChoromanskaLeCun15} and the references mentioned therein for the derivation and the proposal of SGD type optimization algorithms,
cf., for example, \cite{Amari98,Bottou99,BottouLeCun05,BroadieCicekZeevi09,ChapelleErhan11,ChauKumarRasonyiSabanis17,Chung54,DereichGronbach2017,Dippon03,DipponRenz97,GaposhkinKrasulina74,Gerencser99,KieferWolfowitz52,KomlosRevesz73,KondaTsitsiklis04,KushnerHuang79,KushnerYang93,LaiRobbins78,LeBretonNovikov94,LeBretonNovikov95,LecuyerYin98,MartensSutskeverSwersky12,NemirovskiJuditskiLanShapiro09,Nevelson86,Pelletier98b,Pelletier98,Ruppert88} and the references mentioned therein
 for numerical simulations and convergence rates proofs for SAAs,
cf., for example, \cite{Bach14,BachMoulines11,BachMoulines13,Bottou12,BottouBousquet11,BottouLeCun04,DarkenChangMoody92,DieuleveutDurmusBach17,InoueParkOkada03,LeCunBottouOrrMuller98,MizutaniDreyfus10,PascanuBengio13,Pillaud-VivienRudiBach17,Qian99,RakhlinShamirSridharan12,RattraySaadAmari98,SutskeverMartensDahlHinton13,Sutton86,TangMonteleoni15,Xu11,Zhang04} and the references mentioned therein 
for numerical simulations and convergence rates proofs for
SGD type optimization algorithms,
cf., for example,  \cite{Benaim99,BenevisteMetivierPriouret90,BhatnagarPrasadPrashanth13,Borkar08,Chen02,Duflo96,Duflo97,ClarkKushner78,KushnerYin97,Lai03,LjungPflugWalk92,Ruppert91,Schmetterer61} and the references mentioned therein for overview articles and monographs on SAAs,
cf., for example, \cite{BercuFort13,BottouCourtisNocedal16,Ruder16} and the references mentioned therein for overview articles on SGD type optimization algorithms,
and
cf., for example, \cite{DeanETAL12,DengETAL13,Graves13,GravesMohamedHinton13,HintonETAL12,HintonSalakhutdinov06,KrizhevskySutskeverHinton12,LeCunBottouBengioHaffner98,SchaulZhangLeCun12} and the references mentioned therein 
for applications involving  neural networks and SGD type optimization algorithms.

%
%
In this paper we develop a rigorous strong error analysis for SAAs and SGD optimization algorithms. In particular, we prove for every arbitrarily small $\varepsilon \in (0,\infty)$ and every arbitrarily large $p \in (0,\infty)$ that the considered SGD optimization algorithm converges in the strong $L^p$-sense with order $\nicefrac{1}{2} - \varepsilon$ to the global minimum of the objective function of the considered stochastic approximation problem under standard convexity-type assumptions on the objective function (cf.\ \eqref{SGD_intro:assumption3} in Theorem~\ref{SGD_intro} below, \eqref{Lp_theorem:assumption1} in Theorem~\ref{Lp_theorem} in Subsection~\ref{subsection:Lp} below, and, e.g., Dereich \& Mueller-Gronbach \cite[Assumption~A.1]{DereichGronbach2017}) and relaxed assumptions on the moments of the stochastic errors appearing in the employed SGD optimization algorithm (cf.\ \eqref{SGD_intro:assumption4} in Theorem~\ref{SGD_intro} below and \eqref{Lp_theorem:assumption4} in Theorem~\ref{Lp_theorem} in Subsection~\ref{subsection:Lp} below). To illustrate the findings of this article, we now present in the following theorem a special case of our strong error analysis for SAAs and SGD optimization algorithms (cf.\ Theorem~\ref{Lp_theorem} in Subsection~\ref{subsection:Lp} below and Corollary~\ref{SGD} in Subsection~\ref{subsection:SGD} below).
%
\pagebreak
\begin{theorem}
\label{SGD_intro}
Let $d \in \N$, $p, \alpha,\kappa,c \in (0,\infty)$, $ \nu \in (0,1)$, $q=\min(\{2,4,6,\dots\} \cap [p,\infty))$, $\xi, \vartheta \in \R^d$,
let $ ( \Omega , \mathcal{F}, \P) $ be a probability space, let $(S, \mathcal{S})$ be a measurable space, let $X_{n} \colon \Omega \to S$, $n\in\N$, be i.i.d.\ random variables, let $F = ( F(\theta,x) )_{\theta \in \R^d, x \in S} \colon \R^d \times S \to \R$ be $(\mathcal{B}(\R^{d}) \otimes \mathcal{S}) / \mathcal{B}(\R)$-measurable, assume for all $x \in S$ that $(\R^d \ni \theta \mapsto F(\theta,x) \in \R) \in C^1(\R^d, \R)$, assume for all $\theta \in \R^d$ that
\vspace{-.2cm}
\begin{equation}
\label{SGD_intro:assumption2}
\EXP{|F(\theta,X_1)| + \norm{(\nabla_{\theta}F)(\theta, X_1)}_{\R^d}}<\infty,
\end{equation}
\vspace{-.6cm}
\begin{equation}
\label{SGD_intro:assumption3}
\lll \theta - \vartheta, \Exp{(\nabla_{\theta}F)(\theta, X_1)} \rrr_{\R^d}   \geq c\max\!\big\{\norm{\theta - \vartheta}_{\R^d}^2, \norm{\Exp{(\nabla_{\theta}F)(\theta, X_1)}\!}_{\R^d}^2 \big\},  
\end{equation}
\begin{equation}
\label{SGD_intro:assumption4}
\EXP{\norm{(\nabla_{\theta}F)(\theta, X_1) - \Exp{(\nabla_{\theta}F)(\theta, X_1)}\!}_{\R^d}^q } \leq \kappa \big(1+ \norm{\theta}_{\R^d}^q\big),
\end{equation}
and let $\Theta \colon \N_0 \times \Omega \to \R^d$ be the stochastic process which satisfies for all $n \in \N$ that 
\begin{equation}
\label{SGD_intro:assumption5}
\Theta_0 = \xi   \qquad \text{ and } 
\qquad \Theta_n = \Theta_{n-1} - \tfrac{\alpha}{n^\nu}(\nabla_{\theta}F)(\Theta_{n-1},X_n).
\end{equation}
Then 
\begin{enumerate}[(i)]
\item \label{SGD_intro:item1}
it holds that $\big\{\theta \in \R^d \colon \big(\Exp{F(\theta, X_1)} = \inf\nolimits_{v \in \R^d}\Exp{F(v, X_1)}\!\big)\!\big\} = \{\vartheta\}$ and 
\item \label{SGD_intro:item2}
there exists $C \in (0,\infty)$ such that for all $n \in \N$ it holds that
\begin{equation}
\left(\EXP{ \norm{\Theta_n-\vartheta}_{\R^d}^p }\right)^{\nicefrac{1}{p}} \leq C n^{-\nicefrac{\nu}{2}}.
\end{equation}
\end{enumerate} 
\end{theorem}
Theorem~\ref{SGD_intro} is an immediate consequence of Jensen's inequality and Corollary~\ref{SGD} in  Subsection~\ref{subsection:SGD} below. Corollary~\ref{SGD}, in turn, follows from 
Theorem~\ref{Lp_theorem} in Subsection~\ref{subsection:Lp} below, which is the main result of this article.
%
A strong convergence result related to Theorems~\ref{SGD_intro} and \ref{Lp_theorem} in this article has been obtained in Dereich \& Mueller-Gronbach \cite[Theorem 2.4]{DereichGronbach2017} (cf.\ also \cite[Proposition 2.2]{DereichGronbach2017}). A key difference between Theorem~2.4 in \cite{DereichGronbach2017} and Theorems~\ref{SGD_intro} and \ref{Lp_theorem} in this article is the hypothesis on the moments of the stochastic errors appearing in the employed SAA (cf.\ item (ii) in Assumption A.2 in \cite{DereichGronbach2017} with \eqref{SGD_intro:assumption4} in Theorem~\ref{SGD_intro} and \eqref{Lp_theorem:assumption4} in Theorem~\ref{Lp_theorem} in this article). More formally, in Theorem 2.4 in \cite{DereichGronbach2017} the stochastic errors appearing in the employed SAA are assumed to be bounded in the state space variable $ \theta \in \R^d $ (cf.\ item (ii) in Assumption A.2 in \cite{DereichGronbach2017}) while Theorems~\ref{SGD_intro} and \ref{Lp_theorem} in this article allow the $ L^p $-norm of the stochastic errors to grow linearly in the state space variable $ \theta \in \R^d $ (cf.\ \eqref{SGD_intro:assumption4} in Theorem~\ref{SGD_intro} and \eqref{Lp_theorem:assumption4} in Theorem~\ref{Lp_theorem} in this article). This relaxed hypothesis enables us to achieve for every arbitrarily small $ \varepsilon \in (0,\infty) $ and every arbitrarily large $ p \in (0,\infty) $ the essentially sharp strong $ L^p $-convergence rate $ \nicefrac{1}{2} - \varepsilon $ in the case of very natural stochastic optimization examples with quadratically growing loss functions such as in the case of linear regression; see Corollary~\ref{SA_for_linear_regression} in Subsection~\ref{subsection:LinReg} below for details.
%
%
The key ideas in our proofs of Theorem~\ref{SGD_intro} and Theorem~\ref{Lp_theorem}, respectively, are, first, to employ techniques from the theory of Lyapunov-type functions for dynamical systems to develop a general convergence result for SAAs and SGD optimization algorithms based on such functions (see Proposition~\ref{Lyapunov_convergence_explicit} and Corollary~\ref{Lyapunov_convergence_nonexplicit} in Subsection~\ref{subsection:Lyapunov} below for details), then, to apply this general convergence result to concrete Lyapunov-type functions of the form $ \R^d \ni \theta \mapsto V_q( \theta ) = \| \theta - \vartheta \|_{ \R^d }^q \in [0,\infty) $ for $ q \in \{ 2, 4, 6, 8, \dots \} $ (see \eqref{168intro} in the proof of Proposition~\ref{L2_convergence_explicit} in Subsection~\ref{subsection:L2} as well as \eqref{212intro} in the proof of Proposition~\ref{Lp_convergence} in Subsection~\ref{subsection:Lp} below for details), and, thereafter, to perform an induction argument on $ q \in \{ 2, 4, 6, 8, \dots \} \cap [0,p] $ in order to establish for every arbitrarily large $ p \in (0,\infty) $ strong $ L^p $-convergence rates. In previous error analysis results for SAAs and SGD optimization algorithms in the literature induction arguments have been frequently employed along the time variable (cf., e.g., also Lemma~\ref{main_estimate} in Subsection~\ref{subsection:Gronwall} below as well as \eqref{139intro} in the proof of Proposition~\ref{Lyapunov_convergence_explicit} in Subsection~\ref{subsection:Lyapunov} below). A key idea in this work is to perform an induction argument along the powers $ q \in \{ 2, 4, 6, 8, \dots \} $ appearing in the Lyapunov-type functions $ \R^d \ni \theta \mapsto V_q( \theta ) = \| \theta - \vartheta \|_{ \R^d }^q \in [0,\infty) $.

The remainder of this article is organized as follows. In Section~\ref{section:aux_results} we present several auxiliary results which we employ in our strong $ L^p $-error analysis. In Section~\ref{section:main_section} we develop our strong $ L^p $-error analysis for general SAAs. In particular, in Subsection~\ref{subsection:Lp} of Section~\ref{section:main_section} we present and prove Theorem~\ref{Lp_theorem}, which is the main result of this article. In Section~\ref{section:applications} we specialize the abstract findings of Section~\ref{section:main_section} to SGD optimization algorithms. In particular, in Corollary~\ref{SGD} in Subsection~\ref{subsection:SGD} we establish for every arbitrarily large $ p \in (0,\infty) $ strong $ L^p $-convergence rates for SGD optimization algorithms. Theorem~\ref{SGD_intro} above is in immediate consequence of Jensen's inequality and Corollary~\ref{SGD} in Subsection~\ref{subsection:SGD} below. In Subsection~\ref{subsection:LinReg} we also illustrate the statement of Corollary~\ref{SGD} by means of a simple example.


\section{Auxiliary Results}
\label{section:aux_results}

\subsection{Norms on Euclidean spaces}
\label{subsection:norms}
In this subsection we establish in Lemmas~\ref{weak_triangle}--\ref{derivative_of_norm} below some elementary and essentially well-known results for norms in Euclidean spaces. Lemmas~\ref{weak_triangle}, ~\ref{weak_reverse_triangle2}, and ~\ref{derivative_of_norm} are used in our strong error analysis for SGD methods in Proposition~\ref{Lp_convergence} in Subsection~\ref{subsection:Lp} below. Lemma~\ref{weak_reverse_triangle1}, in turn, is employed in the proof of Lemma~\ref{weak_reverse_triangle2}.
\begin{lemma}[Convexity of powers of the norm]
\label{weak_triangle}
Let $d \in \N$, $p \in [1,\infty)$, $v,w \in \R^d$ and let $\left \| \cdot \right \| \! \colon \R^d \to [0,\infty)$ be a norm. Then
\begin{equation}
\begin{split}
\norm{v+w}^p &\leq \left[\sup_{x,y \in (0,\infty)}\frac{(x+y)^p}{(x^p+y^p)}\right] (\norm{v}^p + \norm{w}^p)\\
&=2^{p-1} (\norm{v}^p + \norm{w}^p) \leq 2^{p} (\norm{v}^p + \norm{w}^p).
\end{split}
\end{equation}
\end{lemma}

\begin{proof}[Proof of Lemma \ref{weak_triangle}]
Throughout this proof assume w.l.o.g.\ that $p > 1$ and let $f \colon (0,\infty) \to \R$ be the function which satisfies for all $t \in (0,\infty)$ that
\begin{equation}
f(t) = \frac{(1+t)^p}{1+t^p}.
\end{equation}
Note that 
\begin{equation}
\label{weak_triangle:eq1}
\begin{split}
\sup_{x,y \in (0,\infty)}\left[\frac{(x+y)^p}{x^p+y^p}\right] = \sup_{x,t \in (0,\infty)}\left[\frac{(x+tx)^p}{x^p+(tx)^p}\right] = \sup_{t \in (0,\infty)}\left[\frac{(1+t)^p}{1+t^p}\right] = \sup_{t \in (0,\infty)} f(t).
\end{split}
\end{equation}
Next observe that for all $t \in (0,\infty)$ it holds that
\begin{equation}
\begin{split}
f'(t) &= \frac{p(1+t)^{p-1}(1+t^p) - p(1+t)^{p}t^{p-1}}{(1+t^p)^2} \\
&= \frac{p(1+t)^{p-1}(1+t^p -(1+t)t^{p-1})}{(1+t^p)^2} \\
&= \frac{p(1+t)^{p-1} (1 -t^{p-1})}{(1+t^p)^2}.
\end{split}
\end{equation}
This implies that 
\begin{equation}
\label{weak_triangle:eq2}
\{t \in (0,\infty) \colon f'(t) = 0\} = \{1\}.
\end{equation} 
Moreover, observe that the fact that for all $t\in (0,\infty)$ it holds that
\begin{equation}
f(t) = \frac{(1+t)^p}{1+t^p}=\frac{(1+\nicefrac{1}{t})^p}{1+\nicefrac{1}{t^p}}
\end{equation}
assures that $\lim_{t \searrow 0}f(t) =  \lim_{t \to \infty}f(t) = 1$.
The fact that $f(1) = 2^{p-1} > 1$, \eqref{weak_triangle:eq1}, and (\ref{weak_triangle:eq2}) hence ensure that
\begin{equation}
\sup_{x,y \in (0,\infty)}\left[\frac{(x+y)^p}{x^p+y^p}\right] = \sup_{t \in (0,\infty)} f(t) = f(1) =  2^{p-1}.
\end{equation}
Therefore, we obtain that
\begin{equation}
\begin{split}
\norm{v+w}^p &\leq \big(\norm{v}+\norm{w}\big)^p \\
&\leq \left[\sup_{x,y \in (0,\infty)}\frac{(x+y)^p}{x^p+y^p}\right] \big(\norm{v}^p+\norm{w}^p\big) \\
&= 2^{p-1}\big(\norm{v}^p+\norm{w}^p\big).
\end{split}
\end{equation}
The proof of Lemma \ref{weak_triangle} is thus completed.
\end{proof}

\begin{lemma}
\label{weak_reverse_triangle1}
Let $p \in \N$. Then it holds for all $x,y \in [0,\infty)$ that
\begin{equation}
\label{weak_reverse_triangle1:conclusion1}
\abs{x^p - y^p} \leq 2^p \abs{x-y}\big(\! \min\{x^{p-1}, y^{p-1}\} + \abs{x-y}^{p-1}\big).
\end{equation}
\end{lemma}
\begin{proof}[Proof of Lemma \ref{weak_reverse_triangle1}]
First, observe that for all $x,y \in [0,\infty)$ with $x \geq y$ it holds that
\begin{equation}
\label{weak_reverse_triangle1:eq1a}
\begin{split}
\abs{x^p - y^p} &= x^p - y^p =  (y + (x-y))^p - y^p \\
&= \left[\sum_{k = 0}^p\binom{p}{k}y^{p-k}(x-y)^k\right] - y^p \\
&= \left[\sum_{k = 1}^{p}\binom{p}{k} y^{p-k}(x-y)^k\right] \\
&= (x-y)\left[\sum_{k = 1}^{p}\binom{p}{k} y^{p-k}(x-y)^{k-1}\right].
\end{split}
\end{equation}
This demonstrates that for all $x,y \in [0,\infty)$ with $x\geq y$ it holds that 
\begin{equation}
\label{weak_reverse_triangle1:eq1}
\begin{split}
\abs{x^p - y^p} &\leq (x-y)\left[\sum_{k = 1}^{p}\binom{p}{k} \big[\max\{y, x-y\}\big]^{p-1}\right]\\
&\leq (x-y)\max\!\big\{y^{p-1}, (x-y)^{p-1}\big\}\left[\sum_{k = 0}^{p}\binom{p}{k}\right]  \\
&= 2^p(x-y)\max\!\big\{y^{p-1}, (x-y)^{p-1}\big\} \\
&\leq 2^p (x-y)(y^{p-1}+(x-y)^{p-1})\\
&= 2^p \abs{x-y}(y^{p-1}+ \abs{x-y}^{p-1}).
\end{split}
\end{equation}
Hence, we obtain that for all $x,y \in [0,\infty)$ with $x \leq y$ it holds that
\begin{equation}
\begin{split}
\abs{x^p - y^p} &= \abs{y^p - x^p} \leq 2^p \abs{y-x}(x^{p-1}+ \abs{y-x}^{p-1})
\end{split}
\end{equation}
This and (\ref{weak_reverse_triangle1:eq1}) establish (\ref{weak_reverse_triangle1:conclusion1}). The proof of Lemma \ref{weak_reverse_triangle1} is thus completed.
\end{proof}

\begin{lemma}
\label{weak_reverse_triangle2}
Let $d , p \in \N$, $v,w \in \R^d$ and let $\norm{\!\cdot\!} \colon \R^d \to [0,\infty)$ be a norm. Then 
\begin{equation}
\label{weak_reverse_triangle2:conclusion1}
\begin{split}
\Abs{\norm{v}^p - \norm{w}^p} &\leq 2^p \norm{v-w}\big(\min\{\norm{v}^{p-1}, \norm{w}^{p-1}\}+ \norm{v-w}^{p-1}\big)\\
&\leq 2^p \norm{v-w}\big(\norm{w}^{p-1}+ \norm{v-w}^{p-1}\big).
\end{split}
\end{equation}
\end{lemma}
\begin{proof}[Proof of Lemma \ref{weak_reverse_triangle2}]
Observe that Lemma \ref{weak_reverse_triangle1} ensures that
\begin{equation}
\begin{split}
\Abs{\norm{v}^p - \norm{w}^p} &\leq 2^p \Abs{\norm{v}-\norm{w}}\big(\min\{\norm{v}^{p-1}, \norm{w}^{p-1}\}+ \Abs{\norm{v}-\norm{w}}^{p-1}\big) \\
& \leq 2^p \norm{v-w}\big(\min\{\norm{v}^{p-1}, \norm{w}^{p-1}\}+ \norm{v-w}^{p-1}\big).
\end{split}
\end{equation}
The proof of Lemma \ref{weak_reverse_triangle2} is thus completed.
\end{proof}

\begin{lemma}[Derivative of the norm]
\label{derivative_of_norm}
Let $d \in \N$, $p\in \{2,3,\ldots \}$, $\vartheta \in \R^d$, let 
$\lll \cdot,\cdot \rrr \colon \R^d \times \R^d \to \R$
 be a scalar product, let 
 $\left \| \cdot \right \| \! \colon \R^d \to [0,\infty)$ be the function which satisfies for all $\theta\in \R^d$ that $\norm{\theta} = \sqrt{\lll \theta, \theta \rrr}$, and let $V \colon \R^d \to [0,\infty)$ be the function which satisfies for all $\theta \in \R^d$ that $V(\theta) = \norm{\theta+\vartheta}^p$. Then 
 \begin{enumerate}[(i)]
\item it holds that $V \in C^1(\R^d, [0,\infty))$ and
\item it holds for all $\theta,v \in \R^d$ that \begin{equation}
\label{derivative_of_norm:conclusion}
V'(\theta)(v) = 
p \norm{\theta+\vartheta}^{p-2} \lll\theta+\vartheta, v \rrr.
\end{equation}
 \end{enumerate}
\end{lemma}

\begin{proof}[Proof of Lemma \ref{derivative_of_norm}]
Throughout this proof assume w.l.o.g.\ that $p\geq 3$  and let  $f \colon \R^d \to [0,\infty)$ and $g\colon \R \to [0,\infty)$ be the functions which satisfy for all $\theta \in \R^d$, $x\in \R$ that
\begin{equation}
f(\theta) = \norm{\theta + \vartheta}^2 \qandq g(x) = \abs{x}^{\nicefrac{p}{2}}.
\end{equation}
Note that for all $x \in \R$ it holds that $g \in C^1(\R, [0,\infty))$ and 
\begin{equation}
g'(x) = 
\begin{cases}
\tfrac{p}{2}\, \abs{x}^{\nicefrac{p}{2}-1} & : x\geq 0\\
-\tfrac{p}{2}\,\abs{x}^{\nicefrac{p}{2}-1} & : x<0.
 \end{cases}
\end{equation}
The chain rule hence implies that for all $\theta, v \in \R^d$ it holds that $g \circ f \in C^1(\R^d, [0,\infty))$ and
\begin{equation}
\begin{split}
\big( (g \circ f)'(\theta)\big)(v)
&= \tfrac{p}{2} \, \Abs{\norm{\theta+\vartheta}^2}^{\nicefrac{p}{2}-1} \big( 2\lll\theta+\vartheta,v\rrr\big) \\
&= p \norm{\theta+\vartheta}^{p-2} \lll\theta+\vartheta,v\rrr.
\end{split}
\end{equation}
Combining this with the fact that $V = g \circ f$ completes the proof of Lemma \ref{derivative_of_norm}.
\end{proof}

\subsection{Conditional expectation}
In this subsection we present in Lemma~\ref{Orthogonality} a well-known property associated to conditional expectations (cf., e.g., Klenke \cite[Theorem 8.14]{Klenke}), which we employ in our strong error analyses in Propositions~\ref{L2_convergence_explicit} and \ref{Lp_convergence} below. For completeness we also provide the proof of Lemma~\ref{Orthogonality} in this subsection.
%
\label{subsection:L2_orthogonality}
\begin{lemma}
\label{Orthogonality}
Let $d \in \N$, let 
$\lll \cdot,\cdot \rrr \colon \R^d \times \R^d \to \R$
 be a scalar product, let 
 $\left \| \cdot \right \| \! \colon \R^d \to [0,\infty)$ 
 be the function which satisfies for all $\theta \in \R^d$ that $\norm{\theta} = \sqrt{\lll \theta,\theta \rrr}$,
 let $ ( \Omega , \mathcal{F}, \P ) $ be a probability space, let $\mathcal{G} \subseteq \mathcal{F}$ be a sigma-algebra on $\Omega$, 
let $ X \colon \Omega \to \R^d$ be $\mathcal{F}/\mathcal{B}(\R^d)$-measurable, let $ Y \colon \Omega \to \R^d$ be $\mathcal{G}/\mathcal{B}(\R^d)$-measurable, and assume for all $A \in \mathcal{G}$ that   
\begin{equation}
\label{Orthogonality:assumption1}
\EXP{\norm{X} + \norm{Y} + \norm{X}\norm{Y}} < \infty \qandq\EXP{X \mathbbm{1}_A} = 0.
\end{equation}
Then it holds for all $A \in \mathcal{G}$ that
\begin{equation}
\label{Orthogonality:conclusion1}
\EXP{|\lll X,Y\rrr|}<\infty \qandq \EXP{\lll X,Y \rrr \ind{A}} = 0.
\end{equation} 
\end{lemma}

\begin{proof}[Proof of Lemma \ref{Orthogonality}]
Throughout this proof let $c\in(0,\infty)$ satisfy
\begin{equation}
\label{Orthogonality:constant-norm-equiv}
c= \sup_{\theta=(\theta_1,\dots,\theta_d) \in \R^d\setminus\{0\}}\left(\frac{\big(\sum_{i=1}^d |\theta_i|\big)}{\norm{\theta}}\right),
\end{equation} 
let $X_{i}\colon \Omega \to \R$, $i \in \{1,2,\ldots,d\}$, and $Y_{i} \colon \Omega \to \R$, $i \in \{1,2,\ldots,d\}$, be the functions which satisfy that
\begin{equation}
X = (X_{1},X_{2},\ldots,X_{d})\qquad \text{and} \qquad Y= (Y_{1},Y_{2},\ldots,Y_{d}), 
\end{equation}
let $e_1=(1,0,\dots,0)$, $e_2=(0,1,0,\dots,0)$, $ \dots $, $e_d=(0,\dots,0,1) \in \R^d$,
and let $M = (M_{i,j})_{(i,j) \in \{1,2,\ldots,d\}^2} \in \R^{d\times d}$ be the $(d \times d)$-matrix which satisfies for all $i,j \in \{1,\dots,d\}$ that
\begin{equation}
M_{i,j}=\lll e_i,e_j\rrr.
\end{equation}
Observe that the Cauchy-Schwarz inequality and the hypothesis that $\EXP{\norm{X}\norm{Y}} < \infty$ imply that 
\begin{equation}
\label{conditioning-int}
\EXP{\abs{\lll X,Y \rrr}}\leq \EXP{\norm{X}\norm{Y}} <\infty.
\end{equation}
Next note that \eqref{Orthogonality:constant-norm-equiv} and the hypothesis that $\EXP{\norm{X}\norm{Y}} < \infty$ ensure that for all $i,j \in \{1,2,\ldots,d\}$ it holds that
\begin{equation}
\EXP{\abs{X_{i}Y_{j}}}
 = \EXP{\abs{X_{i}}\abs{Y_{j}}} 
 \leq \EXP{\big(\textstyle\sum_{k=1}^d\abs{X_{k}}\big) \big(\textstyle\sum_{k=1}^d\abs{Y_{k}}\big)} \leq c^2 \EXP{\norm{X}\norm{Y}} <\infty.
\end{equation}
Item (iii) in Theorem 8.14 in Klenke \cite{Klenke}, the hypothesis that the function $Y$ is $\mathcal{G}/\mathcal{B}(\R^d)$-measurable, (\ref{Orthogonality:assumption1}), and \eqref{conditioning-int} hence ensure that for all $A \in \mathcal{G}$ it holds that
\begin{equation}
\begin{split}
\EXP{\lll X,Y \rrr \ind{A}}
&= \EXPPPP{\lllll \textstyle\sum\limits_{i=1}^d X_i e_i, \textstyle\sum\limits_{j=1}^d Y_j e_j \rrrrr  \ind{A}}\\
&= \EXPPPP{\bigg( \textstyle\sum\limits_{i,j=1}^d X_i Y_j \, \lll e_i,e_j \rrr\bigg) \ind{A}}\\
 &= \EXPPPP{\EXPPP{\bigg(\textstyle\sum\limits_{i,j = 1}^d X_{i}Y_{j}M_{i,j}\bigg) \ind{A} \bigg| \mathcal{G}}} \\
 &= \EXPPPP{\EXPPP{ \textstyle\sum\limits_{i,j = 1}^d X_{i}Y_{j}M_{i,j} \Big| \mathcal{G}}\ind{A}} \\
&=\EXPPP{\bigg(\textstyle\sum\limits_{i,j = 1}^d\EXP{  X_{i}Y_{j} \big| \mathcal{G}}M_{i,j}\bigg)\ind{A}} \\
&=\EXPPP{\bigg(\textstyle\sum\limits_{i,j = 1}^d \EXP{  X_{i}\big| \mathcal{G}}Y_{j}M_{i,j}\bigg)\ind{A}} \\
&= \EXPP{  \llll \EXP{  X\big| \mathcal{G}}, Y \rrrr \ind{A}}= 0.
\end{split}
\end{equation}
This and \eqref{conditioning-int} establish  \eqref{Orthogonality:conclusion1}.
The proof of Lemma \ref{Orthogonality} is thus completed.
\end{proof}
\subsection{Factorization lemma for conditional expectations}
In this subsection we recall the statement and the proof of 
the well-known factorization lemma for conditional expectations from the literature (cf., e.g., Da Prato \& Zabczyk~\cite[Proposition~1.12]{DaPratoZabczyk92} and Pusnik \& Jentzen \cite[Subsection~2.1]{JentzenPusnik16}). 
\begin{lemma}\label{condexp_and_independence-nonneg:le1}
Let $ ( \Omega , \mathcal{F}, \P ) $ be a probability space, let $\mathcal{G} \subseteq \mathcal{F}$ be a sigma-algebra on $\Omega$, let $(\mathbb{X},\mathcal{X})
$ and $(\mathbb{Y},\mathcal{Y})$ be measurable spaces,
let $ X  \colon \Omega \to \mathbb{X}$ be $\mathcal{F}/\mathcal{X}$-measurable, assume that $X$ is independent of $\mathcal{G}$, let $ Y \colon \Omega \to \mathbb{Y}$ be $\mathcal{G}/\mathcal{Y}$-measurable, let
$B \in (\mathcal{X}\otimes \mathcal{Y})$, and let $\phi \colon \mathbb{Y} \to [0,\infty)$ be the function which satisfies for all $y\in \mathbb{Y}$ that $\phi(y)= \EXP{\ind{B}(X,y)}$.
Then
\begin{enumerate}[(i)]
	\item \label{condexp_and_independence-nonneg:le1:item1}
	it holds that the function $\phi$ 
	is $\mathcal{Y}/ \mathcal{B}([0,\infty))$-measurable and
	\item \label{condexp_and_independence-nonneg:le1:item2}
	it holds for all $A \in \mathcal{G}$ that
	\begin{equation}
	\EXP{\ind{B}(X,Y) \ind{A}} = \EXP{\phi(Y) \ind{A}}.
	\end{equation}
\end{enumerate}	
\end{lemma}
\begin{proof}[Proof of Lemma \ref{condexp_and_independence-nonneg:le1}]
Throughout this proof
for every set $S$ and every subset $\mathcal{S} \subseteq \mathcal{P}(S)$ of the power set $\mathcal{P}(S)$ of $S$ let $\delta_S(\mathcal{S})$ be the set given by
%
\begin{equation}
\begin{split}
\delta_S(\mathcal{S})
=
\bigcap\nolimits_{ \mathcal{B}
	\in
	\big\{\substack{ 
		\mathcal{C}	
		\text{ is a Dynkin system}
		\\
		\text{on }
		S
		\text{ with }
		\mathcal{C} \supseteq \mathcal{S} 
	}
	\big\}
}
\mathcal{B}
,
\end{split}
\end{equation}
for every set $S$ and every subset $\mathcal{S} \subseteq \mathcal{P}(S)$ of the power set $\mathcal{P}(S)$ of $S$ let $\sigma_S(\mathcal{S})$ be the set given by 
%
\begin{equation}
\begin{split}
\sigma_{S}(\mathcal{S})
=
\bigcap\nolimits_{ \mathcal{B}
	\in
	\big\{\substack{ 
		\mathcal{C}	
		\text{ is a sigma-algebra}
		\\
		\text{on }
		S
		\text{ with }
		\mathcal{C} \supseteq \mathcal{S} 
	}
	\big\}
}
\mathcal{B}
,
\end{split}
\end{equation}
let $\mathcal{E}\subseteq (\mathcal{X}\otimes \mathcal{Y})$ be the set given by
\begin{equation}
\mathcal{E} = \big\{S \in (\mathcal{X}\otimes \mathcal{Y}) \colon  (\exists\, E_1 \in \mathcal{X}, E_2 \in \mathcal{Y} \colon S= E_1 \times E_2) \big\},
\end{equation}
and let $\mathcal{D}\subseteq (\mathcal{X}\otimes \mathcal{Y})$ be the set given by
\begin{equation}
\mathcal{D} = \left\{
\begin{split}
&D \in (\mathcal{X}\otimes \mathcal{Y})\colon\\
&\left[
\begin{split}
&\big(\mathbb{Y} \ni y \mapsto \Exp{\ind{D}(X,y)} \in[0,\infty) \big) \mbox{ is } \mathcal{Y}/\mathcal{B}([0,\infty))\mbox{-measurable}\\
&\mbox{and }\big( \forall\, A \in \mathcal{G} \colon \EXP{\ind{D}(X,Y) \ind{A}} = \EXP{(\Exp{\ind{D}(X,y)})|_{y = Y} \ind{A}} \big)
\end{split}
\right]
\end{split}
\right\}
\end{equation}
Note that Fubini's theorem (cf., e.g., Klenke \cite[(14.6) in Theorem 14.16]{Klenke}) and the assumption that the function $ X  \colon \Omega \to \mathbb{X}$ is $\mathcal{F}/\mathcal{X}$-measurable
demonstrate that for all $D \in (\mathcal{X}\otimes \mathcal{Y})$ it holds that the function
\begin{equation}
\mathbb{Y} \ni y \mapsto \EXP{\ind{D}(X,y)} = \int_\Omega \ind{D}(X(\omega),y) \,\P(\mathrm{d}\omega) \in [0,\infty)
\end{equation}
is $\mathcal{Y}/\mathcal{B}([0,\infty))$-measurable.
Hence, we obtain  that
\begin{equation}
\begin{split}
&\mathcal{D}=\\
&\left\{D \in(\mathcal{X}\otimes \mathcal{Y}) \colon  \big( \forall\, A \in \mathcal{G} \colon \EXP{\ind{D}(X,Y) \ind{A}} = \EXP{(\Exp{\ind{D}(X,y)})|_{y = Y} \ind{A}} \big) \right\}.
\end{split}
\end{equation}
Next observe that the hypothesis that the function $X$ is independent of $\mathcal{G}$ and the hypothesis that the function $Y$ is $\mathcal{G}/\mathcal{Y}$-measurable ensure that for all $E_1 \in \mathcal{X}, E_2 \in \mathcal{Y}, A \in \mathcal{G}$ it holds that
\begin{equation}
\begin{split}
&\EXPP{\big(\EXP{\ind{E_1 \times E_2}(X,y)} \big)\big|_{y = Y} \ind{A}} =  \EXPP{\big(\EXP{\ind{E_1}(X)\ind{E_2}(y)} \big)\big|_{y = Y} \ind{A}} \\ 
&=\EXP{\P(X \in E_1)\ind{E_2}(Y)  \ind{A}} =  \P(X \in E_1) \, \P(\{Y \in E_2 \}\cap A) \\
&= \P( \{X \in E_1\}\cap\{Y \in E_2 \}\cap A) = \EXP{\ind{E_1 \times E_2}(X,Y)\ind{A}}.
\end{split}
\end{equation}
Therefore, we obtain that $\mathcal{E} \subseteq \mathcal{D}$. Next observe that for all $D \in \mathcal{D}$, $A \in \mathcal{G}$ it holds that
\begin{equation}
\label{condexp_and_independence-nonneg:le1:eq1}
\begin{split}
&\EXP{\ind{((\mathbb{X}\times \mathbb{Y})\setminus D)}(X,Y)\ind{A}}
 = \EXP{(1-\ind{D}(X,Y))\ind{A}}  \\
& = \EXP{\ind{A}}-\EXP{\ind{D}(X,Y) \ind{A}}
= \EXP{\ind{A}} -\EXPP{\big(\EXP{\ind{D}(X,y)} \big)\big|_{y = Y}\ind{A}}\\
& = \EXPP{\big(1-\big(\EXP{\ind{D}(X,y)} \big)\big|_{y = Y}\big)\ind{A}}
=\EXPP{\big(\EXP{1-\ind{D}(X,y)} \big)\big|_{y = Y}\ind{A}} \\
&= \EXPP{\big(\EXP{\ind{((\mathbb{X}\times \mathbb{Y})\setminus D)}(X,y)} \big)\big|_{y = Y} \ind{A}}.
\end{split}
\end{equation}
Moreover, note that the monotone convergence theorem implies that for all $A \in \mathcal{G}$, $(D_k)_{k \in \N}\subseteq \mathcal{D}$ with $\forall\, i \in \N, j \in \N\backslash\{i\} \colon D_i \cap D_j =  \emptyset$ it holds that
\begin{equation}
\begin{split}
\EXPP{\ind{(\cup_{k = 1}^\infty D_k)} (X,Y)\, \ind{A}} 
&= \EXPP{\lim\limits_{n \to \infty}\left[\ind{(\cup_{k = 1}^n D_k)} (X,Y)\, \ind{A}\right]}
\\
&=\EXPP{\lim\limits_{n \to \infty}\left[\textstyle\sum_{k = 1}^n\ind{D_k}(X,Y)\, \ind{A}\right]}
\\
&=\lim\limits_{n \to \infty}\EXPP{\textstyle\sum_{k = 1}^n\ind{D_k}(X,Y)\, \ind{A}}
\\
&= \lim_{n \to \infty} \left[\sum_{k = 1}^n \EXP{\ind{ D_k} (X,Y)\, \ind{A}} \right] \\
&= \lim_{n \to \infty} \left[\sum_{k = 1}^n\EXPP{\big(\EXP{\ind{D_k}(X,y)} \big)\big|_{y = Y} \,\ind{A}}\right].
\end{split}
\end{equation}
Again the monotone convergence theorem hence implies that for all $A \in \mathcal{G}$, $(D_k)_{k \in \N}\subseteq \mathcal{D}$ with $\forall\, i \in \N, j \in \N\backslash\{i\} \colon D_i \cap D_j =  \emptyset$ it holds that
\begin{equation}
\begin{split}
\label{condexp_and_independence-nonneg:le1:eq2}
\EXPP{\ind{(\cup_{k = 1}^\infty D_k)} (X,Y)\, \ind{A}}
&=\EXPPP{\lim_{n \to \infty}\Big(\textstyle\sum_{k = 1}^n\EXP{\ind{D_k}(X,y)}\Big)\Big|_{y = Y} \,\ind{A}}\\
 &=  \EXPPP{\Big(\lim_{n \to \infty}\textstyle\sum_{k = 1}^n\EXP{\ind{D_k}(X,y)}\Big)\Big|_{y = Y} \,\ind{A}} \\
&=\EXPPP{ \Big(\lim_{n \to \infty}\EXP{\textstyle\sum_{k = 1}^n\ind{D_k}(X,y)} \Big)\Big|_{y = Y} \,\ind{A}} \\
&=\EXPP{\big(\EXP{\ind{(\cup_{k = 1}^\infty D_k)}(X,y)} \big)\big|_{y = Y} \,\ind{A}}.
\end{split}
\end{equation}
This, (\ref{condexp_and_independence-nonneg:le1:eq1}), and the fact that $(\mathbb{X}\times \mathbb{Y}) \in \mathcal{D}$ show that $\mathcal{D}$ is a Dynkin-system.
The fact that $\mathcal{E}$ is 
$\cap$-stable, the fact that $\mathcal{E}\subseteq\mathcal{D}$, and Dynkin's $\pi$-$\lambda$-Theorem (see, e.g., \cite[Theorem~2.5]{JentzenPusnik16})
therefore demonstrate that
\begin{equation}
(\mathcal{X}\otimes \mathcal{Y})  = \sigma_{\mathbb{X}\times \mathbb{Y}}(\mathcal{E}) = \delta_{\mathbb{X}\times \mathbb{Y}}(\mathcal{E}) \subseteq \mathcal{D} \subseteq (\mathcal{X}\otimes \mathcal{Y}).
\end{equation}
Hence, we obtain that $\mathcal{D} = \mathcal{X}\otimes \mathcal{Y}$. 
The assumption that $B \in (\mathcal{X} \otimes \mathcal{Y})$ hence assures that $B \in \mathcal{D}$.
 This completes the proof of Lemma~\ref{condexp_and_independence-nonneg:le1}.
\end{proof}
\begin{lemma}
	\label{condexp_and_independence-nonneg:le2}
	Let $N \in \N$, $c_1,\dots, c_N \in [0,\infty)$, let $ ( \Omega , \mathcal{F}, \P ) $ be a probability space, let $\mathcal{G} \subseteq \mathcal{F}$ be a sigma-algebra on $\Omega$, let $(\mathbb{X},\mathcal{X})
	$ and $(\mathbb{Y},\mathcal{Y})$ be measurable spaces,
	let $D_1,\dots,D_N \in (\mathcal{X}\otimes \mathcal{Y})$,
	let $ X  \colon \Omega \to \mathbb{X}$ be $\mathcal{F}/\mathcal{X}$-measurable, assume that $X$ is independent of $\mathcal{G}$, let $ Y \colon \Omega \to \mathbb{Y}$ be $\mathcal{G}/\mathcal{Y}$-measurable, let $\Phi \colon \mathbb{X}\times \mathbb{Y} \to [0,\infty)$ be $(\mathcal{X}\otimes \mathcal{Y}) / \mathcal{B}([0,\infty))$-measurable, assume for all $x\in \mathbb{X}$, $y\in \mathbb{Y}$ that
	\begin{equation}
	\Phi(x,y)= \sum_{k=1}^N c_k \ind{D_k}(x,y),
	\end{equation} 
 and let $\phi \colon \mathbb{Y} \to [0,\infty)$ be the function which satisfies for all $y\in \mathbb{Y}$ that $\phi(y)= \EXP{\Phi(X,y)}$.
	Then
	\begin{enumerate}[(i)]
		\item \label{condexp_and_independence-nonneg:le2:item1}
		it holds that the function $\phi$ 
		is $\mathcal{Y}/ \mathcal{B}([0,\infty))$-measurable and
		\item \label{condexp_and_independence-nonneg:le2:item2}
		it holds for all $A \in \mathcal{G}$ that
		\begin{equation}
		\EXP{\Phi(X,Y) \ind{A}} = \EXP{\phi(Y) \ind{A}}.
		\end{equation}
	\end{enumerate}
\end{lemma}
\begin{proof}[Proof of Lemma \ref{condexp_and_independence-nonneg:le2}]
First, note that for all $y \in \mathbb{Y}$ it holds that
\begin{equation}
\phi(y)= \EXP{\Phi(X,y)} = \sum_{k=1}^N c_k \Exp{\ind{D_k}(X,y)}.
\end{equation}
Item~(\ref{condexp_and_independence-nonneg:le1:item1}) in Lemma~\ref{condexp_and_independence-nonneg:le1} therefore ensures that the function $\phi$ 
is $\mathcal{Y}/ \mathcal{B}([0,\infty))$-measurable. This establishes item~(\ref{condexp_and_independence-nonneg:le2:item1}). In addition, observe that 
 Lemma~\ref{condexp_and_independence-nonneg:le1} implies that
\begin{equation}
\begin{split}
\EXP{\Phi(X,Y) \ind{A}} 
&= 
\EXP{\textstyle\sum_{k=1}^N c_k \ind{D_k}(X,Y) \ind{A}}\\
&= \sum_{k=1}^N c_k \EXP{ \ind{D_k}(X,Y) \ind{A}}\\
&= \sum_{k=1}^N c_k \EXP{(\Exp{\ind{D_k}(X,y)})|_{y=Y} \ind{A}}\\
&= \EXP{\big(\mathbb{E}\big[\textstyle\sum_{k=1}^N c_k \ind{D_k}(X,y)\big]\big) \big|_{y=Y} \ind{A}}\\
&= \EXP{(\Exp{\Phi(X,y)})|_{y=Y} \ind{A}}\\
&=\EXP{\phi(Y) \ind{A}}.
\end{split}
\end{equation}
This establishes item~(\ref{condexp_and_independence-nonneg:le2:item2}). The proof of Lemma~\ref{condexp_and_independence-nonneg:le2} is thus completed.
\end{proof}
\begin{lemma}
	\label{condexp_and_independence-nonneg}
	Let $ ( \Omega , \mathcal{F}, \P ) $ be a probability space, let $\mathcal{G} \subseteq \mathcal{F}$ be a sigma-algebra on $\Omega$, let $(\mathbb{X},\mathcal{X})
	$ and $(\mathbb{Y},\mathcal{Y})$ be measurable spaces,
	let $ X  \colon \Omega \to \mathbb{X}$ be $\mathcal{F}/\mathcal{X}$-measurable, assume that $X$ is independent of $\mathcal{G}$, let $ Y \colon \Omega \to \mathbb{Y}$ be $\mathcal{G}/\mathcal{Y}$-measurable, let $\Phi \colon \mathbb{X}\times \mathbb{Y} \to [0,\infty]$ be $(\mathcal{X}\otimes \mathcal{Y}) / \mathcal{B}([0,\infty])$-measurable, and let $\phi \colon \mathbb{Y} \to [0,\infty]$ be the function which satisfies for all $y\in \mathbb{Y}$ that $\phi(y)= \EXP{\Phi(X,y)}$.
	Then
	\begin{enumerate}[(i)]
		\item \label{condexp_and_independence-nonneg:item1}
		it holds that the function $\phi$ 
		is $\mathcal{Y}/ \mathcal{B}([0,\infty])$-measurable and
		\item \label{condexp_and_independence-nonneg:item2}
		it holds for all $A \in \mathcal{G}$ that
		\begin{equation}
		\EXP{\Phi(X,Y) \ind{A}} = \EXP{\phi(Y) \ind{A}}.
		\end{equation}
	\end{enumerate}
\end{lemma}
\begin{proof}[Proof of Lemma \ref{condexp_and_independence-nonneg}]
First, note that Fubini's theorem (cf., e.g.,   Klenke \cite[(14.6) in Theorem 14.16]{Klenke}), the assumption that the function $ X  \colon \Omega \to \mathbb{X}$ is $\mathcal{F}/\mathcal{X}$-measurable, and the assumption that the function $\Phi \colon \mathbb{X}\times \mathbb{Y} \to [0,\infty]$ is $(\mathcal{X}\otimes \mathcal{Y}) / \mathcal{B}([0,\infty])$-measurable demonstrate that the function
\begin{equation}
\mathbb{Y} \ni y \mapsto \phi(y)= \EXP{\Phi(X,y)} = \int_\Omega \Phi(X(\omega),y) \,\P(\mathrm{d}\omega) \in [0,\infty]
\end{equation}
is $\mathcal{Y}/\mathcal{B}([0,\infty])$-measurable. 
This establishes item (\ref{condexp_and_independence-nonneg:item1}). It thus remains to prove item (\ref{condexp_and_independence-nonneg:item2}).
For this let $\Phi_n \colon \mathbb{X}\times \mathbb{Y} \to [0,\infty)$, $n \in \N$, be the functions which satisfy for all $n \in \N$, $x\in \mathbb{X}$, $y\in\mathbb{Y}$ that
\begin{equation}
\begin{split}
&\Phi_n(x,y)= \\
& 2^n \,\ind{ \{(v,w)\in \mathbb{X}\times\mathbb{Y}: \Phi(v,w)\geq 2^n\}}(x,y) + \sum_{k = 0}^{2^{2n}-1}\bigg[\frac{k}{2^n}\,\ind{ \{(v,w)\in \mathbb{X}\times\mathbb{Y}: \, k 2^{-n}\leq \Phi(v,w)<(k+1) 2^{-n}\!\}}(x,y)\bigg].
\end{split}
\end{equation}
Observe that the hypothesis that the function $\Phi \colon \mathbb{X} \times \mathbb{Y} \to [0,\infty]$ is $(\mathcal{X} \otimes \mathcal{Y})/\mathcal{B}([0,\infty])$-measurable assures that for all $n \in \N$, $k \in \{0,1, \ldots, 2^{2n}-1 \}$ it holds that 
\begin{equation}
\big\{ (v,w) \in \mathbb{X}\times \mathbb{Y} \colon \Phi(v,w)\geq k 2^{n} \!\big\} \in (\mathcal{X} \otimes \mathcal{Y})
\end{equation}
and
\begin{equation}
\big\{ (v,w) \in \mathbb{X}\times \mathbb{Y} \colon k2^{-n}\leq \Phi(v,w)< (k+1)2^{-n}\!\big\} \in (\mathcal{X} \otimes \mathcal{Y}).
\end{equation}
This and Lemma~\ref{condexp_and_independence-nonneg:le2}
 ensure  that for all $n \in \N$, $A \in \mathcal{G}$ it holds that
\begin{equation}
\EXP{\Phi_n(X,Y) \ind{A}}= \EXPP{\big(\EXP{\Phi_n(X,y)} \big)\big|_{y = Y} \ind{A}}.
\end{equation} 
The fact that  $\forall \, (x,y) \in \mathbb{X}\times \mathbb{Y}, n \in \N\colon \Phi_n(x,y)\leq \Phi_{n+1}(x,y)$, the fact that $\forall \, (x,y) \in \mathbb{X}\times \mathbb{Y}\colon\lim_{n\to \infty} \Phi_n(x,y)=\Phi(x,y)$, and the monotone convergence theorem therefore imply that for all $A \in \mathcal{G}$ it holds that
\begin{equation}
\begin{split}
\EXP{\Phi(X,Y) \ind{A}} &= \lim_{n \to \infty} \EXP{\Phi_n (X,Y) \ind{A}}
=  \lim_{n \to \infty} \EXPP{\big(\EXP{\Phi_n(X,y)} \big)\big|_{y = Y} \ind{A}} \\
&=\EXPP{\big(\lim\nolimits_{n \to \infty}\EXP{\Phi_n(X,y)} \big)\big|_{y = Y} \ind{A}}\\
&=  \EXPP{\big(\EXP{\Phi(X,y)} \big)\big|_{y = Y} \ind{A}}
=\EXP{\phi(Y) \ind{A}}.
\end{split}
\end{equation}
This establishes item (\ref{condexp_and_independence-nonneg:item2}). 
The proof of Lemma \ref{condexp_and_independence-nonneg} is thus completed.
\end{proof}
\begin{cor}
\label{condexp_and_independence-integrable}
Let $ ( \Omega , \mathcal{F}, \P ) $ be a probability space, let $\mathcal{G} \subseteq \mathcal{F}$ be a sigma-algebra on $\Omega$, let $(\mathbb{X},\mathcal{X})$ and $(\mathbb{Y},\mathcal{Y})$ be measurable spaces,
let $ X  \colon \Omega \to \mathbb{X}$ be $\mathcal{F}/\mathcal{X}$-measurable, assume that $X$ is independent of $\mathcal{G}$, let $ Y \colon \Omega \to \mathbb{Y}$ be $\mathcal{G}/\mathcal{Y}$-measurable, let $\Phi \colon \mathbb{X}\times \mathbb{Y} \to \R$ be  $(\mathcal{X}\otimes \mathcal{Y}) / \mathcal{B}(\R)$-measurable, assume that $\EXP{\abs{\Phi(X,Y)}} < \infty$, let $c \in \R$, let $\phi \colon \mathbb{Y} \to \R$ be a function, assume for all $y \in \mathbb{Y}$ with $\EXP{\abs{\Phi(X,y)}} < \infty$ that $\phi(y)= \EXP{\Phi(X,y)}$, and assume for all $y \in \mathbb{Y}$ with $\EXP{\abs{\Phi(X,y)}} = \infty$ that $\phi(y)= c$.
Then
\begin{enumerate}[(i)]
\item \label{condexp_and_independence-integrable:item0}
 it holds that $\big\{y \in \mathbb{Y}\colon \EXP{\abs{\Phi(X,y)}}<\infty\big\} \in \mathcal{Y}$,
 \item \label{condexp_and_independence-integrable:item0.5}
 it holds that $\P\big(Y \in \big\{y \in \mathbb{Y} \colon \EXP{|\Phi(X,y)|}<\infty\big\}\big)=1$,
\item \label{condexp_and_independence-integrable:item1}
it holds that the function  $\phi$
is $\mathcal{Y}/ \mathcal{B}(\R)$-measurable,
\item \label{condexp_and_independence-integrable:item2}
it holds that $\EXP{\abs{\phi(Y)}} < \infty$, and
\item \label{condexp_and_independence-integrable:item3}
it holds for all $A \in \mathcal{G}$ that 
\begin{equation}
\EXP{\Phi(X,Y) \ind{A}} = \EXP{\phi(Y)\ind{A}}.
\end{equation}
\end{enumerate}
\end{cor}

\begin{proof}[Proof of Corollary \ref{condexp_and_independence-integrable}]
Throughout this proof let $\Phi_k \colon \mathbb{X} \times \mathbb{Y} \to [0,\infty)$, $k\in \{1,2\}$, be the functions which satisfy for all $k\in \{1,2\}$, $x \in \mathbb{X}$, $y \in \mathbb{Y}$ that
\begin{equation}
\label{def-Phi-k}
 \Phi_k(x,y) = \max\!\big\{(-1)^{k+1}\Phi(x,y),0\big\}, 
\end{equation}
let $B\subseteq \mathbb{Y}$ be the set given by
\begin{equation}
\label{2.9-setB}
B=\big\{y \in \mathbb{Y}\colon \EXP{\abs{\Phi(X,y)}}<\infty\big\},
\end{equation}
let $\mu: \mathcal{Y} \to [0,1]$ be the measure which satisfies for all $E \in \mathcal{Y}$ that
\begin{equation}
\mu(E)=\P(Y^{-1}(E)) =\P(Y \in E),
\end{equation}
let $\Psi_k\colon\mathbb X\times\mathbb Y\to[0,\infty)$, $k\in \{1,2\}$,
be the functions which satisfy for all $k\in \{1,2\}$, $x\in\mathbb X$, $y\in\mathbb Y$ that
\begin{equation}
\label{2.9a}
\Psi_k(x,y)=
\begin{cases}
\Phi_k(x,y) & : y\in B\\
0 & : y\in\mathbb Y\setminus B,\\
\end{cases}
\end{equation}
and let $\psi_k\colon\mathbb Y\to[0,\infty)$, $k\in \{1,2\}$, 
be the functions which satisfy for all $k\in \{1,2\}$, $y\in \mathbb Y$ that
\begin{equation}
\label{2.9c}
\psi_k(y)=\E[\Psi_k(X,y)].
\end{equation}
Observe that the hypothesis that the function $\Phi \colon \mathbb{X}\times \mathbb{Y} \to \R$ is $(\mathcal{X}\otimes \mathcal{Y}) / \mathcal{B}(\R)$-measurable and the hypothesis that the function $X \colon \Omega \to \mathbb{X}$ is $\mathcal{F}/\mathcal{X}$-measurable assure that the function
\begin{equation}
\Omega\times\mathbb{Y}\ni(\omega,y) \mapsto \abs{\Phi(X(\omega),y)} \in [0,\infty)
\end{equation}
is $(\mathcal{F}\otimes \mathcal{Y}) / \mathcal{B}([0,\infty))$-measurable. Fubini's theorem (cf., e.g.,   Klenke \cite[(14.6) in Theorem 14.16]{Klenke})  hence proves that
\begin{equation}
\label{B-meas}
B=\big\{y \in \mathbb{Y}\colon \EXP{\abs{\Phi(X,y)}}<\infty\big\} \in \mathcal{Y}.
\end{equation}
This establishes item (\ref{condexp_and_independence-integrable:item0}). In addition, observe that for all $y\in \mathbb{Y}$ it holds that
\begin{equation}
\label{condexp_and_independence-integrable:eq1}
\phi(y)=
\begin{cases}
\EXP{\Phi(X,y)} & \colon y\in B\\
c & \colon y \in \mathbb{Y}\!\setminus\!B.
\end{cases}
\end{equation}
Next observe that the hypothesis that the function
$\Phi:\mathbb X\times\mathbb Y\to\R$ is $(\mathcal X\otimes\mathcal Y)/\mathcal B(\R)$-measurable
and 
the fact that $B\in\mathcal Y$ 
ensure that the functions 
$\Psi_k\colon \mathbb X\times\mathbb Y\to[0,\infty)$, $k\in \{1,2\}$, 
are $(\mathcal X\otimes\mathcal Y)/\mathcal B([0,\infty))$-measurable.
Moreover, note that \eqref{2.9-setB} implies that for all $k\in \{1,2\}$, $y\in B$ it holds that
\begin{equation}
\label{2.9d}
\E[\Psi_k(X,y)]=\E[\abs{\Psi_k(X,y)}]=\E[\Phi_k(X,y)]\leq\E[\abs{\Phi(X,y)}]<\infty.
\end{equation}
Hence,  we obtain that for all $y\in B$ it holds that
\begin{equation}
\label{cor29:new1}
\begin{split}
\psi_1(y)-\psi_2(y)+c\ind{\mathbb Y\setminus B}(y)
&=\psi_1(y) -\psi_2(y)\\
&=\E[\Psi_1(X,y)]-\E[\Psi_2(X,y)]\\
&=\E[\Psi_1(X,y)-\Psi_2(X,y)]\\
&=\E[\Phi_1(X,y)-\Phi_2(X,y)]\\
&=\E[\Phi(X,y)] =\phi(y).
\end{split}
\end{equation}
Furthermore, observe that \eqref{2.9a}, \eqref{2.9c}, and \eqref{condexp_and_independence-integrable:eq1} ensure that for all $y\in\mathbb Y\setminus B$ it holds that
\begin{equation}
\psi_1(y)-\psi_2(y)+c\ind{\mathbb Y\setminus B}(y)=c\ind{\mathbb Y\setminus B}(y) =c=\phi(y).
\label{cor29:new2}
\end{equation}
Moreover, note that
Fubini's theorem (cf., e.g., Klenke \cite[(14.6) in Theorem 14.16]{Klenke}), the fact that the function $X\colon\Omega\to \mathbb X$
is $\mathcal F/\mathcal X$-measurable, and the fact that the functions
$\Psi_k\colon \mathbb X\times\mathbb Y\to[0,\infty)$, $k\in\{1,2\}$,
are $(\mathcal X\otimes\mathcal Y)/\mathcal B([0,\infty))$-measurable
demonstrate that the functions
$\psi_k\colon\mathbb Y\to[0,\infty)$, $k\in\{1,2\}$, 
are $\mathcal Y/\mathcal B([0,\infty))$-measurable.
Combining this 
and the fact that $B\in\mathcal Y$ with
 \eqref{cor29:new1} and 
\eqref{cor29:new2} demonstrates that
the function $\phi$ is $\mathcal Y/\mathcal B(\R)$-measurable.
This establishes item (\ref{condexp_and_independence-integrable:item1}).
 Next observe that 
  Lemma~\ref{condexp_and_independence-nonneg}, (\ref{condexp_and_independence-integrable:eq1}), and the hypothesis that $\Exp{|\Phi(X,Y)|}<\infty$ ensure that
\begin{equation}
\begin{split}
\EXP{\abs{\phi(Y)}}
\leq \abs{c} +\EXPP{\big(\EXP{\abs{\Phi(X,y)}} \big)\big|_{y = Y}} =
\abs{c} + \EXP{\abs{\Phi(X,Y)}} < \infty.
\end{split}
\end{equation}
This establishes item (\ref{condexp_and_independence-integrable:item2}).
%
Moreover, note  that the hypothesis that $\E[\abs{\Phi(X,Y)}]<\infty$ and 
Lemma~\ref{condexp_and_independence-nonneg} assure that
\begin{equation}
\int_{\mathbb Y} \E[\abs{\Phi(X,y)}]\,\mu(\mathrm{d}y)
=\E\!\left[\big(\E[\abs{\Phi(X,y)}]\big)\big|_{y=Y}\right]
=\E[\abs{\Phi(X,Y)}]<\infty.
\end{equation}
Combining this with \eqref{2.9-setB} shows that
\begin{equation}
\mu(B)=\mu(\{y\in\mathbb Y:\E[\abs{\Phi(X,y)}]<\infty\})=1.
\label{cor29:new3}
\end{equation}
Hence, we obtain that
\begin{equation}
\label{69b}
\P(Y \in B) 
=1.
\end{equation}
This establishes item~\eqref{condexp_and_independence-integrable:item0.5}. It thus remains to prove item~\eqref{condexp_and_independence-integrable:item3}. For this observe that \eqref{def-Phi-k}, \eqref{2.9a}, \eqref{69b}, and the fact that 
$\Exp{\Psi_1(X,Y) + \Psi_2(X,Y)}\leq \Exp{\Phi_1(X,Y) + \Phi_2(X,Y)}$ $= \Exp{\abs{\Phi(X,Y)}}<\infty$ ensure that for all $A \in \mathcal{G}$ it holds that
\begin{equation}
\begin{split}
\EXP{\Phi(X,Y) \ind{A}} 
&= \EXP{(\Phi_1(X,Y)-\Phi_2(X,Y)) \ind{A}}\\
&= \EXP{\Phi_1(X,Y) \ind{A}} -\EXP{\Phi_2(X,Y) \ind{A}}\\
&= \EXP{\Phi_1(X,Y) \ind{B}(Y)\ind{A}} -\EXP{\Phi_2(X,Y) \ind{B}(Y) \ind{A}}\\
&= \EXP{\Psi_1(X,Y) \ind{A}} -\EXP{\Psi_2(X,Y) \ind{A}}.\\
\end{split}
\end{equation}
Combining the fact that $\Exp{\Psi_1(X,Y) + \Psi_2(X,Y)}
\leq  \E[\Phi_1(X,Y) + \Phi_2(X,Y)]
= \Exp{\abs{\Phi(X,Y)} }<\infty$ and Lemma~\ref{condexp_and_independence-nonneg} with \eqref{2.9d}, \eqref{cor29:new2}, and \eqref{69b} demonstrate that for all $A\in \mathcal{G}$ it holds that
\begin{equation}
\begin{split}
\EXP{\Phi(X,Y) \ind{A}} 
&= \EXPP{\big(\EXP{\Psi_1(X,y)} \big)\big|_{y = Y}\ind{A}} - \EXPP{\big(\EXP{\Psi_2(X,y)} \big)\big|_{y = Y}\ind{A}}\\
&= \EXP{\psi_1(Y) \ind{A}} - \EXP{\psi_2(Y) \ind{A}} \\
&= \EXP{\psi_1(Y) \ind{A}} - \EXP{\psi_2(Y) \ind{A}} 
+ c\,\P(\{Y \in \mathbb{Y}\setminus B \}\cap A)\\
&= \EXP{(\psi_1(Y)-\psi_2(Y)) \ind{A}} + c \, \EXP{\ind{\mathbb{Y}\setminus B}(Y) \ind{A}}\\
&= \EXP{(\psi_1(Y)-\psi_2(Y)+c \ind{\mathbb{Y}\setminus B}(Y)) \ind{A}} \\
&=\EXP{\phi(Y) \ind{A}}.
\end{split}
\end{equation}
This establishes item (\ref{condexp_and_independence-integrable:item3}). The proof of Corollary \ref{condexp_and_independence-integrable} is thus completed.
\end{proof}

\subsection{On convergence properties of a specific class of sequences}
\label{subsection:decay}
In this subsection we present in Lemma~\ref{decay_fraction} an elementary  auxiliary result on the convergence of a specific class of sequences. Lemma~\ref{decay_fraction} is used in the proof of Lemma~\ref{example_learningrate} in Subsection~\ref{section:applications1} below.
\begin{lemma}
\label{decay_fraction}
Let $\beta, \delta \in (0,\infty)$ with $\beta < \delta + 1$. Then 
\begin{equation}
\limsup_{n \to \infty} \left[\frac{\big|n^{-\delta}- (n-1)^{-\delta}\big|}{n^{-\beta}}\right] = 0.
\end{equation}
\end{lemma}
\begin{proof}[Proof of Lemma \ref{decay_fraction}]
First, note that the fundamental theorem of calculus ensures that for all $n \in \{2,3,\ldots \}$ it holds that
\begin{equation}
\label{decay_fraction:eq1}
\begin{split}
0 &\geq \frac{n^{-\delta}- (n-1)^{-\delta}}{n^{-\beta}} = n^\beta\left(\frac{1}{n^\delta}-\frac{1}{(n-1)^\delta}\right)
=n^\beta\big(\left[ x^{-\delta}\right]_{x=n-1}^{x=n}\big)\\
 &= n^\beta (-\delta)\left[\int_{n-1}^n\frac{1}{x^{\delta+1}}\, \mathrm{d}x\right] 
\geq -\frac{\delta n^\beta}{(n-1)^{\delta+1}}.
\end{split}
\end{equation}
The assumption that $\beta < \delta + 1$ therefore implies that
\begin{equation}
\begin{split}
0 &\leq \limsup_{n \to \infty} \left[\frac{\big|n^{-\delta}- (n-1)^{-\delta}\big|}{n^{-\beta}} \right]
\leq \limsup_{n \to \infty} \left[\frac{\delta n^\beta}{(n-1)^{\delta+1}}\right]\\
&= \limsup_{n \to \infty} \left[\frac{\delta (n+1)^\beta}{n^{\delta+1}}\right]
= \limsup_{n \to \infty} \left[\frac{\delta (1+\nicefrac{1}{n})^\beta}{n^{\delta+1-\beta}}\right]
 = 0.
\end{split}
\end{equation}
This completes the proof of Lemma \ref{decay_fraction}.
\end{proof}

\subsection[Stability properties of the Euler scheme for ODEs]{On stability properties of the Euler scheme for ordinary differential equations}
\label{subsection:conditions_on_f}
In this subsection we study in the elementary observations in Lemmas~\ref{fcond}--\ref{le:EoP-l2} and Proposition~\ref{Equivalence_of_properties} below necessary and sufficient conditions which ensure that the Euler scheme admits a suitable Lyapunov-stability-type property (cf. Lemma~\ref{le:Euler} below). Similar results can be found, e.g., in Dereich \& M\"uller-Gronbach \cite[Remark~2.1]{DereichGronbach2017} and the references mentioned therein. Lemma~\ref{fcond} is employed in our strong error analysis in Proposition~\ref{L2_convergence_explicit} in Subsection~\ref{subsection:L2} and Proposition~\ref{Lp_convergence} in Subsection~\ref{subsection:Lp} below.
\begin{lemma}[Lyapunov-stability for the Euler scheme]
\label{le:Euler}
Let $d \in \N$, $\vartheta \in \R^d$, $c, \varrho \in (0,\infty)$, let 
$\left \| \cdot \right \| \! \colon \R^d \to [0,\infty)$ be a norm, let $g\colon \R^d \to \R^d$ and $V\colon \R^d \to \R$ be functions which satisfy of all $\theta \in \R^d$ that
\begin{equation}
V(\theta)= \norm{\theta-\vartheta}^2,
\end{equation}
and let $(\Theta^{r,\theta}_n)_{n\in \N_0} \colon \N_0 \to \R^d$, $r\in [0,\infty)$, $\theta \in \R^d$, be the functions which satisfy for all $r\in [0,\infty)$, $\theta \in \R^d$, $n \in \N$ that
\begin{equation}
\Theta_0^{r,\theta}=\theta  \qandq \Theta_{n}^{r,\theta}=\Theta_{n-1}^{r,\theta} + r g(\Theta_{n-1}^{r,\theta}).
\end{equation}
Then the following three statements are equivalent:
\begin{enumerate}[(i)]
\item\label{le:Euler-item1} It holds for all $r \in [0,\varrho]$, $\theta \in \R^d$, $n\in \N$ that
\begin{equation}
V(\Theta_n^{r,\theta})
\leq (1-cr) V(\Theta_{n-1}^{r,\theta}) 
\leq e^{-cr} \,V(\Theta_{n-1}^{r,\theta}).
\end{equation}
\item\label{le:Euler-item2} It holds for all $r \in [0,\varrho]$, $\theta \in \R^d$ that
\begin{equation}
V(\Theta_1^{r,\theta})\leq (1-cr) V(\theta) 
\leq e^{-cr} \,V(\theta).
\end{equation}
\item\label{le:Euler-item3} It holds for all $r \in [0,\varrho]$, $\theta \in \R^d$ that 
\begin{equation}
\norm{\theta + r g(\theta)-\vartheta}^2 \leq (1-cr) \norm{\theta-\vartheta}^2.
\end{equation}
\end{enumerate}
\end{lemma}
The proof of Lemma~\ref{le:Euler} is obvious. The next result, Lemma~\ref{fcond}, provides a  condition (see \eqref{fcond:eq1} in Lemma~\ref{fcond} below) which is sufficient to ensure that the stability property in item~\eqref{le:Euler-item3} in Lemma~\ref{le:Euler} holds (see item~\eqref{fcond:item4} in Lemma~\ref{fcond} below).
\begin{lemma}
\label{fcond}
Let $d \in \N$, $\vartheta \in \R^d$, $c_1, c_2 \in \left(0,\infty \right)$, let 
$\lll \cdot,\cdot \rrr \colon \R^d \times \R^d \to \R$
 be a scalar product, let 
 $\left \| \cdot \right \| \! \colon \R^d \to [0,\infty)$ 
 be the function which satisfies for all $\theta \in \R^d$ that $\norm{\theta} = \sqrt{\lll \theta,\theta \rrr}$, 
and let 
$
g \colon \R^d \to \R^d
$ 
be a function which satisfies for all $\theta \in \R^d$ that
\begin{equation}
\label{fcond:eq1}
\lll \theta - \vartheta, g(\theta)  \rrr  \leq - \max\!\left\{c_1 \norm{\theta - \vartheta}^2, c_2 \norm{g(\theta)}^2 \right\}.
\end{equation}
Then 
\begin{enumerate}[(i)]
\item \label{fcond:item1} it holds that
\begin{equation}
\{\theta \in \R^d \colon g(\theta) = 0 \} = \{\vartheta \},
\end{equation}
\item \label{fcond:item2a} it holds that $c_1c_2 \leq 1$,
\item \label{fcond:item2}
it holds for all $\theta \in \R^d$ that 
\begin{equation}
c_1 \norm{\theta-\vartheta} \leq \norm{g(\theta)} \leq \tfrac{1}{c_2}\norm{\theta-\vartheta},
\end{equation} 

\item \label{fcond:item3}
it holds for all $\theta \in \R^d$, $r \in [0,2c_2] $ that 
\begin{equation}
\norm{ \theta + r g(\theta) - \vartheta}^2 \leq \big( 1 - c_1r (2 - \tfrac{r}{c_2}) \big)\norm{\theta - \vartheta}^2,
\end{equation}
and 
\item\label{fcond:item4}
it holds for all $\theta \in \R^d$, $r \in [0,c_2] $ that 
\begin{equation}
\norm{ \theta + r g(\theta) - \vartheta}^2 \leq \left( 1 - c_1 r   \right)\norm{\theta - \vartheta}^2.
\end{equation}
\end{enumerate}
\end{lemma}

\begin{proof}[Proof of Lemma \ref{fcond}] 
First, note that (\ref{fcond:eq1}) (with $\theta = \vartheta$ in the notation of (\ref{fcond:eq1})) implies that 
\begin{equation}
0\leq -\max\{0,c_2\norm{g(\vartheta)}^2\}\leq -c_2\norm{g(\vartheta)}^2.
\end{equation}
Hence, we obtain that $0\geq \norm{g(\vartheta)}^2$.
 This assures that $g(\vartheta) = 0$. Next observe that (\ref{fcond:eq1}) and the Cauchy-Schwarz inequality ensure that for all $\theta \in \R^d$ it holds that
\begin{equation}
\label{fcond:eq3}
c_2 \norm{g(\theta)}^2 
\leq \max\big\{ c_1 \norm{\theta-\vartheta}^2, c_2 \norm{g(\vartheta)}^2 \big\}
\leq -\lll \theta - \vartheta, g(\theta) \rrr \leq \norm{\theta - \vartheta}\norm{g(\theta)}.
\end{equation}
Therefore, we obtain  that for all $\theta \in \R^d$ it holds that
\begin{equation}
c_1\norm{\theta-\vartheta}^2 \leq -\lll \theta-\vartheta, g(\theta) \rrr \leq \norm{\theta-\vartheta} \norm{g(\theta)}.
\end{equation}
Combining this with (\ref{fcond:eq3}) and the fact that $g(\vartheta) = 0$ proves that for all $ \theta \in \R^d$ it holds that
\begin{equation}\label{fcond:eq4}
c_1 \norm{\theta-\vartheta} \leq \norm{g(\theta)} \leq \tfrac{1}{c_2}\norm{\theta-\vartheta}.
\end{equation}
Therefore, we obtain that for all $\theta \in \R^d$ it holds that $c_1 c_2 \norm{\theta - \vartheta} \leq \norm{\theta - \vartheta}$.  This demonstrates that $c_1 c_2 \leq 1$. Combining (\ref{fcond:eq4}) and the fact that $g(\vartheta) = 0$ hence establishes items (\ref{fcond:item1})--(\ref{fcond:item2}).
Next observe that for all $r \in [0,2c_2]$ it holds that $r(2 - \frac{r}{c_2}) \geq 0$. This and (\ref{fcond:eq3}) imply that for all $\theta \in \R^d$, $r \in [0,2c_2]$ it holds that
\begin{equation}
\begin{split}
\norm{\theta + r g(\theta)-\vartheta}^2 
&= \norm{\theta - \vartheta}^2 + 2r\lll \theta - \vartheta , g(\theta) \rrr + r ^2 \norm{g(\theta)}^2 \\
&\leq \norm{\theta - \vartheta}^2 + 2r\lll \theta - \vartheta , g(\theta) \rrr  -\tfrac{r^2}{c_2} \lll \theta - \vartheta , g(\theta) \rrr \\
&= \norm{\theta - \vartheta}^2 + r\left(2 - \frac{r}{c_2}\right) \lll \theta - \vartheta , g(\theta) \rrr \\
&\leq \norm{\theta - \vartheta}^2 -  r \left( 2 -  \frac{r}{c_2} \right) c_1 \norm{\theta - \vartheta}^2  \\
&=  \left( 1 - c_1r  \left( 2 - \frac{r}{c_2}   \right) \right) \norm{\theta - \vartheta}^2.
\end{split}
\end{equation}
This proves item (\ref{fcond:item3}). Moreover, note that item (\ref{fcond:item3}) and the fact that for all $r \in [0,c_2]$ it holds that $2-\frac{r}{c_2} \geq 1$ establish item (\ref{fcond:item4}).
The proof of Lemma \ref{fcond} is thus completed.
\end{proof}


\begin{lemma}[On the monotonicity of a property]
\label{monotonicity}
Let $d \in \N$, $\vartheta \in \R^d$, $c,\varrho \in (0,\infty)$, let 
$\lll \cdot,\cdot \rrr \colon \R^d \times \R^d \to \R$
 be a scalar product, let 
 $\left \| \cdot \right \| \! \colon \R^d \to [0,\infty)$ 
 be the function which satisfies for all $\theta \in \R^d$ that $\norm{\theta} = \sqrt{\lll \theta, \theta \rrr}$, and let $g \colon \R^d \to \R^d$ be a function which satisfies for all $\theta \in \R^d$ that
\begin{equation}
\label{monotonicity:eq1}
\norm{ \theta + \varrho g(\theta) - \vartheta}^2 \leq ( 1 - c\varrho) \norm{\theta - \vartheta}^2.
\end{equation}
Then it holds for all $\theta \in \R^d$, $r \in [0,\varrho]$ that
\begin{equation}
\norm{ \theta + r g(\theta) - \vartheta}^2 \leq ( 1 - cr) \norm{\theta - \vartheta}^2.
\end{equation}
\end{lemma}

\begin{proof}[Proof of Lemma \ref{monotonicity}]
First, observe that (\ref{monotonicity:eq1}) implies that for all $\theta \in \R^d$ it holds that
\begin{equation}
\begin{split}
&\norm{\theta - \vartheta}^2 + 2 \varrho \lll \theta-\vartheta,g(\theta) \rrr + \varrho^2 \norm{g(\theta)}^2\\
&=\norm{\theta - \vartheta}^2 + 2 \lll \theta-\vartheta,\varrho  g(\theta) \rrr +  \norm{ \varrho g(\theta)}^2 \\
&=\norm{ (\theta - \vartheta) + \varrho g(\theta)}^2\\
&=\norm{ \theta + \varrho g(\theta) - \vartheta}^2\\
&\leq ( 1 - c\varrho) \norm{\theta - \vartheta}^2\\
&= \norm{\theta - \vartheta}^2 - c\varrho \norm{\theta - \vartheta}^2.
\end{split}
\end{equation}
Therefore, we obtain that for all $ \theta \in \R^d$ it holds that 
\begin{equation}
2 \lll \theta-\vartheta,g(\theta) \rrr + \varrho \norm{g(\theta)}^2 \leq -c\norm{\theta - \vartheta}^2.
\end{equation}
This ensures that for all $\theta \in \R^d$, $r \in [0,\varrho]$ it holds that 
\begin{equation}
\begin{split}
\norm{ \theta + r g(\theta) - \vartheta}^2 
&=\norm{ (\theta - \vartheta) + r g(\theta)}^2\\
&=\norm{\theta - \vartheta}^2 + 2 \lll \theta-\vartheta,r  g(\theta) \rrr +  \norm{ r g(\theta)}^2 \\
&=\norm{\theta - \vartheta}^2 + 2r \lll \theta-\vartheta, g(\theta) \rrr +  r^2\norm{g(\theta)}^2 \\
&= \norm{\theta - \vartheta}^2 + r \left(2 \lll \theta-\vartheta,g(\theta) \rrr + r \norm{g(\theta)}^2 \right) \\
&\leq \norm{\theta - \vartheta}^2 + r \left(2 \lll \theta-\vartheta,g(\theta) \rrr + \varrho \norm{g(\theta)}^2 \right) \\
& \leq \norm{\theta - \vartheta}^2 + r  (-c\norm{\theta - \vartheta}^2) \\
&= (1-cr)\norm{\theta - \vartheta}^2.
\end{split}
\end{equation}
The proof of Lemma \ref{monotonicity} is thus completed.
\end{proof}

\begin{lemma}\label{le:EoP-l1}
	Let $d \in \N$, $\vartheta \in \R^d$, $c,\varrho \in (0,\infty)$, let 
	$\lll \cdot,\cdot \rrr \colon \R^d \times \R^d \to \R$
	be a scalar product, let 
	$\left \| \cdot \right \| \! \colon \R^d \to [0,\infty)$ 
	be the function which satisfies for all $\theta \in \R^d$ that $\norm{\theta} = \sqrt{\lll \theta, \theta \rrr}$, and let $g \colon \R^d \to \R^d$ be a function which satisfies for all $\theta \in \R^d$, $r \in [0,\varrho]$ that
	\begin{equation}\label{le:Equivalence_of_properties:eq1}
	\norm{ \theta + r g(\theta) - \vartheta}^2 \leq ( 1 - cr) \norm{\theta - \vartheta}^2.
	\end{equation}
Then it holds that $g(\vartheta) = 0$ and 
\begin{equation}
\inf_{r \in (0,\infty)} \left(\sup_{\theta \in \R^d \setminus \{\vartheta\}} \left[ \frac{\norm{ \theta + r g(\theta) - \vartheta}^2}{\norm{\theta - \vartheta}^2} \right] \right) \leq 1-c\varrho< 1.
\end{equation}
\end{lemma}

\begin{proof}[Proof of Lemma \ref{le:EoP-l1}]
Observe that (\ref{le:Equivalence_of_properties:eq1}) (with $\theta = \vartheta$, $r=\varrho$ in the notation of (\ref{le:Equivalence_of_properties:eq1})) implies that $\norm{\varrho g(\vartheta)} \leq 0$. The hypothesis that
$\varrho \in (0,\infty)$ hence demonstrates that $g(\vartheta) = 0$. Moreover, note that (\ref{le:Equivalence_of_properties:eq1}) ensures that 
\begin{equation}
\begin{split}
\inf_{r \in (0,\infty)} \left(\sup_{\theta \in \R^d \setminus \{\vartheta\}} \left[ \frac{\norm{ \theta + r g(\theta) - \vartheta}^2}{\norm{\theta - \vartheta}^2} \right] \right)
&\leq \sup_{\theta \in \R^d \setminus \{\vartheta\}} \left[ \frac{\norm{ \theta + \varrho g(\theta) - \vartheta}^2}{\norm{\theta - \vartheta}^2} \right]\\
&\leq 1-c\varrho \\
&< 1.
\end{split}
\end{equation}
The proof of Lemma \ref{le:EoP-l1} is thus completed.
\end{proof}

\begin{lemma}\label{le:EoP-l2}
	Let $d \in \N$, $\vartheta \in \R^d$, $C,r \in (0,\infty)$, let 
	$\lll \cdot,\cdot \rrr \colon \R^d \times \R^d \to \R$
	be a scalar product, let 
	$\left \| \cdot \right \| \! \colon \R^d \to [0,\infty)$ 
	be the function which satisfies for all $\theta \in \R^d$ that $\norm{\theta} = \sqrt{\lll \theta, \theta \rrr}$, and let $g \colon \R^d \to \R^d$ be a function which satisfies that $g(\vartheta) = 0$ and 
	\begin{equation}\label{le:EoP-l2-eq1}
	\sup_{\theta \in \R^d \setminus \{\vartheta\}} \left[  \frac{2 \lll\theta - \vartheta, g(\theta) \rrr + r \norm{  g(\theta)}^2}{\norm{\theta - \vartheta}^2} \right] \leq -C.
	\end{equation}
Then it holds for all $\theta \in \R^d$ that
\begin{equation}
\lll\theta - \vartheta, g(\theta) \rrr \leq - \left[\tfrac{\min\{C,r \}}{2}\right] \max\! \big\{\norm{\theta - \vartheta}^2,  \norm{  g(\theta)}^2\big\}.
\end{equation}
\end{lemma}

\begin{proof}[Proof of Lemma \ref{le:EoP-l2}]
	First, note that (\ref{le:EoP-l2-eq1}) implies that
for all $\theta \in \R^d \setminus \{\vartheta\}$ it holds that 
\begin{equation}
2\lll\theta - \vartheta, g(\theta) \rrr + r\norm{  g(\theta)}^2 \leq  - C\norm{\theta - \vartheta}^2.
\end{equation}
Therefore, we obtain that
for all $\theta \in \R^d \setminus \{\vartheta\}$ it holds that 
\begin{equation}
\begin{split}
\lll\theta - \vartheta, g(\theta) \rrr 
&\leq -\frac{r}{2}\norm{  g(\theta)}^2 - \frac{C}{2}\norm{\theta - \vartheta}^2\\
&\leq -\left[\tfrac{\min\{C,r \}}{2}\right]\!\norm{  g(\theta)}^2 - \left[\tfrac{\min\{C,r \}}{2}\right]\!\norm{\theta - \vartheta}^2\\
&= -\left[\tfrac{\min\{C,r \}}{2}\right]\!\left[\norm{  g(\theta)}^2 + \norm{\theta - \vartheta}^2 \right]\\
&\leq - \left[\tfrac{\min\{C,r \}}{2}\right] \max\! \big\{\norm{\theta - \vartheta}^2,  \norm{  g(\theta)}^2\big\}.
\end{split}
\end{equation}
The assumption that $g(\vartheta) = 0$ hence shows that for all $\theta \in \R^d$  it holds that
\begin{equation}
\lll\theta - \vartheta, g(\theta) \rrr 
\leq - \left[\tfrac{\min\{C,r \}}{2}\right] \max\! \big\{\norm{\theta - \vartheta}^2,  \norm{  g(\theta)}^2\big\}.
\end{equation}
This completes the proof of Lemma \ref{le:EoP-l2}.	
\end{proof}

\begin{prop}[Equivalence of properties]
\label{Equivalence_of_properties}
Let $d \in \N$, $\vartheta \in \R^d$, let 
$\lll \cdot,\cdot \rrr \colon \R^d \times \R^d \to \R$
 be a scalar product, let 
 $\left \| \cdot \right \| \! \colon \R^d \to [0,\infty)$ 
 be the function which satisfies for all $\theta \in \R^d$ that $\norm{\theta} = \sqrt{\lll \theta, \theta \rrr}$, and let $g \colon \R^d \to \R^d$ be a function.
Then the following five statements are equivalent:
\begin{enumerate}[(i)]

\item \label{Equivalence_of_properties:item1}
There exists $c\in (0,\infty)$ such that for all $\theta \in \R^d$ it holds that
\begin{equation}
\lll \theta - \vartheta, g(\theta)  \rrr   \leq - c\max \!\big\{\norm{\theta - \vartheta}^2,  \norm{g(\theta)}^2 \big\}.
\end{equation}

\item \label{Equivalence_of_properties:item2}
There exist $c,\varrho \in (0,\infty)$ such that for all $\theta \in \R^d$ it holds that
\begin{equation}
\norm{ \theta + \varrho g(\theta) - \vartheta}^2 \leq ( 1 - c\varrho) \norm{\theta - \vartheta}^2.
\end{equation}

\item \label{Equivalence_of_properties:item3}
There exist $c,\varrho \in (0,\infty)$ such that for all $\theta \in \R^d$, $r \in[0,\varrho]$ it holds that
\begin{equation}
\norm{ \theta + r g(\theta) - \vartheta}^2 \leq ( 1 - cr) \norm{\theta - \vartheta}^2.
\end{equation}

\item \label{Equivalence_of_properties:item4}
It holds that $g(\vartheta) = 0$ and 
\begin{equation}
\label{Equivalence_of_properties:eq1}
\inf_{r \in (0,\infty)} \left(\sup_{\theta \in \R^d \setminus \{\vartheta\}} \left[ \frac{\norm{ \theta + r g(\theta) - \vartheta}^2}{\norm{\theta - \vartheta}^2} \right] \right) < 1.
\end{equation}

\item  \label{Equivalence_of_properties:item5}
It holds that $g(\vartheta) = 0$ and
\begin{equation}
\inf_{r \in (0,\infty) } \left( \sup_{\theta \in \R^d \setminus \{\vartheta\}} \left[  \frac{2 \lll\theta - \vartheta, g(\theta) \rrr + r \norm{  g(\theta)}^2}{\norm{\theta - \vartheta}^2} \right] \right) < 0.
\end{equation}
\end{enumerate}
\end{prop}

\begin{proof}[Proof of Proposition \ref{Equivalence_of_properties}]
First, note that item (\ref{fcond:item4}) in Lemma \ref{fcond} ensures that ((\ref{Equivalence_of_properties:item1}) $\Rightarrow$ (\ref{Equivalence_of_properties:item2})). Next observe  that Lemma \ref{monotonicity} implies that ((\ref{Equivalence_of_properties:item2}) $\Rightarrow$ (\ref{Equivalence_of_properties:item3})). Moreover, note that Lemma \ref{le:EoP-l1} demonstrates that ((\ref{Equivalence_of_properties:item3}) $\Rightarrow$ (\ref{Equivalence_of_properties:item4})). 
In addition, observe that the fact that for all $r \in (0,\infty)$ and all functions $h \colon \R^d \to \R^d$ it holds that 
\begin{equation}
\begin{split}
&\sup_{\theta \in \R^d \setminus \{\vartheta\}} \left[ \frac{\norm{ \theta + r h(\theta) - \vartheta}^2}{\norm{\theta - \vartheta}^2} \right] \\
& = 
\sup_{\theta \in \R^d \setminus \{\vartheta\}} \left[ \frac{\norm{(\theta-\vartheta) + r h(\theta)}^2}{\norm{\theta - \vartheta}^2} \right] \\
& = 
\sup_{\theta \in \R^d \setminus \{\vartheta\}} \left[ \frac{\norm{\theta-\vartheta}^2 + 2 \lll \theta- \vartheta, r h(\theta) \rrr  + \norm{r h(\theta)}^2}{\norm{\theta - \vartheta}^2} \right] \\
& = 
\sup_{\theta \in \R^d \setminus \{\vartheta\}} \left[ 1+ \frac{2r \lll \theta- \vartheta, h(\theta) \rrr  + r^2 \norm{h(\theta)}^2}{\norm{\theta - \vartheta}^2} \right] \\
& = 
1 + \sup_{\theta \in \R^d \setminus \{\vartheta\}} \left[  \frac{2 r \lll\theta - \vartheta, h(\theta) \rrr + r^2 \norm{  h(\theta)}^2}{\norm{\theta - \vartheta}^2} \right]  \\
&=
 1 + r \left(\sup_{\theta \in \R^d \setminus \{\vartheta\}} \left[  \frac{2 \lll\theta - \vartheta, h(\theta) \rrr + r \norm{  h(\theta)}^2}{\norm{\theta - \vartheta}^2} \right] \right)
\end{split}
\end{equation}
implies that ((\ref{Equivalence_of_properties:item4}) $\Leftrightarrow$ (\ref{Equivalence_of_properties:item5})).
Furthermore, note that Lemma \ref{le:EoP-l2} implies that ((\ref{Equivalence_of_properties:item5}) $\Rightarrow$ (\ref{Equivalence_of_properties:item1})). 
The proof of Proposition \ref{Equivalence_of_properties} is thus completed.
\end{proof}

\subsection{A Gronwall-type inequality}
\label{subsection:Gronwall}
In this subsection we establish in Lemma~\ref{main_estimate} a certain  Gronwall-type inequality. Lemma~\ref{main_estimate} is used in our strong error analysis in Proposition~\ref{Lyapunov_convergence_explicit} in Subsection~\ref{subsection:Lyapunov} below.
\begin{lemma}
\label{main_estimate}
Let $N \in \N_0$, $k,\kappa,c, C \in (0,\infty)$, $(e_n)_{n \in \N_0} \subseteq [0,\infty)$, $(\gamma_n)_{n \in \N_0} \subseteq (0,\infty)$ satisfy for all $n \in \N \cap (N,\infty)$ that
\begin{equation}
\label{main_estimate:assumption1}
e_n \leq (1-c\gamma_n)e_{n-1} +  \kappa (\gamma_n)^{k+1}, \qquad \sup_{l \in  \N \cap (N,\infty)}\gamma_l \leq \nicefrac{1}{c},
\end{equation}
\begin{equation}
\label{main_estimate:assumption2}
\mbox{and } \qquad C= \inf_{l \in  \N \cap (N,\infty)}\left[\tfrac{(\gamma_l)^{k}-(\gamma_{l-1})^{k}}{(\gamma_l)^{k+1}} + \tfrac{c (\gamma_{l-1})^k}{(\gamma_{l})^k} \right].
\end{equation}
Then it holds for all $n \in \N_0$ that
\begin{equation}
e_n \leq \left[\max\!\left\{\frac{e_0}{(\gamma_0)^{k}}, \frac{e_1}{(\gamma_1)^{k}}, \ldots, \frac{e_{N}}{(\gamma_{N})^{k}}, \frac{\kappa} {C}\right\}\right]\,(\gamma_n)^{k}.
\end{equation}
\end{lemma}
\begin{proof}[Proof of Lemma \ref{main_estimate}]
Throughout this proof let $\lambda \in (0,\infty)$ satisfy
\begin{equation}\label{main_estimate:eq0}
\lambda = \max\!\left\{ \frac{e_{N}}{(\gamma_{N})^{k}}, 
 \frac{\kappa}{C}\right\}.
\end{equation}
We claim that for all $n \in \{N,N+1, \ldots \}$ it holds that
\begin{equation}
\label{main_estimate:eq1}
e_n\leq \lambda(\gamma_n)^{k}.
\end{equation}
We now prove (\ref{main_estimate:eq1}) by induction on $n \in \{N,N+1, \ldots \}$.
For the base case $n= N$ observe that 
\begin{equation}
e_{N}=  \left[\frac{e_{N}}{(\gamma_{N})^{k}}\right](\gamma_{N})^{k} \leq \lambda (\gamma_{N})^{k} .
\end{equation} 
This proves (\ref{main_estimate:eq1}) in the base case $n=N$. For the induction step $\{N,N+1,\ldots\} \ni (n-1) \to n \in \N\cap(N,\infty)$ note that (\ref{main_estimate:assumption1}), (\ref{main_estimate:assumption2}), and (\ref{main_estimate:eq0}) demonstrate that for all $n \in \N\cap(N,\infty)$ with $e_{n-1}\leq \lambda(\gamma_{n-1})^{k}$ it holds that
\begin{equation}
\begin{split}
e_{n} 
&\leq
 (1-c\gamma_{n})e_{n-1} +  \kappa (\gamma_{n})^{k+1} \\
&\leq 
(1-c\gamma_{n})\lambda (\gamma_{n-1})^{k} +  \kappa (\gamma_{n})^{k+1} \\
&= 
\lambda(\gamma_{n-1})^k- \lambda c \gamma_n (\gamma_{n-1})^k + \kappa (\gamma_n)^{k+1}\\
&=  \lambda(\gamma_{n})^k - (\gamma_n)^{k+1} \left( \frac{\lambda c \gamma_n (\gamma_{n-1})^k}{(\gamma_n)^{k+1}} + \frac{\lambda (\gamma_{n})^k}{(\gamma_n)^{k+1}}- \frac{\lambda (\gamma_{n-1})^k}{(\gamma_n)^{k+1}} -\kappa\right)\\
&=  \lambda(\gamma_{n})^k - (\gamma_n)^{k+1} \left( \lambda\left[ \frac{c (\gamma_{n-1})^k}{(\gamma_n)^{k}} + \frac{(\gamma_{n})^k}{(\gamma_n)^{k+1}}- \frac{(\gamma_{n-1})^k}{(\gamma_n)^{k+1}}\right] -\kappa\right)\\
&=
\lambda (\gamma_{n})^{k} - (\gamma_{n})^{k+1}\left(\lambda \left[ \frac{(\gamma_{n})^{k}-(\gamma_{n-1})^{k}}{(\gamma_{n})^{k+1}} + \frac{c(\gamma_{n-1})^k}{(\gamma_{n})^k}  \right] - \kappa \right).
\end{split}
\end{equation}
Hence, we obtain that for all $n \in \N\cap(N,\infty)$ with $e_{n-1}\leq \lambda(\gamma_{n-1})^{k}$ it holds that
\begin{equation}
\begin{split}
e_{n} 
&\leq \lambda(\gamma_{n})^k - (\gamma_n)^{k+1} (\lambda C-\kappa)\\
&\leq \lambda(\gamma_{n})^k - (\gamma_n)^{k+1} \left(\left[\frac{\kappa}{C}\right] C-\kappa\right)\\
&=\lambda (\gamma_{n})^{k} - (\gamma_{n})^{k+1} \big(\kappa - \kappa\big) = \lambda (\gamma_{n})^{k}.
\end{split}
\end{equation}
Induction thus proves (\ref{main_estimate:eq1}). Next note that (\ref{main_estimate:eq1}) ensures that for all $n \in \N_0$ it holds that
\begin{equation}
\begin{split}
e_n \leq 
\bigg[\max&\!\left\{\frac{e_0}{(\gamma_0)^{k}}, \frac{e_1}{(\gamma_1)^{k}}, \ldots, \frac{e_{N-1}}{(\gamma_{N-1})^{k}}, \lambda \right\}\!\bigg] \,(\gamma_n)^{k} \\
= \bigg[\max&\!\left\{ \frac{e_0}{(\gamma_0)^{k}}, \frac{e_1}{(\gamma_1)^{k}}, \ldots, \frac{e_{N}}{(\gamma_{N})^{k}}, \frac{\kappa}{C} 
\right\}\!\bigg]\,(\gamma_n)^{k}.
\end{split}
\end{equation}
The proof of Lemma \ref{main_estimate} is thus completed.
\end{proof}

\begin{cor}
\label{main_estimate_asymptotic}
Let $k,\kappa,c \in (0,\infty)$, $(e_n)_{n \in \N_0} \subseteq [0,\infty)$,  $(\gamma_n)_{n \in \N_0} \subseteq (0,\infty)$  satisfy for all $n \in \N$ that
\begin{equation}
\label{main_estimate_asymptotic:assumption1}
e_n \leq (1-c\gamma_n)e_{n-1} +  \kappa (\gamma_n)^{k+1}\qquad \text{and}
\end{equation}
\begin{equation}
\label{main_estimate_asymptotic:assumption2}
\limsup_{l \to \infty} \gamma_l = 0 < \liminf_{l \to \infty}\left[\frac{(\gamma_l)^{k}-(\gamma_{l-1})^{k}}{(\gamma_l)^{k+1}} + \frac{c(\gamma_{l-1})^k}{(\gamma_{l})^k} \right].
\end{equation}
Then there exists $C \in (0,\infty)$ such that for all $n \in \N_0$ it holds that
\begin{equation}
\label{main_estimate_asymptotic:conclusion}
e_n \leq C(\gamma_n)^{k}.
\end{equation}
\end{cor}

\begin{proof}[Proof of Corollary \ref{main_estimate_asymptotic}]
Observe that (\ref{main_estimate_asymptotic:assumption2}) ensures that there exists $N \in \N_0$ such that
\begin{equation}
\sup_{l \in  \N \cap (N,\infty)}\gamma_l \leq \nicefrac{1}{c} \qandq \inf_{l \in  \N \cap (N,\infty)}\left[\frac{(\gamma_l)^{k}-(\gamma_{l-1})^{k}}{(\gamma_l)^{k+1}} + \frac{c(\gamma_{l-1})^k}{(\gamma_{l})^k} \right] > 0.
\end{equation}
Lemma \ref{main_estimate} therefore assures that for all $n \in \N_0$ it holds that
\begin{equation}
e_n \leq \Bigg[\max\!\left\{ \frac{e_0}{(\gamma_0)^{k}}, \frac{e_1}{(\gamma_1)^{k}}, \ldots, \frac{e_{N}}{(\gamma_{N})^{k}}, \tfrac{\kappa}{\inf_{l\in \N \cap (N,\infty)}\left[\frac{(\gamma_l)^{k}-(\gamma_{l-1})^{k}}{(\gamma_l)^{k+1}} + \frac{c(\gamma_{l-1})^k}{(\gamma_{l})^k} \right]}\right\}\!\Bigg] \,(\gamma_n)^{k}.  
\end{equation}
This completes the proof of Corollary \ref{main_estimate_asymptotic}.
\end{proof}

\section{Error analysis for stochastic approximation algorithms (SAAs)}
\label{section:main_section}
In this section we establish in Theorem~\ref{Lp_theorem} in Subsection~\ref{subsection:Lp} below for every $p \in (0,\infty)$ strong $L^p$-convergence rates for stochastic approximation algorithms.
\subsection{Main setting for the strong error analysis}
\label{setting}
Throughout this section the following setting is frequently used.
\begin{setting}\label{setting-Def}
Let $d \in \N$, $(\gamma_n)_{n \in \N_0} \subseteq (0,\infty)$, let $g \colon \R^d \to \R^d$ be $\mathcal{B}(\R^d)/ \mathcal{B}(\R^d)$-measurable, let $(\Omega , \mathcal{F}, \P , ( \mathbb{F}_n )_{ n \in \N_0 })$ be a filtered probability space, let
$D \colon  \N \times \Omega \to \R^d$ be an  $ ( \mathbb{F}_n )_{n \in \N } / \mathcal{B}(\R^d) $-adapted stochastic process, let $\Theta \colon  \N_0 \times \Omega \to \R^d$ be a function, assume that $\Theta_0$ is $\mathbb{F}_0/ \mathcal{B}(\R^d)$-measurable, and assume  for all $n \in \N$  that
\begin{equation}
\label{setting:eq1}
\Theta_n = \Theta_{n-1} + \gamma_n (g(\Theta_{n-1})+ D_n).
\end{equation}
\end{setting}
Note that in Setting \ref{setting-Def} the hypothesis that the function $g$ is $\mathcal{B}(\R^d)/ \mathcal{B}(\R^d) $-measurable, the hypothesis that the function $\Theta_0$ is $\mathbb{F}_0/ \mathcal{B}(\R^d)$-measurable, the hypothesis that $D$ is an $ ( \mathbb{F}_n )_{n \in \N } / \mathcal{B}(\R^d) $-adapted stochastic process, and (\ref{setting:eq1}) imply that $\Theta$ is an $ ( \mathbb{F}_n )_{n \in \N_0 } / \mathcal{B}(\R^d) $-adapted stochastic process.

\subsection{Lyapunov based convergence for SAAs}
\label{subsection:Lyapunov}

\begin{prop}[Lyapunov based convergence for stochastic approximation]
\label{Lyapunov_convergence_explicit}
Assume Setting \ref{setting-Def} and let  $N \in \N_0$, $k, \kappa, c, C \in (0,\infty)$, $V \in C^1(\R^d, [0,\infty))$ satisfy for all $m \in \N_0$, $n\in \N \cap (N,\infty)$, $t \in [0,1]$, $\theta\in \R^d$ that
\begin{equation}
\label{Lyapunov_convergence_explicit:assumption1}
\EXP{V(\Theta_m) + \abs{V'(\Theta_{n-1}+ \gamma_ng(\Theta_{n-1}))(D_n)} }< \infty,
\end{equation}
\begin{equation}
\label{Lyapunov_convergence_explicit:assumption2}
\textstyle\int_0^1\EXP{ \abs{V'(\Theta_{n-1}+ \gamma_n(g(\Theta_{n-1})+sD_n))(D_n)} }\,\mathrm{d}s < \infty, 
\end{equation}
\begin{equation}
\label{Lyapunov_convergence_explicit:assumption3}
\EXP{V'(\Theta_{n-1}+ \gamma_ng(\Theta_{n-1}))(D_n)} = 0, \qquad V(\theta + \gamma_n g(\theta)) \leq (1-c\gamma_n)V(\theta), 
\end{equation}
\begin{equation}
\label{Lyapunov_convergence_explicit:assumption4}
\begin{split}
&\EXP{\abs{V'(\Theta_{n-1}+ \gamma_n(g(\Theta_{n-1})+tD_n))(D_n) - V'(\Theta_{n-1}+ \gamma_n g(\Theta_{n-1}))(D_n)} } \\
&\leq \kappa\!\left((\gamma_n)^{k}+ \gamma_n\EXP{V(\Theta_{n-1})}\right), 
\end{split} 
\end{equation}
\begin{equation}
\label{Lyapunov_convergence_explicit:assumption5}
\sup_{l \in  \N \cap (N,\infty)}\gamma_l \leq \min\!\big\{\tfrac{c}{2\kappa}, \tfrac{2}{c} \big\}, \ \ \ \mbox{ and } \ \ \ C=\inf_{l \in  \N \cap (N,\infty)}\left[\tfrac{(\gamma_l)^{k}-(\gamma_{l-1})^{k}}{(\gamma_l)^{k+1}} + \tfrac{c (\gamma_{l-1})^k}{2 (\gamma_{l})^k} \right].
\end{equation}
Then it holds for all $n \in  \N_0$ that
\begin{equation}
\EXP{V(\Theta_n)}  \leq \Bigg[\max\!\left(\left\{\frac{\kappa}{C}\right\} \cup
\left\{\frac{\EXP{V(\Theta_l)} }{(\gamma_l)^{k}} : l \in \{0,1,\ldots, N\} \right\}  
\right) \Bigg](\gamma_n)^{k}<\infty.
\end{equation}
\end{prop}

\begin{proof}[Proof of Proposition \ref{Lyapunov_convergence_explicit}]
Throughout this proof let $(e_n)_{n \in \N_0}\subseteq [0,\infty]$ satisfy for all $n \in \N_0$ that 
\begin{equation}\label{Lyapunov_convergence_explicit:eq1}
e_n = \EXP{V(\Theta_{n})}.
\end{equation}
Note that \eqref{Lyapunov_convergence_explicit:assumption1} ensures that for all $n \in \N_0$ it holds that $e_n<\infty$. Moreover, observe that (\ref{Lyapunov_convergence_explicit:assumption1}) and (\ref{Lyapunov_convergence_explicit:assumption3}) assure that for all $n \in \N \cap (N,\infty)$ it holds that
\begin{equation}
\label{Lyapunov_convergence_explicit:eq1b}
\EXP{V(\Theta_{n-1}+ \gamma_{n}g(\Theta_{n-1}))}
\leq
 \EXP{(1-c\gamma_n) V(\Theta_{n-1})} 
\leq \EXP{V(\Theta_{n-1})}
<\infty.
\end{equation}
This, (\ref{setting:eq1}), and (\ref{Lyapunov_convergence_explicit:eq1}) imply that for all $n \in \N \cap (N,\infty)$ it holds that
\begin{equation}
\begin{split}
e_n&=\EXP{V(\Theta_{n})}\\
&= \EXP{V(\Theta_{n-1}+ \gamma_n(g(\Theta_{n-1})+D_n))} \\
&= \EXP{V(\Theta_{n-1}+ \gamma_{n}(g(\Theta_{n-1})+D_n))- V(\Theta_{n-1}+ \gamma_{n}g(\Theta_{n-1}))} \\
&\quad + \EXP{V(\Theta_{n-1}+ \gamma_{n}g(\Theta_{n-1}))}.
\end{split}
\end{equation}
Combining this with (\ref{Lyapunov_convergence_explicit:eq1b}) assures that for all $n \in \N \cap (N,\infty)$ it holds that
\begin{equation}
\label{Lyapunov_convergence_explicit:eq2}
\begin{split}
\EXP{V(\Theta_{n})} 
& \leq \EXP{V(\Theta_{n-1}+ \gamma_{n}(g(\Theta_{n-1})+D_n))- V(\Theta_{n-1}+ \gamma_{n}g(\Theta_{n-1}))}\\
&\quad + \EXP{(1-c\gamma_n)V(\Theta_{n-1})}\\
& = \EXP{V(\Theta_{n-1}+ \gamma_{n} g(\Theta_{n-1})+\gamma_{n} D_n)- V(\Theta_{n-1}+ \gamma_{n}g(\Theta_{n-1}))}\\
&\quad + \EXP{(1-c\gamma_n)V(\Theta_{n-1})}.
\end{split}
\end{equation}
Next note that the assumption that $V \in C^1(\R^d, [0,\infty))$, the chain rule, and the fundamental theorem of calculus ensure that for all $x,y \in \R^d$ it holds that $(\R\ni t\mapsto V(x+ty) \in [0,\infty)) \in  C^1(\R, [0,\infty))$ and
\begin{equation}
V(x+y) -V(x) 
=\big[V(x+ty)\big]_{t=0}^{t=1} 
= \int_0^1 V'(x+sy)(y)\,\mathrm{d}s.
\end{equation}
Combining (\ref{Lyapunov_convergence_explicit:assumption2}), (\ref{Lyapunov_convergence_explicit:assumption3}), and (\ref{Lyapunov_convergence_explicit:eq2}) with Fubini's theorem hence shows that for all $n \in \N \cap (N,\infty)$ it holds that
\begin{equation}
\begin{split}
&\EXP{V(\Theta_{n})} \\
& \leq
\EXPPP{\textstyle\int\limits_0^1 V'(\Theta_{n-1}+ \gamma_{n} g(\Theta_{n-1})+s\gamma_{n} D_n)(\gamma_n D_n)\,\mathrm{d}s} + (1-c\gamma_n) \,\EXP{V(\Theta_{n-1})}\\
&= 
\int\limits_0^1\EXP{ V'(\Theta_{n-1}+ \gamma_{n}(g(\Theta_{n-1})+sD_n))(\gamma_n D_n)}\,\mathrm{d}s + (1-c\gamma_n) \,\EXP{V(\Theta_{n-1})}
 \\
&= \gamma_n\int_0^1\EXP{ V'(\Theta_{n-1}+ \gamma_{n}(g(\Theta_{n-1})+ sD_n))( D_n)} \,\mathrm{d}s\\ 
 &\quad -\gamma_n \int_0^1 \EXP{ V'(\Theta_{n-1}+ \gamma_{n}g(\Theta_{n-1}))( D_n)}  \,\mathrm{d}s
 +(1-c\gamma_n)\,\EXP{V(\Theta_{n-1})}\\
&\leq  \gamma_n\sup_{s \in [0,1]}\EXP{ \abs{ V'(\Theta_{n-1}+ \gamma_{n}(g(\Theta_{n-1})+sD_n))(D_n) - V'(\Theta_{n-1}+ \gamma_{n}g(\Theta_{n-1}))( D_n)}}\\
&\quad + (1-c\gamma_n)\,\EXP{V(\Theta_{n-1})} .
\end{split}
\end{equation}
The fact that $\sup_{l \in \N \cap (N,\infty)}\gamma_{l} \leq \frac{c}{2\kappa}$ and (\ref{Lyapunov_convergence_explicit:assumption4}) therefore ensure that for all $n \in \N \cap (N,\infty)$ it holds that
\begin{equation}
\begin{split}
\EXP{V(\Theta_{n})} &\leq (1-c\gamma_n)\,\EXP{V(\Theta_{n-1})}+\gamma_n\kappa\big((\gamma_n)^{k}+ \gamma_n \EXP{V(\Theta_{n-1})}\big) \\[3pt]
&= \big(1-c\gamma_n + \kappa (\gamma_n)^2\big)\,\EXP{V(\Theta_{n-1})}+\kappa(\gamma_n)^{k+1} \\[3pt]
&\leq \Big(1-c\gamma_n + \frac{\gamma_n \kappa c}{2\kappa} \Big)\,\EXP{V(\Theta_{n-1})}+\kappa(\gamma_n)^{k+1} \\[3pt]
&= \Big(1-\frac{\gamma_n c}{2}\Big)\,\EXP{V(\Theta_{n-1})}+\kappa(\gamma_n)^{k+1}.
\end{split}
\end{equation}
Hence, we obtain that  for all $n \in \N \cap (N,\infty)$ it holds that
\begin{equation}
\begin{split}
e_n &= \EXP{V(\Theta_{n})} \leq  \Big(1-\frac{\gamma_{n} c}{2}\Big)\,\EXP{V(\Theta_{{n}-1})}+\kappa(\gamma_{n})^{k+1} \\
&=\Big(1-\frac{\gamma_{n} c}{2}\Big)e_{n-1}+\kappa(\gamma_{n})^{k+1} .
\end{split}
\end{equation}
Combining this with (\ref{Lyapunov_convergence_explicit:assumption5}) and Lemma \ref{main_estimate} (with $N = N$, $k = k$, $\kappa = \kappa$, $c = \nicefrac{c}{2}$, $e_n = e_n$, $\gamma_n = \gamma_{n}$ for $n \in \N_0$ in the notation of Lemma \ref{main_estimate}) demonstrates that for all $n \in \N_0$ it holds that
\begin{equation}
\label{139intro}
\begin{split}
\EXP{V(\Theta_{n})}&= e_n \\ 
 &\leq \bigg[\max\!\left\{ \frac{e_0}{(\gamma_0)^{k}}, \frac{e_1}{(\gamma_1)^{k}}, \ldots, \frac{e_{N}}{(\gamma_{N})^{k}}, \frac{\kappa}{C} \right\}\!\bigg]\,(\gamma_n)^{k} \\
&=  \Bigg[\max\!\left(\left\{\frac{\kappa}{C}\right\} \cup
\left\{\frac{\EXP{V(\Theta_l)} }{(\gamma_l)^{k}}: l \in \{0,1,\ldots, N\} \right\}
\right)\!\Bigg]\,(\gamma_n)^{k}.
\end{split}
\end{equation}
The proof of Proposition \ref{Lyapunov_convergence_explicit} is thus completed.
\end{proof}
\begin{cor}
	\label{Lyapunov_convergence_nonexplicit}
	Assume Setting \ref{setting-Def} and let $N \in \N_0$, $k, \kappa, c \in (0,\infty)$, $\varrho \in (0,\nicefrac{1}{c}]$, $V \in C^1(\R^d, [0,\infty))$ satisfy for all $m \in \N_0$, $n\in \N \cap (N,\infty)$, $r \in [0,\varrho]$, $t \in [0,1]$, $\theta\in \R^d$ that
	\begin{equation}
	\label{Lyapunov_convergence_nonexplicit:assumption1}
	\EXP{V(\Theta_m) + \abs{V'(\Theta_{n-1}+ \gamma_n g(\Theta_{n-1}))(D_n)} }< \infty,
	\end{equation}
	\begin{equation}
	\label{Lyapunov_convergence_nonexplicit:assumption2}
	\textstyle\int_0^1 \EXP{ \abs{V'(\Theta_{n-1}+ \gamma_n(g(\Theta_{n-1})+sD_n))(D_n)}}\, \mathrm{d}s < \infty, 
	\end{equation}
	\begin{equation}
	\label{Lyapunov_convergence_nonexplicit:assumption3}
	\EXP{V'(\Theta_{n-1}+ \gamma_n g(\Theta_{n-1}))(D_n)} = 0, \qquad V(\theta + rg(\theta)) \leq (1-c r)V(\theta), 
	\end{equation}
	\begin{equation}
	\label{Lyapunov_convergence_nonexplicit:assumption4}
	\begin{split}
	&\EXP{\abs{V'(\Theta_{n-1}+ \gamma_n(g(\Theta_{n-1})+tD_n))(D_n) - V'(\Theta_{n-1}+ \gamma_n g(\Theta_{n-1}))(D_n)} } \\
	&\leq \kappa\!\left((\gamma_n)^{k}+ \gamma_n\EXP{V(\Theta_{n-1})}\right), 
	\end{split} 
	\end{equation}
	\begin{equation}
	\label{Lyapunov_convergence_nonexplicit:assumption5}
	\text{and } \qquad \limsup_{l\to\infty} \gamma_l = 0 < \liminf_{l \to \infty}\left[\tfrac{(\gamma_l)^{k}-(\gamma_{l-1})^{k}}{(\gamma_l)^{k+1}} +\tfrac{c(\gamma_{l-1})^k}{2(\gamma_{l})^k} \right].
	\end{equation}
	Then there exists $C \in (0,\infty)$ such that for all $n \in \N_0$ it holds that
	\begin{equation}
	\label{Lyapunov_convergence_nonexplicit:conclusion1}
	\EXP{V(\Theta_n)} \leq C (\gamma_n)^{k}.
	\end{equation}
\end{cor}
\begin{proof}[Proof of Corollary \ref{Lyapunov_convergence_nonexplicit}]
First, note that (\ref{Lyapunov_convergence_nonexplicit:assumption5}) ensures that there exists $M \in \{N,N+1,\ldots \}$ such that 
$\sup_{l \in  \N \cap (M,\infty)}\gamma_l \leq \min\{\nicefrac{c}{2\kappa}, \varrho \}$ and
\begin{equation}
\label{Lyapunov_convergence_nonexplicit:eq1}
\inf_{l \in  \N \cap (M,\infty)}\left[\frac{(\gamma_l)^{k}-(\gamma_{l-1})^{k}}{(\gamma_l)^{k+1}} + \frac{c (\gamma_{l-1})^k}{2 (\gamma_{l})^k} \right] > 0.
\end{equation}
Next observe that (\ref{Lyapunov_convergence_nonexplicit:assumption3})
and the fact that $\forall\, n \in \N \cap (M,\infty) \colon \gamma_n \leq \varrho$
demonstrate that for all $n \in \N \cap (M,\infty)$, $\theta \in \R^d$ it holds that
\begin{equation}
V(\theta + \gamma_n g(\theta) ) \leq (1-c\gamma_n)V(\theta).
\end{equation}
The fact that $\sup_{l \in  \N \cap (M,\infty)}\gamma_l \leq \min\{\tfrac{c}{2\kappa}, \varrho\} \leq \min\{\tfrac{c}{2\kappa}, \tfrac{2}{c}\}$, (\ref{Lyapunov_convergence_nonexplicit:assumption1})--(\ref{Lyapunov_convergence_nonexplicit:assumption4}), (\ref{Lyapunov_convergence_nonexplicit:eq1}), and Proposition \ref{Lyapunov_convergence_explicit} (with $N = M$ in the notation of Proposition  \ref{Lyapunov_convergence_explicit}) hence assure that for all $n \in \N_0$ it holds that
\begin{equation}
\begin{split}
&\EXP{V(\Theta_n)}\\
&\leq\max\!\Bigg( \Bigg\{\tfrac{\kappa}{\inf\limits_{l\in \N \cap (M,\infty)}\left[\frac{(\gamma_l)^{k}-(\gamma_{l-1})^{k}}{(\gamma_l)^{k+1}} + \frac{c (\gamma_{l-1})^k}{2 (\gamma_{l})^k} \right]}\Bigg\} \cup
\Bigg\{\tfrac{\mathbb{E}[{V(\Theta_l)}]}{(\gamma_l)^{k}} : l \in \{0,1,\ldots, M\}\! \Bigg\}
\Bigg)\,(\gamma_n)^{k}.
\end{split}
\end{equation}
Combining this with \eqref{Lyapunov_convergence_nonexplicit:assumption1} establishes \eqref{Lyapunov_convergence_nonexplicit:conclusion1}. Corollary \ref{Lyapunov_convergence_nonexplicit} is thus completed.
\end{proof}

\subsection[Strong $L^2$-convergence rate for SAAs]{Strong $L^2$-convergence rate for SAAs}
\label{subsection:L2}
\begin{prop}[Mean square error of stochastic approximation]
\label{L2_convergence_explicit}
Assume Setting~\ref{setting-Def}, let 
$\lll \cdot,\cdot \rrr \colon \R^d \times \R^d \to \R$
 be a scalar product, let 
 $\left \| \cdot \right \| \! \colon \R^d \to [0,\infty)$ 
 be the function which satisfies for all $\theta \in \R^d$ that $\norm{\theta} = \sqrt{\lll \theta,\theta \rrr}$, let $N\in\N_0$, $c, \kappa, C\in (0,\infty)$, $\vartheta \in \R^d$, assume for all $n \in \N\cap(N,\infty)$, $A \in \mathbb{F}_{n-1}$ with $\Exp{\norm{D_n}} < \infty$ that $\Exp{D_n \ind{A}} = 0$, 
 and assume for all $n \in \N$, $\theta \in \R^d$ that 
\begin{equation}
\label{L2_convergence_explicit:assumption1}
 \EXP{\norm{\Theta_0}^2 } < \infty, \qquad \EXP{\norm{D_n}^2 } \leq \kappa \big(1 +\EXP{\norm{\Theta_{n-1}-\vartheta}^2}\big),
 \end{equation}
\begin{equation}
\label{L2_convergence_explicit:assumption2}
\sup_{l \in  \N \cap (N,\infty)}\gamma_l \leq \min\!\big\{\tfrac{c}{4\kappa}, c \big\}, \qquad C=\inf_{l \in  \N \cap (N,\infty)}\left[\tfrac{\gamma_l-\gamma_{l-1}}{(\gamma_l)^{2}} + \tfrac{c \gamma_{l-1}}{2 \gamma_{l}} \right],
\end{equation}
\begin{equation}
\label{L2_convergence_explicit:assumption3}
\text{and} \qquad \lll \theta - \vartheta, g(\theta)  \rrr   \leq - c\max\! \left\{ \norm{\theta - \vartheta}^2, \norm{g(\theta)}^2 \right\}.
\end{equation}
Then it holds for all $n \in  \N_0$ that
\begin{equation}
\begin{split}
\EXP{\norm{\Theta_n-\vartheta}^2}
&\leq \gamma_n \max\!\left(\left\{\tfrac{2 \kappa}{C} \right\} \cup
\left\{\tfrac{\E[\norm{\Theta_l-\vartheta}^2] }{\gamma_l} : l \in \{0,1,\ldots, N\} \right\}  
\right)<\infty.
\end{split}
\end{equation}
\end{prop}

\begin{proof}[Proof of Proposition \ref{L2_convergence_explicit}]
Throughout this proof let $V \colon \R^d\to [0,\infty)$ be the function which satisfies for all $\theta\in \R^d$ that 
\begin{equation}
V(\theta) = \norm{\theta-\vartheta}^2.
\end{equation}
Observe that (\ref{L2_convergence_explicit:assumption3}) and Lemma \ref{fcond} imply that for all $\theta \in \R^d$, $r\in [0,c]$ it holds that
\begin{equation}
\label{L2_convergence_explicit:eq1}
c \leq 1 \leq \nicefrac{1}{c}, \qquad \norm{g(\theta)} \leq \tfrac{1}{c} \norm{\theta-\vartheta}, \qquad\text{and} 
\end{equation}
\begin{equation}
V(\theta + r g(\theta)) = \norm{\theta + r g(\theta)-\vartheta}^2 \leq (1-cr)\norm{\theta -\vartheta}^2 = (1-cr)V(\theta).
\end{equation} 
This and \eqref{L2_convergence_explicit:assumption2} ensure that  for all $n \in \N \cap (N,\infty)$, $\theta \in \R^d$ it holds that
\begin{equation}
\label{L2_convergence_explicit:eq7}
V(\theta + \gamma_n g(\theta)) \leq (1-c\gamma_n)V(\theta).
\end{equation}
In addition, note that \eqref{L2_convergence_explicit:assumption2} and \eqref{L2_convergence_explicit:eq1} show that
\begin{equation}
\sup_{l \in  \N \cap (N,\infty)}\gamma_l 
\leq 
\min\!\big\{\tfrac{c}{4\kappa}, c \big\}
\leq 
\min\!\big\{\tfrac{c}{4\kappa}, 1 \big\}
\leq 
\min\!\big\{\tfrac{c}{4\kappa}, \tfrac{1}{c} \big\}
\leq 
\min\!\big\{\tfrac{c}{4\kappa}, \tfrac{2}{c} \big\}.
\end{equation}
Next we claim that for all $n \in \N$ it holds that
\begin{equation}
\label{L2_convergence_explicit:eq3}
\EXP{V(\Theta_{n-1})} = \EXP{\norm{\Theta_{n-1}-\vartheta}^2} < \infty \qquad \text{and} \qquad \EXP{\norm{D_n}^2} < \infty.
\end{equation} 
We now prove (\ref{L2_convergence_explicit:eq3}) by induction on $n \in \N$. For the base case $n = 1$ note that (\ref{L2_convergence_explicit:assumption1}) implies that
\begin{equation}
\EXP{\norm{\Theta_{0}-\vartheta}^2} < \infty \qquad \text{and} \qquad \EXP{\norm{D_1}^2} \leq  \kappa \big(1 +\EXP{\norm{\Theta_{0}-\vartheta}^2}\big) < \infty.
\end{equation}
This establishes (\ref{L2_convergence_explicit:eq3}) in the base case $n=1$.
For the induction step $\N \ni n \to n+1 \in \{ 2,3, \ldots \}$ observe that
\eqref{setting:eq1} and
 (\ref{L2_convergence_explicit:eq1}) ensure that for all $n \in \N$ with
$\EXP{V(\Theta_{n-1})+\norm{D_n}^2}<\infty$ it holds that
\begin{equation}
\begin{split}
\EXP{\norm{\Theta_{n}-\vartheta}^2} &= \EXP{\norm{\Theta_{n-1} +\gamma_n (g(\Theta_{n-1})+D_n) - \vartheta}^2}  \\
&\leq \EXPP{\big(\norm{\Theta_{n-1}  - \vartheta}+\gamma_n\norm{g(\Theta_{n-1})} + \gamma_n\norm{D_n}\big)^2} \\ 
&\leq \EXPP{ \big( (1+\tfrac{\gamma_n}{c})\norm{\Theta_{n-1} - \vartheta} + \gamma_n\norm{D_n}\big)^2} < \infty.
\end{split}
\end{equation}
This and (\ref{L2_convergence_explicit:assumption1}) imply that for all $n \in \N$ with
$\EXP{V(\Theta_{n-1})+\norm{D_n}^2}<\infty$ it holds that
\begin{equation}
\EXP{\norm{D_{n+1}}^2} 
\leq 
\kappa \big(1 +\EXP{\norm{\Theta_{n}-\vartheta}^2}\big)
 =\kappa \big(1 +\Exp{V(\Theta_n)}\big)  < \infty.
\end{equation}
Induction thus proves (\ref{L2_convergence_explicit:eq3}).
Next note that Lemma \ref{derivative_of_norm} implies that for all $\theta,v \in \R^d$ it holds that
\begin{equation}
\label{L2_convergence_explicit:eq2}
V \in C^1(\R^d,[0,\infty)) \qandq V'(\theta)(v)= 2\lll\theta-\vartheta, v\rrr.
\end{equation}
Furthermore, observe that (\ref{L2_convergence_explicit:eq1}) and (\ref{L2_convergence_explicit:eq3}) prove that for all $n \in \N$ it holds that
\begin{equation}
\begin{split}
\EXP{\norm{\Theta_{n-1} +\gamma_n g(\Theta_{n-1}) - \vartheta}^2} 
&\leq \EXP{(\norm{\Theta_{n-1}- \vartheta} +\gamma_n \norm{g(\Theta_{n-1})})^2} \\
&\leq \EXP{(\norm{\Theta_{n-1}- \vartheta} +\tfrac{\gamma_n}{c} \norm{\Theta_{n-1}- \vartheta})^2} \\
&= \EXP{([1+\tfrac{\gamma_n}{c}] \,\norm{\Theta_{n-1}- \vartheta})^2} \\
&=[1+\tfrac{\gamma_n}{c}]^2 \,\EXP{\norm{\Theta_{n-1}- \vartheta}^2} \\
&=[1+\tfrac{\gamma_n}{c}]^2 \,\EXP{V(\Theta_{n-1})}<\infty.
\end{split}
\end{equation}
The Cauchy-Schwarz inequality, (\ref{L2_convergence_explicit:eq3}), and (\ref{L2_convergence_explicit:eq2}) therefore ensure that for all $n \in \N$ it holds that
\begin{equation}
\label{L2_convergence_explicit:eq4}
\begin{split}
&\EXP{\abs{V'(\Theta_{n-1}+ \gamma_n g(\Theta_{n-1}))(D_n)}}  = \EXP{2\abs{\lll\Theta_{n-1}+ \gamma_n g(\Theta_{n-1})-\vartheta,D_n \rrr}}  \\
& \leq 2\,\EXP{\norm{\Theta_{n-1}+ \gamma_n g(\Theta_{n-1})-\vartheta}\norm{D_n}} \\
& \leq 2\,\big(\EXP{\norm{\Theta_{n-1}+ \gamma_n g(\Theta_{n-1})-\vartheta}^2}\big)^{\nicefrac{1}{2}} \, \big(\EXP{\norm{D_n}^2}\big)^{\nicefrac{1}{2}} < \infty.
\end{split}
\end{equation}
Combining \eqref{L2_convergence_explicit:eq3} and \eqref{L2_convergence_explicit:eq2} hence demonstrates that for all $n \in \N$ it holds that
\begin{equation}
\label{L2_convergence_explicit:eq5}
\begin{split}
&\int_0^1 \EXP{\abs{V'(\Theta_{n-1}+ \gamma_n(g(\Theta_{n-1})+sD_n))(D_n)}}\,\mathrm{d}s \\
&  = \int_0^1  \EXP{2\abs{\lll\Theta_{n-1}+ \gamma_n (g(\Theta_{n-1}) +s D_n)-\vartheta,D_n \rrr}}\,\mathrm{d}s \\
&  = \int_0^1  \EXP{2\abs{\lll\Theta_{n-1}+ \gamma_n g(\Theta_{n-1})-\vartheta,D_n \rrr + s\gamma_n \norm{D_n}^2}}\,\mathrm{d}s \\
&  \leq \EXP{2\abs{\lll\Theta_{n-1}+ \gamma_n g(\Theta_{n-1})-\vartheta,D_n \rrr}} + \gamma_n \,\EXP{\norm{D_n}^2} \\
&  \leq 2\big(\,\EXP{\norm{\Theta_{n-1}+ \gamma_n g(\Theta_{n-1})-\vartheta}^2}\big)^{\nicefrac{1}{2}} \, \big(\EXP{\norm{D_n}^2}\big)^{\nicefrac{1}{2}} + \gamma_n\,\EXP{\norm{D_n}^2} < \infty.
\end{split}
\end{equation}
Moreover, observe that the fact that for all $n \in \N$ it holds that the function $\Theta_{n-1}$ is $\mathbb{F}_{n-1}/\mathcal{B}(\R^d)$-measurable, \eqref{L2_convergence_explicit:eq3}, the fact that for all $n\in \N\cap (N,\infty)$, $A \in \F_{n-1}$ it holds that $\Exp{D_n\ind{A}} = 0$, \eqref{L2_convergence_explicit:eq2}, (\ref{L2_convergence_explicit:eq4}), and Lemma \ref{Orthogonality} prove that for all $n \in \N \cap (N,\infty)$ it holds that
\begin{equation}
\label{L2_convergence_explicit:eq6}
\EXP{V'(\Theta_{n-1}+ \gamma_n g(\Theta_{n-1}))(D_n)} = 2\,\EXP{\lll\Theta_{n-1}+ \gamma_n g(\Theta_{n-1})-\vartheta,D_n \rrr} = 0.
\end{equation}
Furthermore, note that (\ref{L2_convergence_explicit:assumption1}) and \eqref{L2_convergence_explicit:eq2} imply that for all $n \in \N$, $t \in [0,1]$ it holds that
\begin{equation}
\label{L2_convergence_explicit:eq8}
\begin{split}
&\EXP{\abs{V'(\Theta_{n-1}+ \gamma_n(g(\Theta_{n-1})+tD_n))(D_n) - V'(\Theta_{n-1}+ \gamma_n g(\Theta_{n-1}))(D_n)} } \\
 &= \EXP{2\abs{\lll\Theta_{n-1}+ \gamma_n(g(\Theta_{n-1})+tD_n)- (\Theta_{n-1}+ \gamma_n g(\Theta_{n-1}) ),D_n \rrr}} \\
&=\EXP{2t\gamma_n\norm{D_n}^2} =2t\gamma_n\,\EXP{\norm{D_n}^2} \leq 2\gamma_n\,\EXP{\norm{D_n}^2}\\
&\leq 2\gamma_n \kappa\big(1+\EXP{\norm{\Theta_{n-1}-\vartheta}^2}\big) = 2\kappa\big(\gamma_n + \gamma_n \,\EXP{V(\Theta_{n-1})}\big).
\end{split}
\end{equation}
Combining the fact that $\sup_{l \in  \N \cap (N,\infty)}\gamma_l 
\leq 
\min\!\big\{\tfrac{c}{4\kappa}, \tfrac{2}{c} \big\}$, (\ref{L2_convergence_explicit:assumption2}), (\ref{L2_convergence_explicit:eq7}), (\ref{L2_convergence_explicit:eq3}), (\ref{L2_convergence_explicit:eq4}), (\ref{L2_convergence_explicit:eq5}), and (\ref{L2_convergence_explicit:eq6}) with Proposition \ref{Lyapunov_convergence_explicit} (with $N = N$, $k = 1$, $\kappa = 2 \kappa$, $c = c$, and $V= V$ in the notation of Proposition \ref{Lyapunov_convergence_explicit}) therefore demonstrates that for all $n \in \N_0$ it holds that
\begin{equation}
\label{168intro}
\begin{split}
&\EXP{\norm{\Theta_n-\vartheta}^2}\\
& = \EXP{V(\Theta_n)} \\
 &\leq \gamma_n \max\!\left(\left\{\frac{2 \kappa}{C}\right\}  \cup
\left\{\frac{\EXP{V(\Theta_l)} }{\gamma_l} : l \in \{0,1,\ldots, N\} \right\}  
\right) \\
&= \gamma_n \max\!\left(\left\{\frac{2 \kappa}{C}\right\} \cup
\left\{\frac{\EXP{\norm{\Theta_l-\vartheta}^2} }{\gamma_l} : l \in \{0,1,\ldots, N\} \right\}  
\right) <\infty.
\end{split}
\end{equation}
The proof of Proposition \ref{L2_convergence_explicit} is thus completed.
\end{proof}

\begin{cor}
\label{L2_convergence_nonexplicit}
Assume Setting~\ref{setting-Def}, let 
$\lll \cdot,\cdot \rrr \colon \R^d \times \R^d \to \R$
 be a scalar product, let 
 $\left \| \cdot \right \| \! \colon \R^d \to [0,\infty)$ 
 be the function which satisfies for all $\theta \in \R^d$ that $\norm{\theta} = \sqrt{\lll \theta,\theta \rrr}$, let $c, \kappa\in (0,\infty)$, $\vartheta \in \R^d$, assume for all $n \in \N$, $A \in \mathbb{F}_{n-1}$ with $\Exp{\norm{D_n}} < \infty$ that $\Exp{D_n \ind{A}} = 0$, 
 and assume for all $n \in \N$, $\theta \in \R^d$ that 
\begin{equation}
\label{L2_convergence_nonexplicit:assumption1}
 \EXP{\norm{\Theta_0}^2 } < \infty, \qquad \EXP{\norm{D_n}^2 } \leq \kappa \big(1 +\EXP{\norm{\Theta_{n-1}-\vartheta}^2}\big),
 \end{equation}
\begin{equation}
\label{L2_convergence_nonexplicit:assumption2}
\limsup_{l\to\infty} \gamma_l = 0 < \liminf_{l \to \infty}\left[\tfrac{\gamma_l-\gamma_{l-1}}{(\gamma_l)^2} + \tfrac{c \gamma_{l-1}}{2 \gamma_{l}} \right],
\end{equation}
\begin{equation}
\label{L2_convergence_nonexplicit:assumption3}
\text{and} \qquad \lll \theta - \vartheta, g(\theta)  \rrr   \leq - c\max\! \left\{ \norm{\theta - \vartheta}^2, \norm{g(\theta)}^2 \right\}.
\end{equation}
Then there exists $C \in (0,\infty)$ such that for all $n \in \N_0$ it holds that
\begin{equation}
\EXP{\norm{\Theta_n-\vartheta}^2} \leq C \gamma_n.
\end{equation}
\end{cor}

\begin{proof}[Proof of Corollary \ref{L2_convergence_nonexplicit}]
Observe that (\ref{L2_convergence_nonexplicit:assumption2}) ensures that there exists $N \in \N_0$ such that
\begin{equation}
\sup_{l \in  \N \cap (N,\infty)}\gamma_l \leq \min\!\bigg\{\frac{c}{4\kappa}, c \bigg\} \qandq \inf_{l \in  \N \cap (N,\infty)}\!\left[\frac{\gamma_l-\gamma_{l-1}}{(\gamma_l)^{2}} + \frac{c \gamma_{l-1}}{2 \gamma_{l}} \right] > 0.
\end{equation}
Proposition \ref{L2_convergence_explicit} therefore establishes that for all $n \in \N_0$ it holds that
\begin{equation}
\begin{split}
&\EXP{\norm{\Theta_n-\vartheta}^2}  \\
&\leq \gamma_n \max\!\left(\left\{\frac{2 \kappa}{\inf\limits_{l\in \N \cap (N,\infty)}\left[\frac{\gamma_l-\gamma_{l-1}}{(\gamma_l)^2} +\frac{c \gamma_{l-1}}{2\gamma_{l}} \right]}\right\} \cup
\Bigg\{\tfrac{\mathbb E[\norm{\Theta_l-\vartheta}^2] }{\gamma_l} : l \in \{0,1,\ldots, N\} \Bigg\}  
\right) <\infty.
\end{split}
\end{equation}
This completes the proof of Corollary \ref{L2_convergence_nonexplicit}.
\end{proof}

\subsection[Strong $L^p$-convergence rate for SAAs]{Strong $L^p$-convergence rate for SAAs}
\label{subsection:Lp}
\begin{prop}[$L^p$-convergence rate for stochastic approximation]
\label{Lp_convergence}
Assume Setting~\ref{setting-Def}, let 
$\lll \cdot,\cdot \rrr \colon \R^d \times \R^d \to \R$
 be a scalar product, let 
 $\left \| \cdot \right \| \! \colon \R^d \to [0,\infty)$ 
 be the function which satisfies for all $\theta \in \R^d$ that $\norm{\theta} = \sqrt{\lll \theta,\theta \rrr}$, let $p \in \{2,4,6,\ldots \}$, $c, \kappa \in (0,\infty)$, $\vartheta \in \R^d$, assume for all $n \in \N$, $A \in \mathbb{F}_{n-1}$ with $\Exp{\norm{D_n}} < \infty$ that $\Exp{D_n \mathbbm{1}_A} = 0$,
and assume for all $n \in \N$, $A \in \mathbb{F}_{n-1}$, $\theta \in \R^d$ that 
\begin{equation}
\label{Lp_convergence:assumption1}
 \EXP{\norm{\Theta_0}^p} < \infty, \qquad \EXP{\norm{D_n}^p \ind{A}} \leq \kappa\,\EXP{(1 +\norm{\Theta_{n-1}-\vartheta}^p) \mathbbm{1}_A},
 \end{equation}
\begin{equation}
\label{Lp_convergence:assumption2}
\quad \limsup_{l\to\infty} \gamma_l = 0 < \min_{k \in \{1,2,\ldots,\nicefrac{p}{2}\}}\!\left(\liminf_{l \to \infty}\left[\tfrac{(\gamma_l)^{k}-(\gamma_{l-1})^{k}}{(\gamma_l)^{k+1}} +\tfrac{c(\gamma_{l-1})^k}{2(\gamma_{l})^k} \right] \right),
\end{equation}
\begin{equation}
\label{Lp_convergence:assumption3}
\text{and} \qquad \lll \theta - \vartheta, g(\theta)  \rrr   \leq - c\max\!\big\{ \norm{\theta - \vartheta}^2, \norm{g(\theta)}^2 \big\}.
\end{equation}
Then there exists $C \in (0,\infty)$ such that for all $n \in \N_0$ it holds that
\begin{equation}
\big(\EXP{\norm{\Theta_n-\vartheta}^p}\big)^{\nicefrac{1}{p}} \leq C (\gamma_n)^{\nicefrac{1}{2}}.
\end{equation}
\end{prop}

\begin{proof}[Proof of Proposition \ref{Lp_convergence}]
Throughout this proof assume w.l.o.g.\ that $\kappa \geq 1$, let $N \in \N$ satisfy $\sup_{n \in \N \cap (N,\infty)}\gamma_n \leq c$, and let $V_q \colon \R^d\to [0,\infty)$, $q \in \N$, be the functions which satisfy for all $q \in \N$, $\theta\in \R^d$ that 
\begin{equation}
\qquad V_q(\theta) = \norm{\theta-\vartheta}^q.
\end{equation}
Note that Lemma \ref{fcond} implies that for all $q \in \{2,4,6,\ldots\}\cap [2,p]$, $\theta \in \R^d$, $r\in [0,c]$ it holds that
\begin{equation}
\label{Lp_convergence:eq1}
c \leq 1 \leq \nicefrac{1}{c}, \qquad \norm{g(\theta)} \leq \tfrac{1}{c} \norm{\theta-\vartheta}, \qquad\text{and} 
\end{equation}
\begin{equation}
\label{Lp_convergence:eq2}
V_q(\theta + r g(\theta)) = \norm{\theta + r g(\theta)-\vartheta}^q \leq (1-cr)^{\nicefrac{q}{2}}\norm{\theta -\vartheta}^q \leq (1-cr)V_q(\theta).
\end{equation}
In the next step we claim that for all $n \in \N$ it holds that
\begin{equation}
\label{Lp_convergence:eq3}
\EXP{\norm{\Theta_{n-1}-\vartheta}^p}  < \infty \qquad\text{and}\qquad \EXP{\norm{D_n}^p} < \infty.
\end{equation}
We now prove (\ref{Lp_convergence:eq3}) by induction on $n \in \N$. 
For the base case $n = 1$ note that (\ref{Lp_convergence:assumption1}) implies that
\begin{equation}
\EXP{\norm{\Theta_{0}-\vartheta}^p} < \infty \qquad \text{and} \qquad \EXP{\norm{D_1}^p} \leq  \kappa \big(1 +\EXP{\norm{\Theta_{0}-\vartheta}^p }\big) < \infty.
\end{equation}
This establishes (\ref{Lp_convergence:eq3}) in the base case $n=1$.
For the induction step $\N \ni n \to n+1 \in \{ 2,3, \ldots \}$ observe that \eqref{setting:eq1} and (\ref{Lp_convergence:eq1}) ensure that for all  $n \in \N$ with
$\EXP{\norm{\Theta_{n-1}-\vartheta}^p + \norm{D_n}^p}<\infty$ it holds that
\begin{equation}
\begin{split}
\EXP{\norm{\Theta_{n}-\vartheta}^p} &= \EXP{\norm{\Theta_{n-1} +\gamma_n (g(\Theta_{n-1}) + D_n) - \vartheta}^p}  \\
&\leq \EXP{\big(\norm{\Theta_{n-1}  - \vartheta}+\gamma_n\norm{g(\Theta_{n-1})} + \gamma_n\norm{D_n}\big)^p} \\ 
&\leq \EXP{ \big( (1+\tfrac{\gamma_n}{c})\norm{\Theta_{n-1} - \vartheta} + \gamma_n\norm{D_n}\big)^p} < \infty.
\end{split}
\end{equation}
This and (\ref{Lp_convergence:assumption1}) imply that for all  $n \in \N$ with
$\EXP{\norm{\Theta_{n-1}-\vartheta}^p + \norm{D_n}^p}<\infty$ it holds that
\begin{equation}
\label{Lp_convergence:eq4a}
\EXP{\norm{D_{n+1}}^p} \leq \kappa \big(1 +\EXP{\norm{\Theta_{n}-\vartheta}^p }\big) < \infty.
\end{equation}
Induction thus proves (\ref{Lp_convergence:eq3}).
%
Next note that the conditional Jensen inequality (see, e.g., Klenke \cite[Theorem~8.20]{Klenke}) and the fact that for all $q \in \{1,2,\ldots\}\cap[0,p]$ it holds that the function $ \R \ni z \mapsto |z|^{\nicefrac{p}{q}} \in [0,\infty)$ is convex ensure that for all $q \in \{1,2,\ldots\} \cap[0,p]$, $n \in \N$  it holds $\P$-a.s.\ that
\begin{equation}
\big|\EXP{\norm{D_n}^q \!\mid\! \F_{n-1}}\big|^{\nicefrac{p}{q}} \leq \EXP{\norm{D_n}^p \!\mid\! \F_{n-1}}.
\end{equation}
Hence, we obtain that for all $q\in \{1,2,\dots\}\cap[0,p]$, $n\in \N$ it holds $\mathbb{P}$-a.s.\ that
\begin{equation}
\label{NRX}
\EXP{\norm{D_n}^q \!\mid\! \F_{n-1}} \leq \big|\EXP{\norm{D_n}^p \!\mid\! \F_{n-1}}\big|^{\nicefrac{q}{p}}.
\end{equation}
Moreover, observe that (\ref{Lp_convergence:assumption1}) and (\ref{Lp_convergence:eq3}) demonstrate that for all $n \in \N$ it holds $\mathbb{P}$-a.s.\ that
\begin{equation}
\EXP{\norm{D_n}^p \!\mid\! \F_{n-1}}\leq \kappa\big(1 +\norm{\Theta_{n-1}-\vartheta}^p\big).
\end{equation}
Combining this with \eqref{NRX} proves that for all $q\in \{1,2,\dots\}\cap[0,p]$, $n\in \N$ it holds $\mathbb{P}$-a.s.\ that
\begin{equation}
\EXP{\norm{D_n}^q \!\mid\! \F_{n-1}} \leq \big|\EXP{\norm{D_n}^p \!\mid\! \F_{n-1}}\big|^{\nicefrac{q}{p}} \leq \big[\kappa\big(1 +\norm{\Theta_{n-1}-\vartheta}^p\big) \big]^{\nicefrac{q}{p}}.
\end{equation}
The fact that $\forall\, x,y \in [0,\infty)$, $r \in (0,1] \colon (x+y)^r \leq x^r + y^r$ and the assumption that $\kappa\geq 1$ hence assure that for all $q\in \{1,2,\dots\}\cap[0,p]$, $n \in \N$  it holds $\P$-a.s.\ that
\begin{equation}
\label{Lp_convergence:eq4}
\begin{split}
\EXP{\norm{D_n}^q \!\mid\! \F_{n-1}} & \leq \big[\kappa\big(1 +\norm{\Theta_{n-1}-\vartheta}^p\big) \big]^{\nicefrac{q}{p}} \\
& \leq \kappa^{\nicefrac{q}{p}}\big(1 +\norm{\Theta_{n-1}-\vartheta}^q\big)\\
&\leq \kappa\big(1 +\norm{\Theta_{n-1}-\vartheta}^q\big).
\end{split}
\end{equation}
The tower property for conditional expectations therefore shows that for all $q\in \{1,2,\dots\}\cap[0,p]$, $n \in \N$ it holds that
\begin{equation}
\label{V}
\begin{split}
\Exp{\norm{D_n}^q}
 &= \EXP{\Exp{\norm{D_n}^q \!\mid\! \F_{n-1}}\!}\\
& \leq \EXP{\kappa\big(1 +\norm{\Theta_{n-1}-\vartheta}^q\big)}\\
 &=\kappa\big(1+\EXP{\norm{\Theta_{n-1}-\vartheta}^q}\big).
 \end{split}
\end{equation}
Next we claim that for all $q \in \{2,4,6,\ldots\}\cap[0,p]$ there exists $C \in (0,\infty)$ such that for all $n\in \N_0$ it holds that
\begin{equation}
\label{Lp_convergence:eq10}
\EXP{\norm{\Theta_n-\vartheta}^q} \leq C (\gamma_n)^{\nicefrac{q}{2}}.
\end{equation}
We now prove (\ref{Lp_convergence:eq10}) by induction on $q \in \{2,4,6,\ldots\}\cap[0,p]$.
Observe that Corollary~\ref{L2_convergence_nonexplicit} and (\ref{V}) establish (\ref{Lp_convergence:eq10}) in the base case $q = 2$. For the induction step $\{2,4,6,\ldots\}\cap[0,p-2] \ni (q-2) \to q \in \{4,6,8,\ldots\}\cap[0,p]$ let $q \in \{4,6,8,\ldots\}\cap[0,p]$, $C \in (0,\infty)$ satisfy
for all $n\in \N_0$ that 
\begin{equation}
\label{Lp_convergence:eq5}
\EXP{\norm{\Theta_n-\vartheta}^{q-2}} \leq C (\gamma_n)^{\nicefrac{(q-2)}{2}}.
\end{equation}
Note that (\ref{Lp_convergence:eq3}) and Jensen's inequality ensure that for all $n \in \N_0$ it holds that
\begin{equation}
\label{Lp_convergence:eq6}
\EXP{V_q(\Theta_n)} = \EXP{\norm{\Theta_n-\vartheta}^q}  < \infty.
\end{equation}
Next observe that \eqref{Lp_convergence:eq1} implies that for all $\theta\in\R^d$, $n\in\N$ it holds that
\begin{equation}
\label{prop36:new1}
\begin{split}
\norm{\theta+\gamma_ng(\theta)-\vartheta}^{q-1}
&\leq (\norm{\theta-\vartheta}+\norm{\gamma_ng(\theta)})^{q-1}\\
&\leq \left(\norm{\theta-\vartheta}+\abs{\tfrac{\gamma_n}{c}}\norm{\theta-\vartheta}\right)^{q-1}\\
&=\left(1+\tfrac{\gamma_n}c\right)^{q-1}\norm{\theta-\vartheta}^{q-1}.
\end{split}
\end{equation}
Combining this and \eqref{Lp_convergence:eq3} with Jensen's inequality ensures that for all $n\in\N$ it holds that
\begin{equation}
\label{prop36:new5}
\E\!\left[\norm{\Theta_{n-1}+\gamma_ng(\Theta_{n-1})-\vartheta}^{q-1}\right]
\leq \left(1+\tfrac{\gamma_n}c\right)^{q-1}\E\!\left[\norm{\Theta_{n-1}-\vartheta}^{q-1}\right]
<\infty.
\end{equation}
Moreover, note that \eqref{prop36:new1}, the tower property for conditional expectations,
and the fact that for all $n\in\N$ it holds that $\Theta_{n-1}$ is $\mathbb F_{n-1}/\mathcal B(\R^d)$-measurable
assure that for all $n\in\N$ it holds that
\begin{equation}
\begin{split}
&\E\!\left[\norm{\Theta_{n-1}+\gamma_ng(\Theta_{n-1})-\vartheta}^{q-1}\norm{D_n}\right]\\
&\leq \left(1+\tfrac{\gamma_n}c\right)^{q-1}\E\!\left[\norm{\Theta_{n-1}-\vartheta}^{q-1}\norm{D_n}\right]\\
&=\left(1+\tfrac{\gamma_n}c\right)^{q-1}\E\!\left[\norm{\Theta_{n-1}-\vartheta}^{q-1}\E\!\left[\norm{D_n}\!\mid\!\mathbb F_{n-1}\right]\right].
\end{split}
\end{equation}
Combining this with \eqref{Lp_convergence:eq3}, \eqref{Lp_convergence:eq4}, and Jensen's inequality proves that for all $n\in\N$ it holds that
\begin{equation}
\begin{split}
&\E\!\left[\norm{\Theta_{n-1}+\gamma_ng(\Theta_{n-1})-\vartheta}^{q-1}\norm{D_n}\right]\\
&\leq \kappa\left(1+\tfrac{\gamma_n}c\right)^{q-1}\E\!\left[\norm{\Theta_{n-1}-\vartheta}^{q-1}\left(1+\norm{\Theta_{n-1}-\vartheta}\right)\right]\\
&=\kappa\left(1+\tfrac{\gamma_n}c\right)^{q-1}\bigl(\E[\norm{\Theta_{n-1}-\vartheta}^{q-1}]+\E\left[\norm{\Theta_{n-1}-\vartheta}^q\right]\bigr)<\infty.
\label{prop36:new2}
\end{split}
\end{equation}
Furthermore, observe that Lemma \ref{derivative_of_norm} implies that for all 
 $\theta, v \in \R^d$ it holds that
\begin{equation}
\label{Lp_convergence:eqV'}
V_q \in C^1(\R^d,[0,\infty)) \qandq V_q'(\theta)(v)= q\norm{\theta-\vartheta}^{q-2}\lll\theta-\vartheta, v\rrr.
\end{equation}
This, \eqref{prop36:new2}, and the Cauchy-Schwarz inequality demonstrate that for all $n\in\N$ it holds that
\begin{equation}
\label{prop36:new6}
\begin{split}
&\E\!\left[\abs{V_q'(\Theta_{n-1}+\gamma_ng(\Theta_{n-1}))(D_n)}\right]\\
&=q\,\E\!\left[\norm{\Theta_{n-1}+\gamma_ng(\Theta_{n-1})-\vartheta}^{q-2}\abs{\lll\Theta_{n-1}+\gamma_ng(\Theta_{n-1})-\vartheta,D_n\rrr}\right]\\
&\leq q\,\E\!\left[\norm{\Theta_{n-1}+\gamma_ng(\Theta_{n-1})-\vartheta}^{q-1}\norm{D_n}\right]<\infty.
\end{split}
\end{equation}
%
Furthermore, note that  \eqref{Lp_convergence:eqV'}, the Cauchy-Schwarz inequality, and Lemma \ref{weak_triangle} (with $p = q-1$ in the notation of Lemma \ref{weak_triangle}) imply that for all $n \in \N$ it holds that
\begin{equation}
\label{Lp_convergence:eq8a}
\begin{split}
&\int_0^1\EXP{\abs{V_q'(\Theta_{n-1}+ \gamma_n(g(\Theta_{n-1})+sD_n))(D_n)}}\, \mathrm{d}s \\
&  = \int_0^1\EXP{q\norm{\Theta_{n-1}+ \gamma_n(g(\Theta_{n-1})+sD_n)-\vartheta}^{q-2} \\
&\quad \cdot\abs{\lll\Theta_{n-1}+ \gamma_n(g(\Theta_{n-1})+sD_n)-\vartheta,D_n \rrr}}\, \mathrm{d}s \\
&  \leq q\int_0^1\EXP{\norm{\Theta_{n-1}+ \gamma_n(g(\Theta_{n-1})+sD_n)-\vartheta}^{q-1} \norm{D_n}} \, \mathrm{d}s \\
&  \leq q 2^{q-1}\,\int_0^1 \EXP{\big(\norm{\Theta_{n-1}+ \gamma_n g(\Theta_{n-1})-\vartheta}^{q-1} +s^{q-1} |\gamma_n|^{q-1} \norm{D_n}^{q-1} \big) \norm{D_n}} \,\mathrm{d}s.
\end{split}
\end{equation}
%
This and \eqref{prop36:new1} ensure that for all $n\in \N$ it holds that
\begin{equation}
\label{Lp_convergence:eq8b}
\begin{split}
&\int_0^1\EXP{\abs{V_q'(\Theta_{n-1}+ \gamma_n(g(\Theta_{n-1})+sD_n))(D_n)}}\, \mathrm{d}s \\
&  \leq q 2^{q-1}\,\int_0^1 \EXPP{\big(\big[1+\tfrac{\gamma_n}{c}\big]^{q-1}\norm{\Theta_{n-1}-\vartheta}^{q-1} +s^{q-1} |\gamma_n|^{q-1} \norm{D_n}^{q-1} \big) \norm{D_n}} \,\mathrm{d}s\\
& \leq q2^{q-1} \Big(\big[1+\tfrac{\gamma_n}{c}\big]^{q-1}
\EXP{\norm{\Theta_{n-1}-\vartheta}^{q-1} \norm{D_n}} + |\gamma_n|^{q-1} \EXP{\norm{D_n}^q}
\Big).
\end{split}
\end{equation}
The tower property for conditional expectations, \eqref{Lp_convergence:eq3}, \eqref{Lp_convergence:eq4}, and Jensen's inequality hence demonstrate that for all $n \in \N$ it holds that
\begin{equation}
\label{Lp_convergence:eq8}
\begin{split}
&\int_0^1\EXP{\abs{V_q'(\Theta_{n-1}+ \gamma_n(g(\Theta_{n-1})+sD_n))(D_n)}}\, \mathrm{d}s \\
& \leq q2^{q-1} \big[1+\tfrac{\gamma_n}{c}\big]^{q-1}
\EXPP{\norm{\Theta_{n-1}-\vartheta}^{q-1} \EXP{\norm{D_n} \!\mid\! \mathbb{F}_{n-1}}} + q2^{q-1} |\gamma_n|^{q-1} \EXP{\norm{D_n}^q}\\
& \leq \kappa q2^{q-1} \big[1+\tfrac{\gamma_n}{c}\big]^{q-1}
\EXPP{\norm{\Theta_{n-1}-\vartheta}^{q-1} \big(1 + \norm{\Theta_{n-1}-\vartheta}\big)} + q2^{q-1} |\gamma_n|^{q-1} \EXP{ \norm{D_n}^q}\\
& \leq q2^{q-1}\big[1+\tfrac{\gamma_n}{c}\big]^{q-1} \max\{\kappa,(\gamma_n)^{q-1}\} \, \EXP{\norm{\Theta_{n-1}-\vartheta}^{q-1} + \norm{\Theta_{n-1}-\vartheta}^{q} +\norm{D_n}^{q}}\\
&<\infty.
\end{split}
\end{equation}
%
Next observe that Jensen's inequality, the hypothesis that for all $n\in\N$, $A\in\mathbb F_{n-1}$ with $\E[\norm{D_n}]<\infty$
it holds that $\E[D_n\mathbbm{1}_A]=0$, and \eqref{Lp_convergence:eq3} ensure that for all $n\in\N$, $A\in\mathbb F_{n-1}$ it holds that
\begin{equation}
\E[\norm{D_n}]<\infty \qandq \E[D_n\mathbbm{1}_A]=0.
\end{equation}
This, the fact that for all $n\in\N$ it holds that the function $\Theta_{n-1}$ is 
$\mathbb F_{n-1}/\mathcal B(\R^d)$-measurable, \eqref{Lp_convergence:eq3}, \eqref{prop36:new5}, \eqref{prop36:new2}, \eqref{Lp_convergence:eqV'}, \eqref{prop36:new6},
and Lemma~\ref{Orthogonality}
assure that for all $n\in\N$ it holds that
\begin{equation}
\label{prop36:new7}
\begin{split}
&\E[V_q'(\Theta_{n-1}+\gamma_ng(\Theta_{n-1}))(D_n)]\\
&=q\,\E\!\left[\bigl\langle\norm{\Theta_{n-1}+\gamma_ng(\Theta_{n-1})-\vartheta}^{q-2}(\Theta_{n-1}+\gamma_ng(\Theta_{n-1})-\vartheta),D_n\bigr\rangle\right]=0.
\end{split}
\end{equation}
In addition, note that \eqref{Lp_convergence:eqV'} ensures that for all  $n \in \N \cap (N,\infty)$, $t \in [0,1]$ it holds that
\begin{equation}
\begin{split}
&\EXPP{\Abs{V_q'(\Theta_{n-1}+ \gamma_n(g(\Theta_{n-1})+tD_n))(D_n) - V_q'(\Theta_{n-1}+ \gamma_n g(\Theta_{n-1}))(D_n)} } \\
&= \E\Big[q\, \Abs{ \norm{\Theta_{n-1}+ \gamma_n(g(\Theta_{n-1})+tD_n)-\vartheta}^{q-2}\lll\Theta_{n-1}+ \gamma_n(g(\Theta_{n-1})+tD_n)-\vartheta,D_n \rrr \\
	&\quad - \norm{\Theta_{n-1}+ \gamma_n g(\Theta_{n-1})-\vartheta}^{q-2}\lll\Theta_{n-1}+ \gamma_n g(\Theta_{n-1})-\vartheta,D_n \rrr } \Big] \\
&\leq  q \, \EXPP{\Abs{\norm{\Theta_{n-1}+ \gamma_n(g(\Theta_{n-1})+tD_n)-\vartheta}^{q-2} -  \norm{\Theta_{n-1}+ \gamma_n g(\Theta_{n-1})-\vartheta}^{q-2}} \\
	&\quad  \cdot \abs{\lll\Theta_{n-1}+ \gamma_n g(\Theta_{n-1})-\vartheta,D_n \rrr}} \\
	& \quad + q \, \EXPP{\norm{\Theta_{n-1}+ \gamma_n(g(\Theta_{n-1})+tD_n)-\vartheta}^{q-2} |\lll \gamma_n t D_n, D_n \rrr|}\\
	&=  q \, \EXPP{\Abs{\norm{\Theta_{n-1}+ \gamma_n(g(\Theta_{n-1})+tD_n)-\vartheta}^{q-2} -  \norm{\Theta_{n-1}+ \gamma_n g(\Theta_{n-1})-\vartheta}^{q-2}} \\
		&\quad  \cdot \abs{\lll\Theta_{n-1}+ \gamma_n g(\Theta_{n-1})-\vartheta,D_n \rrr}} \\
	& \quad + q  \gamma_n t \,\EXPP{\norm{\Theta_{n-1}+ \gamma_n(g(\Theta_{n-1})+tD_n)-\vartheta}^{q-2} \norm{D_n}^2}.
\end{split}
\end{equation}
Lemma \ref{weak_triangle} (with $p = q-2$ in the notation of Lemma \ref{weak_triangle}), Lemma \ref{weak_reverse_triangle2} (with $p = q-2$ in the notation of Lemma \ref{weak_reverse_triangle2}), and the Cauchy-Schwarz inequality  hence demonstrate that for all  $n \in \N \cap (N,\infty)$, $t \in [0,1]$ it holds that
\begin{equation}
\label{Lp_convergence:eq9b}
\begin{split}
&\EXPP{\Abs{V_q'(\Theta_{n-1}+ \gamma_n(g(\Theta_{n-1})+tD_n))(D_n) - V_q'(\Theta_{n-1}+ \gamma_n g(\Theta_{n-1}))(D_n)} } \\
&\leq  q\,\EXPP{2^{q-2} \norm{\gamma_n t D_n}\big(\norm{\Theta_{n-1}+ \gamma_n g(\Theta_{n-1})-\vartheta}^{q-3} + \norm{\gamma_n tD_n}^{q-3} \big) \\
	& \quad  \cdot \norm{\Theta_{n-1}+ \gamma_n g(\Theta_{n-1})-\vartheta} \norm{D_n}} \\
	& \quad + q \gamma_n t \,\EXPP{2^{q-2}\big(\norm{\Theta_{n-1}+ \gamma_n g(\Theta_{n-1})-\vartheta}^{q-2} + \norm{\gamma_n tD_n}^{q-2} \big) \norm{D_n}^2}\\
&=  q2^{q-2}\,\EXPP{\gamma_n t \norm{D_n}^2\big(\norm{\Theta_{n-1}+ \gamma_n g(\Theta_{n-1})-\vartheta}^{q-3} + (\gamma_n t)^{q-3}\norm{D_n}^{q-3} \big) \\
& \quad  \cdot \norm{\Theta_{n-1}+ \gamma_n g(\Theta_{n-1})-\vartheta}  \\
& \quad + \gamma_n t\big( \norm{\Theta_{n-1}+ \gamma_n g(\Theta_{n-1})-\vartheta}^{q-2} + (\gamma_n t)^{q-2}\norm{D_n}^{q-2} \big) \norm{D_n}^2}.
\end{split}
\end{equation}
In addition, observe that \eqref{Lp_convergence:eq1} and \eqref{Lp_convergence:eq2} ensure that for all $r \in [0,c]$, $\theta \in \R^d$ it holds that 
$
\norm{\theta + r g(\theta) - \vartheta} \leq \norm{\theta-\vartheta}.
$
This, the fact that $\sup_{n \in \N \cap(N,\infty)} \gamma_n \leq c$, and \eqref{Lp_convergence:eq9b} assure that for all  $n \in  \N \cap (N,\infty)$, $t \in [0,1]$ it holds that
\begin{equation}
\begin{split}
&\EXPP{\Abs{V_q'(\Theta_{n-1}+ \gamma_n(g(\Theta_{n-1})+tD_n))(D_n) - V_q'(\Theta_{n-1}+ \gamma_n g(\Theta_{n-1}))(D_n)} } \\
&\leq  q2^{q-2} \gamma_n\,\EXPP{ t \norm{D_n}^2\big(\norm{\Theta_{n-1}-\vartheta}^{q-3} + (\gamma_n t)^{q-3}\norm{D_n}^{q-3} \big) \norm{\Theta_{n-1}-\vartheta}  \\
	& \quad + t\big( \norm{\Theta_{n-1}-\vartheta}^{q-2} + (\gamma_n t)^{q-2}\norm{D_n}^{q-2} \big) \norm{D_n}^2}\\
&=  q2^{q-2} \gamma_n t\,\EXPP{ \norm{D_n}^2\norm{\Theta_{n-1}-\vartheta}^{q-2} + (\gamma_n t)^{q-3}\norm{D_n}^{q-1}\norm{\Theta_{n-1}-\vartheta}  \\
	& \quad +  \norm{D_n}^2 \norm{\Theta_{n-1}-\vartheta}^{q-2} + (\gamma_n t)^{q-2}\norm{D_n}^{q}}\\
&\leq q2^{q-1} \gamma_n \,\EXPP{\norm{D_n}^2 \norm{\Theta_{n-1}-\vartheta}^{q-2} +
(\gamma_n)^{q-3}\norm{D_n}^{q-1}\norm{\Theta_{n-1}-\vartheta} + (\gamma_n)^{q-2}\norm{D_n}^{q}}.
\end{split}
\end{equation}
The tower property for conditional expectations and (\ref{Lp_convergence:eq4}) hence imply that for all $n \in \N \cap (N,\infty)$, $t \in [0,1]$ it holds that
\begin{equation}
\begin{split}
&\EXPP{\Abs{V_q'(\Theta_{n-1}+ \gamma_n(g(\Theta_{n-1})+tD_n))(D_n) - V_q'(\Theta_{n-1}+ \gamma_n g(\Theta_{n-1}))(D_n)} } \\
&\leq q2^{q-1} \gamma_n \,\EXPP{\EXP{\norm{D_n}^2 \!\mid\! \mathbb{F}_{n-1}} \norm{\Theta_{n-1}-\vartheta}^{q-2} \\
&\quad +(\gamma_n)^{q-3}\,\EXP{\norm{D_n}^{q-1} \!\mid\! \mathbb{F}_{n-1}} \norm{\Theta_{n-1}-\vartheta} + (\gamma_n)^{q-2}\,\EXP{\norm{D_n}^q \!\mid\! \mathbb{F}_{n-1}}} \\
&\leq q2^{q-1} \gamma_n \,\EXPP{\kappa\big(1+\norm{\Theta_{n-1}-\vartheta}^2\big)\norm{\Theta_{n-1}-\vartheta}^{q-2} \\
& \quad +(\gamma_n)^{q-3}\kappa\big(1+\norm{\Theta_{n-1}-\vartheta}^{q-1}\big) \norm{\Theta_{n-1}-\vartheta}+ (\gamma_n)^{q-2}\kappa\big(1+\norm{\Theta_{n-1}-\vartheta}^q\big)}\\
&= q2^{q-1}\kappa \gamma_n \Big( \EXP{\norm{\Theta_{n-1}-\vartheta}^{q-2}} + (\gamma_n)^{q-3} \,\EXP{\norm{\Theta_{n-1}-\vartheta}} + (\gamma_n)^{q-2} \\
&\quad + \big(1 + (\gamma_n)^{q-3} +  (\gamma_n)^{q-2}\big)\, \EXP{\norm{\Theta_{n-1}-\vartheta}^{q}}\Big).
\end{split}
\end{equation}
The fact that $\forall\, x \in [0,\infty) \colon x \leq 1 + x^q$ 
and the induction hypothesis (see \eqref{Lp_convergence:eq5}) therefore ensure that for all $n \in \N \cap (N,\infty)$, $t \in [0,1]$ it holds that
\begin{equation}
\begin{split}
&\EXPP{\Abs{V_q'(\Theta_{n-1}+ \gamma_n(g(\Theta_{n-1})+tD_n))(D_n) - V_q'(\Theta_{n-1}+ \gamma_n g(\Theta_{n-1}))(D_n)} } \\
&\leq q2^{q-1}\kappa \gamma_n \Big( C(\gamma_n)^{\nicefrac{(q-2)}{2}} + (\gamma_n)^{q-3} \,\EXP{1+\norm{\Theta_{n-1}-\vartheta}^q} + (\gamma_n)^{q-2} \\
&\quad + \big(1 + (\gamma_n)^{q-3} +  (\gamma_n)^{q-2}\big)\, \EXP{\norm{\Theta_{n-1}-\vartheta}^{q}}\Big)\\
&= q2^{q-1}\kappa \gamma_n \Big(  C(\gamma_n)^{\nicefrac{(q-2)}{2}} + (\gamma_n)^{q-3} + (\gamma_n)^{q-2}  \\
&\quad +\big(1+ 2(\gamma_n)^{q-3}+  (\gamma_n)^{q-2}\big)\, \EXP{\norm{\Theta_{n-1}-\vartheta}^{q}}\Big).
\end{split}
\end{equation}
The fact that $\sup_{n \in \N \cap (N,\infty)}\gamma_n \leq c \leq 1$ and the fact that $\nicefrac{(q-2)}{2} \leq q-3$ hence demonstrate that for all $n \in \N \cap (N,\infty)$, $t \in [0,1]$ it holds that
\begin{equation}
\begin{split}
&\EXPP{\Abs{V_q'(\Theta_{n-1}+ \gamma_n(g(\Theta_{n-1})+tD_n))(D_n) - V_q'(\Theta_{n-1}+ \gamma_n g(\Theta_{n-1}))(D_n)} } \\
&\leq q2^{q-1} \kappa\gamma_n \big( (C + 2) (\gamma_n)^{\nicefrac{(q-2)}{2}}
 + 4 \,\EXP{\norm{\Theta_{n-1}-\vartheta}^{q}} \big) \\
 &= q2^{q-1} \kappa \big( (C + 2) (\gamma_n)^{\nicefrac{q}{2}}
 + 4 \gamma_n \,\EXP{\norm{\Theta_{n-1}-\vartheta}^{q}} \big) \\
  &\leq q2^{q-1} \kappa \max\{C+2,4\} \big((\gamma_n)^{\nicefrac{q}{2}}
 +  \EXP{\norm{\Theta_{n-1}-\vartheta}^{q}} \big) \\
   &\leq q2^{q+1} \kappa \max\{C,1\} \big((\gamma_n)^{\nicefrac{q}{2}}
 +  \EXP{\norm{\Theta_{n-1}-\vartheta}^{q}} \big) \\
& = q2^{q+1}\kappa\max\{C,1\}\big((\gamma_n)^{\nicefrac{q}{2}} + \EXP{V_q(\Theta_{n-1})}\big).
\end{split}
\end{equation}
Combining this, (\ref{Lp_convergence:assumption2}), (\ref{Lp_convergence:eq1}), (\ref{Lp_convergence:eq2}), (\ref{Lp_convergence:eq6}), (\ref{Lp_convergence:eqV'}), \eqref{prop36:new6}, (\ref{Lp_convergence:eq8}), and (\ref{prop36:new7}) with Corollary \ref{Lyapunov_convergence_nonexplicit} (with $N = N$, $k = \nicefrac{q}{2}$, $\kappa = q2^{q+1}\kappa\max\{1,C\}$, $c = c$, $\varrho= c$, $V= V_q$ in the notation of Corollary \ref{Lyapunov_convergence_nonexplicit}) yields that there exists $\mathfrak{C} \in (0,\infty)$ such that for all $n \in \N_0$ it holds that
\begin{equation}
\label{212intro}
\EXP{\norm{\Theta_n-\vartheta}^{q}} = \EXP{V_q(\Theta_n)} \leq \mathfrak{C} (\gamma_n)^{\nicefrac{q}{2}}. 
\end{equation}
Induction thus proves (\ref{Lp_convergence:eq10}).
Next note that (\ref{Lp_convergence:eq10}) demonstrates that for all $q \in \{2,4,6,\ldots\}\cap[0,p]$ there exists
$C \in (0,\infty)$ such that for all $n \in \N_0$ it holds that
\begin{equation}
\big(\EXP{\norm{\Theta_n-\vartheta}^q}\big)^{\nicefrac{1}{q}} \leq \big(C (\gamma_n)^{\nicefrac{q}{2}} \big)^{\nicefrac{1}{q}} = C^{\nicefrac{1}{q}}\,(\gamma_n)^{\nicefrac{1}{2}}.
\end{equation}
This completes the proof of Proposition \ref{Lp_convergence}.
\end{proof}

\begin{theorem}
\label{Lp_theorem}
Let $d \in \N$, $p \in \{2,4,6,\ldots \}$, $\kappa,c \in (0,\infty)$, $(\gamma_n)_{n \in \N} \subseteq (0,\infty)$, $\vartheta \in \R^d$, let 
$\lll \cdot,\cdot \rrr \colon \R^d \times \R^d \to \R$
 be a scalar product, let 
 $\left \| \cdot \right \| \! \colon \R^d \to [0,\infty)$ 
 be the function which satisfies for all $\theta \in \R^d$ that $\norm{\theta} = \sqrt{\lll \theta,\theta \rrr}$, 
let $g \colon \R^d \to \R^d$ be $\mathcal{B}(\R^d)/ \mathcal{B}(\R^d) $-measurable, let $(\Omega , \mathcal{F}, \P , ( \mathbb{F}_n )_{ n \in \N_0 })$ be a filtered probability space, let $D \colon  \N \times \Omega \to \R^d$ be an $ ( \mathbb{F}_n )_{n \in \N } / \mathcal{B}(\R^d) $-adapted stochastic process which satisfies for all $n \in \N$, $A \in \mathbb{F}_{n-1}$ with $\Exp{\norm{D_n}} < \infty$ that $\Exp{D_n \mathbbm{1}_A} = 0$, let $\Theta  \colon  \N_0 \times \Omega \to \R^d$ be a function, assume that  $\Theta_0$ is $\mathbb{F}_0/ \mathcal{B}(\R^d)$-measurable, and assume for all $n \in \N$, $A \in \mathbb{F}_{n-1}$, $\theta \in \R^d$ that 
\begin{equation}
\label{Lp_theorem:assumption1}
\lll \theta - \vartheta, g(\theta)  \rrr   \leq - c\max\!\big\{ \norm{\theta - \vartheta}^2, \norm{g(\theta)}^2 \big\}, 
\end{equation}
\begin{equation}
\label{Lp_theorem:assumption2}
\quad \limsup_{l\to\infty} \gamma_l = 0 < \min_{k \in \{1,2,\ldots,\nicefrac{p}{2}\}}\!\left(\liminf_{l \to \infty}\left[\tfrac{(\gamma_l)^{k}-(\gamma_{l-1})^{k}}{(\gamma_l)^{k+1}} +\tfrac{c (\gamma_{l-1})^k}{2 (\gamma_{l})^k} \right] \right),
\end{equation}
\begin{equation}
\label{Lp_theorem:assumption3}
\Theta_n = \Theta_{n-1} + \gamma_n (g(\Theta_{n-1})+ D_n) , \qquad \EXP{\norm{\Theta_0}^p} < \infty,
\end{equation}
\begin{equation}
\label{Lp_theorem:assumption4}
\andq \EXP{\norm{D_n}^p \ind{A}} \leq \kappa\,\EXP{(1 +\norm{\Theta_{n-1}}^p)\ind{A}}.
\end{equation}
Then there exists $C \in (0,\infty)$ such that for all $n \in \N$ it holds that
\begin{equation}
\label{Lp_theorem:conclusion}
\{\theta \in \R^d \colon g(\theta)=0 \} = \{\vartheta\} \qandq \big(\EXP{\norm{\Theta_n-\vartheta}^p}\big)^{\nicefrac{1}{p}} \leq C (\gamma_n)^{\nicefrac{1}{2}}.
\end{equation}
\end{theorem}
\begin{proof}[Proof of Theorem \ref{Lp_theorem}]
Observe that  Lemma~\ref{weak_triangle} and \eqref{Lp_theorem:assumption4} ensure that for all $n \in \N$, $A\in \mathbb{F}_{n-1}$ it holds that
\begin{equation}
\begin{split}
\EXP{\norm{D_n}^p \ind{A}} &\leq \kappa\,\EXP{(1 +\norm{\Theta_{n-1}}^p)\ind{A}}\\
& =\kappa\,\EXP{(1 +\norm{\Theta_{n-1}-\vartheta + \vartheta}^p)\ind{A}}\\
& \leq \kappa\,\EXP{(1 +2^p\norm{\Theta_{n-1}-\vartheta} + 2^p \norm{\vartheta}^p)\ind{A}}\\
& \leq \kappa\,\EXP{(\max\{1+ 2^p \norm{\vartheta}^p, 2^p\} +2^p\norm{\Theta_{n-1}-\vartheta}^p)\ind{A}}\\
& \leq \kappa\,\max\{1+ 2^p \norm{\vartheta}^p, 2^p\}\,\EXP{(1+\norm{\Theta_{n-1}-\vartheta}^p)\ind{A}}.
\end{split}
\end{equation}	
Combining item (\ref{fcond:item1}) in Lemma \ref{fcond} and Proposition \ref{Lp_convergence} hence establishes \eqref{Lp_theorem:conclusion}.
The proof of Theorem \ref{Lp_theorem} is thus completed.
\end{proof}

\begin{remark}[A comment on assumption (\ref{Lp_theorem:assumption4})]
\label{comment_condexp}
Let $d \in \N$, $p \in \{2,4,6,\ldots \}$, $\kappa \in (0,\infty)$, 
let $D \colon  \N \times \Omega \to \R^d$ be an $ ( \mathbb{F}_n )_{n \in \N} / \mathcal{B}(\R^d) $-adapted stochastic process, and let $\Theta  \colon  \N_0 \times \Omega \to \R^d$ be an $ ( \mathbb{F}_n )_{n \in \N_0 } / \mathcal{B}(\R^d) $-adapted stochastic process. Then the following two statements are equivalent:
\begin{enumerate}[(i)]
\item \label{comment_condexp:item1}
For all $n\in \N$, $A \in \F_{n-1}$ it holds that
\begin{equation}
\EXP{\norm{D_n}^p \ind{A}} \leq \kappa\,\EXP{(1 +\norm{\Theta_{n-1}}^p)\ind{A}}.
\end{equation}
\item \label{comment_condexp:item2}
For all $n \in \N$ it holds $\P$-a.s.\ that
\begin{equation}
\EXP{\norm{D_n}^p \!\mid\! \F_{n-1}} \leq \kappa\big(1 +\norm{\Theta_{n-1}}^p\big).
\end{equation}
\end{enumerate}
\end{remark}

\section{Applications}
\label{section:applications}
In this section we present several consequences of Theorem~\ref{Lp_theorem} above. In particular, we prove in this section for every arbitrarily small $\varepsilon \in (0,\infty)$ and every arbitrarily large $p\in (0,\infty)$ that SGD optimization algorithms converge in the strong $L^p$-sense with order $\nicefrac{1}{2}-\varepsilon$ to the global minimum  of the objective function of a suitable stochastic optimization problem.
\subsection[Strong $L^p$-convergence for a specific type of SAAs]{Strong $L^p$-convergence rate for a specific type of SAAs}
\label{section:applications1}
In order to apply Theorem~\ref{Lp_theorem} to a SGD optimization algorithm we need to verify that the sequence $\gamma_n \in (0,\infty)$, $n \in \N$, of learning rates of the considered SGD optimization algorithm satisfies the hypothesis in \eqref{Lp_theorem:assumption2} in Theorem~\ref{Lp_theorem}. For this we employ the next result, Lemma~\ref{example_learningrate} below, which, in particular, provides explicit examples of sequences that satisfy the hypothesis in \eqref{Lp_theorem:assumption2} in Theorem~\ref{Lp_theorem}.
\begin{lemma}[Example of suitable learning rates]
\label{example_learningrate}
Let $k, \alpha, c\in (0,\infty)$, $\nu \in \R\setminus\{1\}$, $(\gamma_n)_{n \in \N} \subseteq (0,\infty)$ satisfy for all $n \in \N$ that $\gamma_n = \alpha n^{-\nu}$.
Then the following two statements are equivalent:
\begin{enumerate}
	\item[(i)] It holds that $\nu \in (0,1)$.
	\item[(ii)] It holds that 
\begin{equation}
\label{example_learningrate:eq1}
\limsup_{n \to \infty} \gamma_n  = 0  <\liminf_{n \to \infty}\left[\tfrac{(\gamma_n)^{k}-(\gamma_{n-1})^{k}}{(\gamma_n)^{k+1}} + \tfrac{c(\gamma_{n-1})^k}{(\gamma_{n})^k} \right].
\end{equation}
\end{enumerate}
\end{lemma}

\begin{proof}[Proof of Lemma \ref{example_learningrate}]
	Throughout this proof let $(\Gamma_{n,r})_{n\in \N, r \in \R}\subseteq (0,\infty)$ satisfy for all $n \in \N$, $r\in \R$ that $\Gamma_{n,r}= \alpha n^{-r}$.
Note that for all $r \in (0,\infty)$ it holds that
\begin{equation}
\label{example_learningrate:eqa}
\limsup_{n\to \infty} \Gamma_{n,r} =\limsup_{n\to \infty} \left[ \frac{\alpha}{n^r}\right]=0.
\end{equation}
 Moreover, observe that Lemma \ref{decay_fraction} (with $\beta = (k+1)r$, $\delta = kr$ for $r \in (0,1)$ in the notation of Lemma \ref{decay_fraction}) and the fact that for all $r \in (0,1)$ it holds that $(k+1)r= kr +r < kr + 1$ prove that for all $r\in(0,1)$ it holds that
\begin{equation}
\label{example_learningrate:eq2}
\begin{split}
\liminf_{n \to \infty}\left[\frac{(\Gamma_{n,r})^{k}-(\Gamma_{n-1,r})^{k}}{(\Gamma_{n,r})^{k+1}}\right]
&=\liminf_{n \to \infty} \left[
\frac{\big(\frac{\alpha}{n^r}\big)^k-\big(\frac{\alpha}{(n-1)^r}\big)^k}{\big(\frac{\alpha}{n^r}\big)^{k+1}}
\right]\\
& = 
\frac{1}{\alpha} \liminf_{n \to \infty}\left[\frac{n^{-kr}-(n-1)^{-kr}}{n^{-(k+1)r}}\right]\\
&= 0.
\end{split}
\end{equation}
In addition, note that for all $r \in (0,\infty)$ it holds that
\begin{equation}
\begin{split}
\liminf_{n \to \infty} \left[\frac{(\Gamma_{n-1,r})^k}{(\Gamma_{n,r})^k}\right]
&=\liminf_{n \to \infty} \left[  \frac{\big(\frac{\alpha}{(n-1)^r}\big)^k}{\big(\frac{\alpha}{n^r}\big)^k}\right]\\
 &= \liminf_{n \to \infty} \left[\frac{n^{kr}}{(n-1)^{kr}}\right] \\
 &= \liminf_{n \to \infty} \left[\frac{(n+1)^{kr}}{n^{kr}}\right] \\
 &= \liminf_{n \to \infty}  \left[(1+\nicefrac{1}{n})^{kr}\right] = 1 .
\end{split}
\end{equation}
Therefore, we obtain that for all $r \in (0,1)$ it holds that
\begin{equation}
\label{example_learningrate:eqb}
\begin{split}
&\liminf_{n \to \infty}  \left[\frac{(\Gamma_{n,r})^{k}-(\Gamma_{n-1,r})^{k}}{(\Gamma_{n,r})^{k+1}} + \frac{c(\Gamma_{n-1,r})^k}{(\Gamma_{n,r})^k} \right]  \\
&\geq \liminf_{n \to \infty}\left[\frac{(\Gamma_{n,r})^{k}-(\Gamma_{n-1,r})^{k}}{(\Gamma_{n,r})^{k+1}}\right] + \liminf_{n \to \infty} \left[\frac{c(\Gamma_{n-1,r})^k}{(\Gamma_{n,r})^k}\right] = c > 0.
\end{split}
\end{equation}
Next observe that for all $r \in (-\infty,0)$ it holds that
\begin{equation}
\label{example_learningrate:eqc}
\begin{split}
\limsup_{n\to \infty} \Gamma_{n,r} 
=\limsup_{n\to \infty} \left[\alpha n^{|r|}\right]
=\infty.
\end{split}
\end{equation}
Moreover, note that for all $r \in (1,\infty)$, $n\in\{2,3,\dots\}$ it holds that
\begin{equation}
\begin{split}
\frac{(\Gamma_{n,r})^{k}-(\Gamma_{n-1,r})^{k}}{(\Gamma_{n,r})^{k+1}}  
&=\frac{1}{\alpha} \left[\frac{n^{-kr}-(n-1)^{-kr}}{n^{-(k+1)r}}\right]\\
&= \frac{n^{(k+1)r}}{\alpha}\left[\frac{1}{n^{kr}}-\frac{1}{(n-1)^{kr}}\right]\\
& =  -\frac{(kr)n^{(k+1)r}}{\alpha}\int_{n-1}^n\frac{1}{x^{kr+1}}\, \mathrm{d}x\\
&\leq -\frac{(kr)n^{(k+1)r}}{\alpha n^{kr+1}}\\
&=-\frac{(kr)n^{r-1}}{\alpha}.
\end{split}
\end{equation}
This  assures that for all $r\in (1,\infty)$ it holds that
\begin{equation}
\begin{split}
&\liminf_{n \to \infty}  \left[\frac{(\Gamma_{n,r})^{k}-(\Gamma_{n-1,r})^{k}}{(\Gamma_{n,r})^{k+1}} + \frac{c(\Gamma_{n-1,r})^k}{(\Gamma_{n,r})^k} \right]\\
&\leq 
\liminf_{n \to \infty} \left[-\frac{kr n^{r-1}}{\alpha}+\frac{c \big(\frac{\alpha}{(n-1)^r}\big)^k}{\big(\frac{\alpha}{n^r}\big)^k} \right]\\
&=
\liminf_{n \to \infty} \left[-\frac{kr n^{r-1}}{\alpha}+\frac{cn^{rk}}{(n-1)^{rk}} \right]\\
&=
\liminf_{n \to \infty} \left[-\frac{kr (n+1)^{r-1}}{\alpha}+\frac{c(n+1)^{rk}}{n^{rk}} \right]\\
&= 
\liminf_{n \to \infty} \left[-\frac{kr (n+1)^{r-1}}{\alpha}+c \left[1 +\nicefrac{1}{n}\right]^{rk} \right]\\
&= c + \liminf_{n \to \infty} \left[-\frac{kr n^{r-1}}{\alpha}\right]=-\infty.
\end{split}
\end{equation}
Combining this, \eqref{example_learningrate:eqa}, \eqref{example_learningrate:eqb}, and \eqref{example_learningrate:eqc} completes the proof of Lemma \ref{example_learningrate}.
\end{proof}
\begin{lemma}
\label{le:FhB}
Let $(\Omega, \mathcal{F},\mu)$ be a sigma-finite measure space, let $(\mathbb{X},\mathcal{X})$ and $(\mathbb{Y},\mathcal{Y})$ be measurable spaces, let $Y\colon \Omega \to \mathbb{Y}$ be $\mathcal{F}/\mathcal{Y}$ measurable, let $d \in \N$, let $G \colon \mathbb{X} \times \mathbb{Y} \to \R^d$ be $(\mathcal{X} \otimes \mathcal{Y})/\mathcal{B}(\R^d)$-measurable, let $\left \| \cdot \right \| \! \colon \R^d \to [0,\infty)$ be a norm, and assume  for all $x\in \mathbb{X}$ that $\int_\Omega \norm{G(x,Y(\omega))}\, \mu(\mathrm{d}\omega)<\infty$. Then it holds that the function 
\begin{equation}
\mathbb{X} \ni x \mapsto \int_{\Omega} G(x,Y(\omega)) \, \mu(\mathrm{d}\omega) \in \R^d
\end{equation}
is $\mathcal{X}/\mathcal{B}(\R^d)$-measurable.
\end{lemma}
\begin{proof}
Throughout this proof assume w.l.o.g.\ that $\mathbb{X} \neq \emptyset$, let $G_i \colon \mathbb{X} \times \mathbb{Y} \to \R$, $i \in \{1,2,\dots,d\}$, be the functions which satisfy for all $x \in \mathbb{X}$, $y\in \mathbb{Y}$ that 
\begin{equation}
\label{NR1}
G(x,y)=(G_1(x,y), G_2(x,y), \dots, G_d(x,y)),
\end{equation}
let $c\in (0,\infty)$ satisfy
\begin{equation}
\label{NR2}
c= \sup_{\theta=(\theta_1,\dots, \theta_d) \in \R^d\setminus\{0\}} \left( \frac{\big(\sum_{i=1}^d |\theta_i|\big)}{\norm{\theta}}\right),
\end{equation}
let $v \in \mathbb{X}$, and let $\nu\colon \mathcal{X} \to [0,\infty]
$ be the measure which satisfies for all $A \in \mathcal{X}$ that 
\begin{equation}
\nu(A)=
\begin{cases}
1 & \colon v \in A\\
0 & \colon  v \in \mathbb{X} \setminus A.
\end{cases}
\end{equation}
Observe that  \eqref{NR1}, \eqref{NR2}, and the hypothesis that for all $x \in \mathbb{X}$ it holds that $\int_\Omega \norm{G(x,Y(\omega))}\,\mu(\mathrm{d}\omega)<\infty$ ensure that for all $i \in \{1,2,\dots,d\}$, $x\in\mathbb{X}$ it holds that
\begin{equation}
\label{A}
\begin{split}
&\int_\Omega \max\!\big\{G_i(x,Y(\omega)),0\big\} \,\mu(\mathrm{d}\omega)
+ \int_\Omega \max\!\big\{\!-G_i(x,Y(\omega)),0\big\} \,\mu(\mathrm{d}\omega)\\
&=\int_\Omega \big|G_i(x,Y(\omega))\big| \,\mu(\mathrm{d}\omega)\\
&\leq \int_\Omega \bigg[\textstyle\sum\limits_{j=1}^d \big|G_j(x,Y(\omega))\big|\bigg] \,\mu(\mathrm{d}\omega)\\
& \leq c \int_\Omega \norm{G(x,Y(\omega))} \,\mu(\mathrm{d}\omega) <\infty.
\end{split}
\end{equation}
Hence, we obtain that for all $i \in \{1,2,\dots,d\}$, $x\in\mathbb{X}$ it holds that
\begin{equation}
\label{K}
\begin{split}
&\int_\Omega G_i(x,Y(\omega)) \,\mu(\mathrm{d}\omega)\\
&= \int_\Omega \Big[\max\!\big\{G_i(x,Y(\omega)),0\big\} + \min\!\big\{G_i(x,Y(\omega)),0\big\} \Big] \,\mu(\mathrm{d}\omega)\\
&= \int_\Omega \Big[\max\!\big\{G_i(x,Y(\omega)),0\big\} - \max\!\big\{\!-G_i(x,Y(\omega)),0\big\} \Big] \,\mu(\mathrm{d}\omega)\\
&= \int_\Omega \max\!\big\{G_i(x,Y(\omega)),0\big\} \,\mu(\mathrm{d}\omega)
- \int_\Omega \max\!\big\{\!-G_i(x,Y(\omega)),0\big\} \,\mu(\mathrm{d}\omega).
\end{split}
\end{equation}
Next note that Fubini's theorem and the fact that the measure $(\nu \otimes \mu) \colon (\mathcal{X}\otimes \mathcal{F}) \to [0,\infty]$ is sigma-finite prove that for all $i \in \{1,2,\dots,d\}$ it holds that the functions
\begin{equation}
\label{B}
\mathbb{X} \ni x \mapsto \int_\Omega \max\!\big\{G_i(x,Y(\omega)),0\big\} \,\mu(\mathrm{d}\omega) \in [0,\infty]
\end{equation}
and  
\begin{equation}
\label{C}
\mathbb{X} \ni x \mapsto \int_\Omega \max\!\big\{\!-G_i(x,Y(\omega)),0\big\} \,\mu(\mathrm{d}\omega) \in [0,\infty]
\end{equation}
are $\mathcal{X}/\mathcal{B}([0,\infty])$-measurable. Combining this with \eqref{A} demonstrates that for all $i \in \{1,2,\dots,d\}$ it holds that the functions
\begin{equation}
\mathbb{X} \ni x \mapsto \int_\Omega \max\!\big\{G_i(x,Y(\omega)),0\big\} \,\mu(\mathrm{d}\omega) \in [0,\infty)
\end{equation}
and  
\begin{equation}
\mathbb{X} \ni x \mapsto \int_\Omega \max\!\big\{\!-G_i(x,Y(\omega)),0\big\} \,\mu(\mathrm{d}\omega) \in [0,\infty)
\end{equation}
are $\mathcal{X}/\mathcal{B}([0,\infty))$-measurable. This and \eqref{K} ensure that for all $i \in \{1,2,\dots,d\}$ it holds that the function $G_i\colon \mathbb{X} \to \R$ is $\mathcal{X}/\mathcal{B}(\R)$-measurable. The proof of Lemma~\ref{le:FhB} is thus completed.
\end{proof}
\begin{prop}
\label{convergence_of_SA}
Let $d \in \N$, $p \in \{2,4,6,\ldots \}$, $\alpha,\kappa,c \in (0,\infty)$, $ \nu \in (0,1)$, $\xi, \vartheta \in \R^d$, let 
$\lll \cdot,\cdot \rrr \colon \R^d \times \R^d \to \R$
 be a scalar product, let 
 $\left \| \cdot \right \| \! \colon \R^d \to [0,\infty)$ 
 be the function which satisfies for all $\theta \in \R^d$ that $\norm{\theta} = \sqrt{\lll \theta,\theta \rrr}$, let $ ( \Omega , \mathcal{F}, \P) $ be a probability space, let $(S, \mathcal{S})$ be a measurable space, let $X_{n} \colon \Omega \to S$, $n \in \N$, be i.i.d.\ random variables,
let $G \colon \R^d \times S \to \R^d$ be $(\mathcal{B}(\R^{d}) \otimes \mathcal{S}) / \mathcal{B}(\R^{d})$-measurable, let $g \colon \R^d \to \R^d$ be a function, assume for all $\theta \in \R^d$ that
\begin{equation}
\label{convergence_of_SA:assumption1}
\EXP{\norm{G(\theta, X_1) - g(\theta)}^p } \leq \kappa \big(1+ \norm{\theta}^p\big), \qquad g(\theta) = \EXP{G(\theta, X_1)}, 
\end{equation}
\begin{equation}
\label{convergence_of_SA:assumption2}
\text{and} \qquad \lll \theta - \vartheta, g(\theta)  \rrr   \leq - c\max\!\big\{\norm{\theta - \vartheta}^2, \norm{g(\theta)}^2 \big\},  
\end{equation}
and let $\Theta \colon \N_0 \times \Omega \to \R^d$ be the stochastic process which satisfies for all $n \in \N$ that 
\begin{equation}
\label{convergence_of_SA:assumption3}
\Theta_0 = \xi   \qquad \text{ and } 
\qquad \Theta_n = \Theta_{n-1} + \tfrac{\alpha}{n^\nu} G(\Theta_{n-1},X_n).
\end{equation}
Then there exists $C \in (0,\infty)$ such that for all $n \in \N$ it holds that
\begin{equation}
\label{convergence_of_SA:conclusion}
\{\theta \in \R^d \colon g(\theta)=0 \} = \{\vartheta\} \qandq \left(\EXP{ \norm{\Theta_n-\vartheta}^p }\right)^{\nicefrac{1}{p}} \leq C n^{-\nicefrac{\nu}{2}}.
\end{equation}
\end{prop}

\begin{proof}[Proof of Proposition \ref{convergence_of_SA}]
Throughout this proof let $(\gamma_n)_{n \in \N} \subseteq (0,\infty)$ satisfy for all $n \in \N$ that $\gamma_n = \alpha n^{-\nu}$, let $D \colon \N \times \Omega \to \R^d$ be the function which satisfies for all $n\in \N$ that 
\begin{equation}
\label{convergence_of_SA:eq0}
D_n = G(\Theta_{n-1},X_n) - g(\Theta_{n-1}),
\end{equation} 
let $\mathfrak{G} \colon \R^d \to [0,\infty)$ be the function which satisfies for all $\theta \in \R^d$ that
\begin{equation}
\mathfrak{G}(\theta)= \EXP{\norm{G(\theta,X_1)-g(\theta)}^p},
\end{equation} 
and let $\mathbb{F}_n \subseteq \mathcal{F}$, $n \in \N_0$, be the sigma-algebras which satisfy for all $n \in \N$ that
\begin{equation} \mathbb{F}_0 = \{ \{ \}, \Omega \} \qandq \mathbb{F}_n = \sigma_\Omega(X_1,X_2, \ldots, X_n).
\end{equation}
Observe that \eqref{convergence_of_SA:assumption1}, the hypothesis that the function $G \colon \R^d \times S \to \R^d$ is $(\mathcal{B}(\R^{d}) \otimes \mathcal{S}) / \mathcal{B}(\R^{d})$-measurable, and Lemma~\ref{le:FhB} ( with $\Omega=\Omega$, $\mathcal{F}=\mathcal{F}$, $\mu=\mathbb{P}$, $\mathbb{X}=\R^d$, $\mathcal{X}=\mathcal{B}(\R^d)$, $\mathbb{Y}=S$, $\mathcal{Y}=\mathcal{S}$, $Y=X_1$, $d=d$, $G=G$, $\left \| \cdot \right \| \! = \left \| \cdot \right \| \!$ in the notation of Lemma~\ref{le:FhB}) prove that the function $g\colon\R^d \to \R^d$ is $\mathcal{B}(\R^d)/\mathcal{B}(\R^d)$-measurable.
%
Next note that the hypothesis that the function $G \colon \R^d \times S \to \R^d$ is $(\mathcal{B}(\R^{d}) \otimes \mathcal{S}) / \mathcal{B}(\R^{d})$-measurable, the hypothesis that $\Theta_0=\xi$, and (\ref{convergence_of_SA:assumption3}) imply that $\Theta$ is an $ ( \mathbb{F}_n )_{n \in \N_0 } / \mathcal{B}(\R^d) $-adapted stochastic process.
Combining \eqref{convergence_of_SA:eq0} and the fact that the function $g\colon\R^d \to \R^d$ is $\mathcal{B}(\R^d)/\mathcal{B}(\R^d)$-measurable with 
 the hypothesis that the function  $G \colon \R^d \times S \to \R^d$ is $(\mathcal{B}(\R^{d}) \otimes \mathcal{S}) / \mathcal{B}(\R^{d})$-measurable hence demonstrates  that $D$ is an $(\mathbb{F}_n)_{n\in \N}/\mathcal{B}(\R^d)$-adapted stochastic process. Furthermore, note that (\ref{convergence_of_SA:assumption3}) proves that
\begin{equation}
\label{convergence_of_SA:eq1}
\EXP{\norm{\Theta_0}^p} = \norm{\xi}^p < \infty.
\end{equation}
In addition, observe that \eqref{convergence_of_SA:eq0} and (\ref{convergence_of_SA:assumption3}) ensure that for all $n\in \N$ it holds that
\begin{equation}
\begin{split}
\Theta_n 
&= 
\Theta_{n-1} + \tfrac{\alpha}{n^\nu} G(\Theta_{n-1},X_n)\\
&= 
\Theta_{n-1} + \tfrac{\alpha}{n^\nu} \big(g(\Theta_{n-1}) +\left[G(\Theta_{n-1},X_n)-g(\Theta_{n-1}) \right]\!\big)\\
&= 
\Theta_{n-1} + \tfrac{\alpha}{n^\nu} \big(g(\Theta_{n-1}) +D_n \big)\\
&= 
\Theta_{n-1} + \gamma_n \big(g(\Theta_{n-1}) +D_n \big).
\end{split}
\end{equation}
Next observe that (\ref{convergence_of_SA:assumption1}), item (\ref{condexp_and_independence-nonneg:item2}) in Lemma \ref{condexp_and_independence-nonneg}, the fact that for all $n \in \N$ it holds that the function $\Theta_{n-1}$ is $\mathbb{F}_{n-1}/\mathcal{B}(\R^d)$-measurable, and the fact that for all $n \in \N$ it holds that $X_n$ is independent of $\mathbb{F}_{n-1}$ ensure that for all $n \in \N$, $A \in \mathbb{F}_{n-1}$ it holds that 
\begin{equation}
\label{convergence_of_SA:eq2}
\begin{split}
\EXP{\norm{D_n}^p \ind{A}} 
&= \EXP{\norm{G(\Theta_{n-1},X_n) - g(\Theta_{n-1})}^p \ind{A}} \\
&= \EXPP{\EXP{\norm{G(\Theta_{n-1},X_n) - g(\Theta_{n-1})}^p \ind{A}\!\mid\! \F_{n-1}}}\\
&= \EXPP{\EXP{\norm{G(\Theta_{n-1},X_n) - g(\Theta_{n-1})}^p \!\mid\! \F_{n-1}}\ind{A}}\\
&= \EXP{\mathfrak{G}(\Theta_{n-1})\ind{A}}\\
&= \kappa \, \EXP{(1 + \norm{\Theta_{n-1}}^p)\ind{A} }.
\end{split}
\end{equation}
Moreover, note that 
Corollary \ref{condexp_and_independence-integrable}, 
the fact that for all $n \in \N$ it holds that the function $\Theta_{n-1}$ is $\mathbb{F}_{n-1}/\mathcal{B}(\R^d)$-measurable,
 the fact that for all $\theta \in \R^d$ it holds that $\mathbb{E}\big[\Vert G(\theta,X_1)- g(\theta)\Vert\big] < \infty$, and the fact that for all $n \in \N$ it holds that $X_n$ is independent of $\mathbb{F}_{n-1}$ prove that for all $n \in \N$, $A \in \F_{n-1}$  with $\Exp{\norm{D_n}} < \infty$ it holds that
\begin{equation}
\label{convergence_of_SA:eq3}
\begin{split}
\EXP{D_n \ind{A}} 
&= \EXPP{\big(G(\Theta_{n-1},X_n) - g(\Theta_{n-1})\big) \ind{A}} \\
&= \EXPP{\EXP{\big(G(\Theta_{n-1},X_n) - g(\Theta_{n-1})\big) \ind{A} \!\mid\! \F_{n-1}}} \\
&= \EXPP{\big(\EXP{G(\Theta_{n-1},X_n)\!\mid\! \F_{n-1}} - g(\Theta_{n-1}) \big)\ind{A}} \\
&= \EXPP{\big(g(\Theta_{n-1})-g(\Theta_{n-1})\big)\ind{A}} = 0.
\end{split}
\end{equation}
Furthermore, observe that Lemma \ref{example_learningrate} ensures that for all $k \in (0,\infty)$ it holds that
\begin{equation}
\limsup_{n \to \infty} \gamma_n  = 0  <\liminf_{n \to \infty}\left[\frac{(\gamma_n)^{k}-(\gamma_{n-1})^{k}}{(\gamma_n)^{k+1}} + \frac{c (\gamma_{n-1})^k}{2 (\gamma_{n})^k} \right].
\end{equation}
This implies that
\begin{equation}
\limsup_{l \to \infty} \gamma_l =0 < \min_{k \in \{1,2,\dots,\nicefrac{p}{2}\}} \left(\liminf_{l \to \infty} \left[\frac{(\gamma_l)^k-(\gamma_{l-1})^k}{(\gamma_l)^{k+1}} +\frac{c (\gamma_{l-1})^k}{2(\gamma_l)^k} \right] \right).
\end{equation}
Combining the fact that $D$ is an $(\mathbb{F}_n)_{n \in \N}/\mathcal{B}(\R^d)$-adapted stochastic process, the fact that the function $\Theta_0$ is $\mathbb{F}_0/\mathcal{B}(\R^d)$-measurable, (\ref{convergence_of_SA:assumption2}), and (\ref{convergence_of_SA:eq1})--(\ref{convergence_of_SA:eq3}) with Theorem \ref{Lp_theorem} hence demonstrates that there exists $C \in (0,\infty)$ such that for all $n \in \N$ it holds that
\begin{equation}
\{\theta \in \R^d \colon g(\theta)=0 \} = \{\vartheta\}
\end{equation} 
 and
\begin{equation}
\big(\EXP{ \norm{\Theta_n-\vartheta}^p }\big)^{\nicefrac{1}{p}} \leq C (\gamma_n)^{\nicefrac{1}{2}} = [C \sqrt{\alpha}] n^{-\nicefrac{\nu}{2}}.
\end{equation}
This establishes (\ref{convergence_of_SA:conclusion}). The proof of Proposition \ref{convergence_of_SA} is thus completed.
\end{proof}
\subsection{Strong $L^p$-convergence rate for stochastic gradient descent}
\label{subsection:SGD}
\begin{lemma}
\label{le:meas-grad}
Let $d \in \N$,  let $(S, \mathcal{S})$ be a measurable space, 
let $F = ( F(\theta,x) )_{\theta \in \R^d, x \in S}\colon$ $\R^d \times S \to \R$ be $(\mathcal{B}(\R^{d}) \otimes \mathcal{S}) / \mathcal{B}(\R)$-measurable, and assume for all $x \in S$ that
\begin{equation}
\label{le:meas-grad:assumption1}
 (\R^d \ni \theta \mapsto F(\theta,x) \in \R) \in C^1(\R^d, \R). 
\end{equation}
Then it holds that the function 
\begin{equation}
\R^d \times S \ni (\theta, x) \mapsto(\nabla_{\theta}F)(\theta,x) \in \R^d
\end{equation}
 is $(\mathcal{B}(\R^{d}) \otimes \mathcal{S}) / \mathcal{B}(\R^d)$-measurable.
\end{lemma}
\begin{proof}[Proof of Lemma \ref{le:meas-grad}]
Throughout this proof let $G = (G_1, \dots, G_d) \colon \R^d \times S \to \R^d$ be the function which satisfies for all $\theta \in \R^d$, $x\in S$ that
\begin{equation}
\label{le:meas-grad:eq1}
G(\theta,x)= (\nabla_\theta F)(\theta,x).
\end{equation}
The hypothesis that the function $F \colon \R^d \times S \to \R$ is $(\mathcal{B}(\R^{d}) \otimes \mathcal{S}) / \mathcal{B}(\R)$-measurable implies that for all $i \in \{1,\dots,d\}$, $h\in \R\backslash\{0\}$ it holds that the function
\begin{equation}
\begin{split}
&\R^d \times S \ni(\theta,x)=((\theta_1,\dots,\theta_d),x) \mapsto \left(\tfrac{F((\theta_1,\dots,\theta_{i-1},\theta_{i}+h,\theta_{i+1},\dots,\theta_d),x)-F(\theta,x)}{h}
\right) \in \R
\end{split}
\end{equation}
 is $(\mathcal{B}(\R^d) \otimes \mathcal{S}) / \mathcal{B}(\R)$-measurable. The fact that for all $i \in \{1,\dots,d\}$, $\theta=(\theta_1,\dots,\theta_d) \in \R^d$, $x\in S$ it holds that
\begin{equation}
G_i(\theta,x)= \lim\limits_{n \to \infty} \left(\tfrac{F((\theta_1,\dots,\theta_{i-1},\theta_{i}+2^{-n},\theta_{i+1},\dots,\theta_d),x)-F(\theta,x)}{2^{-n}}
\right)
\end{equation}
hence ensures that  for all $i \in \{1,\dots,d\}$ it holds that the function $G_i \colon \R^d \times S \to \R$ is $(\mathcal{B}(\R^{d}) \otimes \mathcal{S}) / \mathcal{B}(\R)$-measurable. 
This and (\ref{le:meas-grad:eq1}) complete the proof of Lemma~\ref{le:meas-grad}.
\end{proof}
\begin{cor}
	\label{convergence_of_SA-Vers-G-F}
	Let $d \in \N$, $p \in \{2,4,6,\ldots \}$, $\alpha,\kappa,c \in (0,\infty)$, $ \nu \in (0,1)$, $\xi, \vartheta \in \R^d$, let 
	$\lll \cdot,\cdot \rrr \colon \R^d \times \R^d \to \R$
	be a scalar product, let 
	$\left \| \cdot \right \| \! \colon \R^d \to [0,\infty)$ 
	be the function which satisfies for all $\theta \in \R^d$ that $\norm{\theta} = \sqrt{\lll \theta,\theta \rrr}$, let $ ( \Omega , \mathcal{F}, \P) $ be a probability space, let $(S, \mathcal{S})$ be a measurable space, let $X_{n} \colon \Omega \to S$, $n \in \N$, be i.i.d.\ random variables,
	let $F = ( F(\theta,x) )_{\theta \in \R^d, x \in S} \colon \R^d \times S \to \R$ be $(\mathcal{B}(\R^{d}) \otimes \mathcal{S}) / \mathcal{B}(\R)$-measurable, let $g \colon \R^d \to \R^d$ be a function,  assume for all $x \in S$ that $(\R^d \ni \theta \mapsto F(\theta,x) \in \R) \in C^1(\R^d, \R)$, assume for all $\theta \in \R^d$ that
	\begin{equation}
	\label{convergence_of_SA-Vers-G-F:assumption1}
	\EXP{\norm{(\nabla_\theta F)(\theta, X_1) - g(\theta)}^p } \leq \kappa \big(1+ \norm{\theta}^p\big), 
	\end{equation}
	\begin{equation}
	\label{convergence_of_SA-Vers-G-F:assumption2}
	\lll \theta - \vartheta, g(\theta)  \rrr   \leq  -c\max\!\big\{\norm{\theta - \vartheta}^2, \norm{g(\theta)}^2 \big\},  
	\end{equation}
	and $g(\theta) = \EXP{(\nabla_\theta F)(\theta, X_1)}$,
	and let $\Theta \colon \N_0 \times \Omega \to \R^d$ be the function which satisfies for all $n \in \N$ that 
	\begin{equation}
	\label{convergence_of_SA-Vers-G-F:assumption3}
	\Theta_0 = \xi   \qquad \text{ and } 
	\qquad \Theta_n = \Theta_{n-1} + \tfrac{\alpha}{n^\nu} (\nabla_\theta F)(\Theta_{n-1},X_n).
	\end{equation}
	Then there exists $C \in (0,\infty)$ such that for all $n \in \N$ it holds that
	\begin{equation}
	\label{convergence_of_SA-Vers-G-F:conclusion}
	\{\theta \in \R^d \colon g(\theta)=0 \} = \{\vartheta\} \qandq \left(\EXP{ \norm{\Theta_n-\vartheta}^p }\right)^{\nicefrac{1}{p}} \leq C n^{-\nicefrac{\nu}{2}}.
	\end{equation}
\end{cor}
\begin{proof}[Proof of Corollary \ref{convergence_of_SA-Vers-G-F}]
Combining Lemma~\ref{le:meas-grad} and Proposition \ref{convergence_of_SA} (with $G(\theta,x) = \nabla_{\theta}F(\theta,x)$, $g(\theta) =  g(\theta)$ for $\theta \in \R^d$, $x\in S$ in the notation of Proposition \ref{convergence_of_SA}) establishes (\ref{convergence_of_SA-Vers-G-F:conclusion}). The proof of Corollary~\ref{convergence_of_SA-Vers-G-F} is thus completed.
\end{proof}
\begin{cor}
	\label{convergence_of_SA-Vers-G-F-2}
	Let $d \in \N$, $p \in \{2,4,6,\ldots \}$, $\alpha,\kappa,c \in (0,\infty)$, $ \nu \in (0,1)$, $\xi, \vartheta \in \R^d$, let 
	$\lll \cdot,\cdot \rrr \colon \R^d \times \R^d \to \R$
	be a scalar product, let 
	$\left \| \cdot \right \| \! \colon \R^d \to [0,\infty)$ 
	be the function which satisfies for all $\theta \in \R^d$ that $\norm{\theta} = \sqrt{\lll \theta,\theta \rrr}$, let $ ( \Omega , \mathcal{F}, \P) $ be a probability space, let $(S, \mathcal{S})$ be a measurable space, let $X_{n} \colon \Omega \to S$, $n \in \N$, be i.i.d.\ random variables,
	 and let $F = ( F(\theta,x) )_{\theta \in \R^d, x \in S} \colon \R^d \times S \to \R$ be $(\mathcal{B}(\R^{d}) \otimes \mathcal{S}) / \mathcal{B}(\R)$-measurable,  assume for all $x \in S$ that $(\R^d \ni \theta \mapsto F(\theta,x) \in \R) \in C^1(\R^d, \R)$, assume for all $\theta \in \R^d$ that
	\begin{equation}
	\label{convergence_of_SA-Vers-G-F-2:assumption0}
	\EXP{\norm{(\nabla_\theta F)(\theta,X_1)}}<\infty,
	\end{equation}
\vspace{-.6cm}
\begin{equation}
\label{convergence_of_SA-Vers-G-F-2:assumption2}
\lll \theta - \vartheta, \Exp{(\nabla_\theta F)(\theta, X_1)}\rrr  \leq  -c\max\!\big\{\norm{\theta - \vartheta}^2, \norm{\Exp{(\nabla_\theta F)(\theta, X_1)}\!}^2\!\big\},  
\end{equation}
	\begin{equation}
	\label{convergence_of_SA-Vers-G-F-2:assumption1}
	\EXP{\norm{(\nabla_\theta F)(\theta, X_1) - \Exp{(\nabla_\theta F)(\theta, X_1)}\!}^p } \leq \kappa \big(1+ \norm{\theta}^p\big), 
	\end{equation}
	and let $\Theta \colon \N_0 \times \Omega \to \R^d$ be the function which satisfies for all $n \in \N$ that 
	\begin{equation}
	\label{convergence_of_SA-Vers-G-F-2:assumption3}
	\Theta_0 = \xi   \qquad \text{ and } 
	\qquad \Theta_n = \Theta_{n-1} + \tfrac{\alpha}{n^\nu} (\nabla_\theta F)(\Theta_{n-1},X_n).
	\end{equation}
	Then 
	\begin{enumerate}[(i)]
	\item it holds that
	 $\big\{\theta \in \R^d \colon \Exp{(\nabla_\theta F)(\theta, X_1)}=0 \big\} = \{\vartheta\}$ and
	 \item there exists $C \in (0,\infty)$ such that for all $n \in \N$ it holds that
	\begin{equation}
	\label{convergence_of_SA-Vers-G-F-2:conclusion}
	\left(\EXP{ \norm{\Theta_n-\vartheta}^p }\right)^{\nicefrac{1}{p}} \leq C n^{-\nicefrac{\nu}{2}}.
	\end{equation}
	\end{enumerate}
\end{cor}
\begin{proof}[Proof of Corollary \ref{convergence_of_SA-Vers-G-F-2}]
Corollary~\ref{convergence_of_SA-Vers-G-F} (with $g(\theta) =\Exp{(\nabla_\theta F)(\theta,X_1)}$ for $\theta \in \R^d$ in the notation of Corollary~\ref{convergence_of_SA-Vers-G-F}) establishes (\ref{convergence_of_SA-Vers-G-F-2:conclusion}). The proof of Corollary~\ref{convergence_of_SA-Vers-G-F-2} is thus completed.
\end{proof}
%
%
\begin{lemma}\label{le:de-la-valee-Poussin-HELP}
	Let $d \in \N$, $p \in [1,\infty)$, let 
	$\lll \cdot,\cdot \rrr \colon \R^d \times \R^d \to \R$
	be a scalar product, let 
	$\left \| \cdot \right \| \! \colon \R^d \to [0,\infty)$ 
	be the function which satisfies for all $\theta \in \R^d$ that $\norm{\theta} = \sqrt{\lll \theta,\theta \rrr}$, let $ ( \Omega , \mathcal{F}, \P) $ be a probability space, let $\mathbb{I}$ be a non-empty set, let $X_{i} \colon \Omega \to \R^d$, $i \in \mathbb{I}$, be random variables, assume for all $i \in \mathbb{I}$ that $\Exp{\norm{X_i}}<\infty$, and assume that
	$\sup_{i \in \mathbb{I}} \Exp{\norm{X_i- \Exp{X_i}\!}^p}<\infty$
	and 
	$\P(\sup_{i\in \mathbb{I}}\norm{X_i}<\infty)>0$. Then it holds that
	\begin{equation}
	\label{le:de-la-valee-Poussin-HELP:item1}
	\sup\nolimits_{i\in \mathbb{I}} \EXP{\norm{X_i}^p}<\infty.
	\end{equation}
\end{lemma}
%
%
\begin{proof}
	[Proof of Lemma \ref{le:de-la-valee-Poussin-HELP}]
Throughout this proof let  $Y_i\colon \Omega  \to [0,\infty)$, $i\in \mathbb{I}$, be the random variables which satisfy for all $i \in \mathbb I$ that
	\begin{equation}
	Y_i= \norm{X_{i}-\Exp{X_{i}}\!}^p
	\end{equation}	
and let $j=(j_k)_{k\in \N}\colon \N \to \mathbb I$ be a function which satisfies for all $k \in \N$ that
\begin{equation}
\label{NRL}
\EXP{\norm{X_{j_k}}^p}\leq \EXP{\norm{X_{j_{k+1}}}^p} 
\qandq
\limsup_{l\to \infty} \EXP{\norm{X_{j_l}}^p} = \sup_{i \in \mathbb I} \EXP{\norm{X_{i}}^p}.
\end{equation}
Observe that the hypothesis that $\sup_{i \in \mathbb{I}} \Exp{\norm{X_i- \Exp{X_i}\!}^p}<\infty$ and Fatou's lemma imply that for all functions $(n_k)_{k \in \N}\colon \N \to \mathbb{I}$ it holds that
	\begin{equation}
	\label{le:de-la-valee-Poussin-HELP:eq1}
	\begin{split}
\EXP{\liminf\nolimits_{k\to \infty} Y_{n_k}}
	&\leq \liminf_{k\to \infty} \EXP{Y_{n_k}}\\
	&= \liminf_{k\to \infty}\EXP{\norm{X_{n_k}- \Exp{X_{n_k}}\!}^p}\\
	&\leq \sup\nolimits_{i \in \mathbb{I}} \Exp{\norm{X_i- \Exp{X_i}\!}^p}<\infty.
	\end{split}
	\end{equation}
Moreover, note that the  triangle inequality assures that for all functions $(n_k)_{k \in \N}\colon \N \to \mathbb{I}$ it holds that
\begin{equation}
\begin{split}
\EXP{\liminf\nolimits_{k\to \infty} Y_{n_k}}
&= \EXP{\liminf\nolimits_{k\to \infty} \norm{ \Exp{X_{n_k}}-X_{n_k}}^p}\\
&\geq \EXPP{\liminf\nolimits_{k\to \infty}\big|\norm{ \Exp{X_{n_k}}\!}-\norm{X_{n_k}}\big|^p}.
\end{split}
\end{equation}
The fact that $\forall \, x \in[0,\infty) \colon x^p\geq \max\{x,0\}-1$ therefore implies that for all functions $(n_k)_{k \in \N}\colon \N \to \mathbb{I}$ it holds that
\begin{equation}
\begin{split}
&\EXP{\liminf\nolimits_{k\to \infty} Y_{n_k}}\\
&\geq \EXPP{\liminf\nolimits_{k\to \infty}\max\!\big\{|\norm{ \Exp{X_{n_k}}\!}-\norm{X_{n_k}}|,0\big\}}-1\\
&\geq \EXPP{\liminf\nolimits_{k\to \infty}\max\!\big\{\norm{ \Exp{X_{n_k}}\!}-\norm{X_{n_k}},0\big \}}-1\\
&\geq \EXPP{\liminf\nolimits_{k\to \infty}\max\!\big\{\norm{ \Exp{X_{n_k}}\!}-\sup\nolimits_{i \in \mathbb{I}} \norm{X_i},0\big\}}-1.
\end{split}
\end{equation}
Combining this with \eqref{le:de-la-valee-Poussin-HELP:eq1} proves that for all functions $(n_k)_{k \in \N}\colon \N \to \mathbb{I}$ it holds that
\begin{equation}
\EXPP{\liminf\nolimits_{k\to \infty}\max\!\big\{\norm{ \Exp{X_{n_k}}\!}-\sup\nolimits_{i \in \mathbb{I}} \norm{X_i},0\big\}}<\infty.
\end{equation}
Hence, we obtain that for all functions $(n_k)_{k \in \N}\colon \N \to \mathbb{I}$ it holds that
\begin{equation}
\mathbb{P}\Big(\liminf\nolimits_{k\to \infty}\max\!\big\{\norm{ \Exp{X_{n_k}}\!}-\sup\nolimits_{i \in \mathbb{I}} \norm{X_i},0\big\}<\infty \Big) =1.
\end{equation}
Therefore, we obtain that for all functions $(n_k)_{k \in \N}\colon \N \to \mathbb{I}$ it holds that
\begin{equation}
\label{NRZ}
\mathbb{P}\Big(\liminf\nolimits_{k\to \infty}\big[\norm{ \Exp{X_{n_k}}\!}-\sup\nolimits_{i \in \mathbb{I}} \norm{X_i}\big]<\infty \Big) =1.
\end{equation}
Next note that for all $A, B \in \mathcal{F}$ with $\mathbb{P}(A)=1$ and $\mathbb{P}(B)>0$ it holds that
\begin{equation}
\mathbb{P}(\Omega \setminus(A \cap B))
=
\mathbb{P}([\Omega\setminus A] \cup [\Omega\setminus B])
\leq  \mathbb{P}(\Omega\setminus A) + \mathbb{P}(\Omega\setminus B)
=\mathbb{P}(\Omega\setminus B)
<1.
\end{equation}
Hence, we obtain that for all $A, B \in \mathcal{F}$ with $\mathbb{P}(A)=1$ and $\mathbb{P}(B)>0$ it holds that $\mathbb{P}(A\cap B)>0$. Combining this and the hypothesis that $\P(\sup_{i\in \mathbb{I}} \norm{X_i}<\infty)>0$ with \eqref{NRZ} proves that for all functions $(n_k)_{k \in \N}\colon \N \to \mathbb{I}$ it holds that
\begin{equation}
\P\Big(\big\{\liminf\nolimits_{k\to \infty}\big[\norm{ \Exp{X_{n_k}}\!}-\sup\nolimits_{i \in \mathbb{I}} \norm{X_i}\big]<\infty\big\} \cap \big\{ \sup\nolimits_{i\in \mathbb{I}} \norm{X_i}<\infty\big\}\Big)>0.
\end{equation}
Therefore, we obtain that for all functions $(n_k)_{k \in \N}\colon \N \to \mathbb{I}$ it holds that
\begin{equation}
\label{NRZ2}
\begin{split}
&\big\{\liminf\nolimits_{k\to \infty}\big[\norm{ \Exp{X_{n_k}}\!}-\sup\nolimits_{i \in \mathbb{I}} \norm{X_i}\big]<\infty\big\} \cap \big\{ \sup\nolimits_{i\in \mathbb{I}} \norm{X_i}<\infty\big\}\\
&=\Big\{ \omega \in \Omega \colon \big(\liminf\nolimits_{k\to \infty}\big[\norm{ \Exp{X_{n_k}}\!}-\sup\nolimits_{i \in \mathbb{I}} \norm{X_i(\omega)}\big]<\infty, \ \sup\nolimits_{i\in \mathbb{I}} \norm{X_i(\omega)}<\infty \big)\Big\}\\
&\neq \emptyset.
\end{split}
\end{equation}
Moreover, note that for all functions $(n_k)_{k \in \N}\colon \N \to \mathbb{I}$ and all $\omega \in (\{\liminf\nolimits_{k\to \infty}$ $[\norm{ \Exp{X_{n_k}}\!}-\sup\nolimits_{i \in \mathbb{I}} \norm{X_i}]<\infty\} \cap \{ \sup\nolimits_{i\in \mathbb{I}} \norm{X_i}<\infty\})$ it holds that
\begin{equation}
\label{le:de-la-valee-Poussin-HELP:eq2}
\liminf_{k \to \infty} \norm{\Exp{X_{n_k}}\!}<\infty.
\end{equation}
Combining this with \eqref{NRZ2} proves that  for all functions $(n_k)_{k \in \N}\colon \N \to \mathbb{I}$ it holds that
\begin{equation}
\label{NRZZZZ}
\liminf_{k \to \infty} \norm{\Exp{X_{n_k}}\!}<\infty.
\end{equation}
Moreover, observe that Lemma~\ref{weak_triangle} ensures that for all functions $(n_k)_{k \in \N}\colon \N \to \mathbb{I}$ it holds that
	\begin{equation}
	\begin{split}
	\label{le:de-la-valee-Poussin-HELP:eq3}
	\liminf_{k\to \infty} \EXP{\norm{X_{n_k}}^p} &= \liminf_{k\to \infty} \EXP{\norm{X_{n_k}-\Exp{X_{n_k}}+\Exp{X_{n_k}}\!}^p}\\
	& \leq 2^p \liminf_{k\to \infty} \Big(\EXP{\norm{X_{n_k}-\Exp{X_{n_k}}\!}^p} +\norm{\Exp{X_{n_k}}\!}^p  \Big)\\
	& \leq 2^p\big(\sup\nolimits_{i \in \mathbb{I}} \Exp{\norm{X_i- \Exp{X_i}\!}^p}\!\big) + 2^p \big[\liminf\nolimits_{k\to \infty}\norm{\Exp{X_{n_k}}\!}\big]^p.
	\end{split}
	\end{equation}
	The hypothesis that $\sup_{i \in \mathbb{I}} \Exp{\norm{X_i- \Exp{X_i}\!}^p}<\infty$ and (\ref{NRZZZZ}) hence imply that for all functions $(n_k)_{k \in \N}\colon \N \to \mathbb{I}$ it holds that
	\begin{equation}
	\label{le:de-la-valee-Poussin-HELP:eq4}
	 \liminf_{k\to \infty} \EXP{\norm{X_{n_k}}^p} <\infty.
	\end{equation}
	This and \eqref{NRL} prove that
	\begin{equation}
	\sup_{i \in \mathbb I} \EXP{\norm{X_i}^p} 
	= \limsup_{k \to \infty} \EXP{\norm{X_{j_k}}^p} 
	= \liminf_{k \to \infty} \EXP{\norm{X_{j_k}}^p} 
	< \infty.
	\end{equation}
	The proof of Lemma~\ref{le:de-la-valee-Poussin-HELP} is thus completed.
\end{proof}
\begin{lemma}
\label{le:derivative-expectation-interchange}
Let $d \in \N$, $p \in (1,\infty)$, $\kappa\in(0,\infty)$, $\vartheta \in\R^d$, let 
$\lll \cdot,\cdot \rrr \colon \R^d \times \R^d \to \R$
be a scalar product, let 
$\left \| \cdot \right \| \! \colon \R^d \to [0,\infty)$ 
be the function which satisfies for all $\theta \in \R^d$ that $\norm{\theta} = \sqrt{\lll \theta,\theta \rrr}$, let $ ( \Omega , \mathcal{F}, \P) $ be a probability space, let $(S, \mathcal{S})$ be a measurable space,
let $X \colon \Omega \to S$ be a random variable, let $F = ( F(\theta,x) )_{\theta \in \R^d, x \in S} \colon \R^d \times S \to \R$ be $(\mathcal{B}(\R^{d}) \otimes \mathcal{S}) / \mathcal{B}(\R)$-measurable, let $f\colon \R^d \to \R$ be a function, assume for all $x \in S$ that $(\R^d \ni \theta \mapsto F(\theta,x) \in \R) \in C^1(\R^d, \R)$, and assume for all $\theta \in \R^d$ that
\begin{equation}
\label{le:derivative-expectation-interchange:ass2}
\EXP{|F(\theta,X)|+ \norm{(\nabla_\theta F)(\theta,X)}}<\infty, 
\end{equation}
\vspace{-.7cm}
\begin{equation}
\label{le:derivative-expectation-interchange:ass3}
\EXP{\norm{(\nabla_{\theta}F)(\theta, X) - \Exp{(\nabla_{\theta}F)(\theta, X)}\!}^p } \leq \kappa \big(1+ \norm{\theta}^p\big),
\end{equation}
 and $f(\theta)=\Exp{F(\theta,X)}$.
Then
\begin{enumerate}[(i)]
\item \label{le:derivative-expectation-interchange:item1}
it holds that $f \in C^1(\R^d,\R)$ and
\item 
\label{le:derivative-expectation-interchange:item2}
it holds for all $\theta \in \R^d$ that
\begin{equation}
\label{le:derivative-expectation-interchange:item2:eq}
(\nabla f)(\theta)=\EXP{(\nabla_{\theta}F)(\theta, X)}.
\end{equation}
\end{enumerate}
\end{lemma}
\begin{proof}[Proof of Lemma \ref{le:derivative-expectation-interchange}]
	Throughout this proof let $c\in (0,\infty)$ satisfy
	\begin{equation}
	\label{le:derivative-expectation-interchange:eq:norm-equiv}
	c=\sup_{\theta=(\theta_1,\dots,\theta_d) \in \R^d\setminus\{0\}}\left(\frac{\big(\sum_{j=1}^d |\theta_j|\big)}{\norm{\theta}}\right),
	\end{equation} 
	 let $q=p-1 \in (0,\infty)$, and let $e_1=(1,0,\dots,0)$, $e_2=(0,1,0,\dots,0)$, $\dots$, $e_d=(0,\dots,0,1) \in \R^d$.
Observe that (\ref{le:derivative-expectation-interchange:ass2}), (\ref{le:derivative-expectation-interchange:ass3}), and Lemma~\ref{weak_triangle} imply that 
for all $\theta \in \R^d$, $v \in [-q,q]^d$ it holds that
\begin{equation}
\label{le:derivative-expectation-interchange:eq1.1}
\begin{split}
&\EXP{\norm{(\nabla_{\theta}F)(\theta+v, X)}^p}\\
&\leq 
2^p\Big( \EXP{\norm{(\nabla_{\theta}F)(\theta+v, X) - \Exp{(\nabla_{\theta}F)(\theta+v, X)}\!}^p } + \norm{\Exp{(\nabla_{\theta}F)(\theta+v, X)}\!}^p\Big)\\
&\leq 
2^p\Big(
\kappa \big(1+ \norm{\theta+v}^p\big) + | \Exp{\norm{(\nabla_{\theta}F)(\theta+v, X)}}\!|^p
\Big)<\infty.
\end{split}
\end{equation}
Moreover, note that (\ref{le:derivative-expectation-interchange:ass3}) assures that for all $\theta \in \R^d$ it holds that
\begin{equation}
\begin{split}
\label{le:derivative-expectation-interchange:eq1.2}
&\sup_{v \in [-q,q]^d}\Big( \EXP{\norm{(\nabla_{\theta}F)(\theta+v, X) - \Exp{(\nabla_{\theta}F)(\theta+v, X)}\!}^p }\Big)\\
& \leq  \sup_{v \in [-q,q]^d} \Big(\kappa \big(1+ \norm{\theta+v}^p\big) \Big) <\infty.
\end{split}
\end{equation}
Furthermore, observe that the hypothesis that for all $x \in S$ it holds that $(\R^d \ni \theta \mapsto F(\theta,x) \in \R) \in C^1(\R^d, \R)$ 
ensures that  for all $\theta \in \R^d$ it holds  that  
\begin{equation}
\label{le:derivative-expectation-interchange:eq1.3}
\P\bigg(\sup_{v \in [-q,q]^d} \norm{(\nabla_{\theta}F)(\theta+v, X)}<\infty\bigg)=1.
\end{equation}
Lemma~\ref{le:de-la-valee-Poussin-HELP} (with $\mathbb{I} =[-q,q]^d$, $X_i=(\nabla_\theta F)(\theta+i,X)$ for $i \in [-q,q]^d$, $\theta \in \R^d$ in the notation of Lemma~\ref{le:de-la-valee-Poussin-HELP}),
 (\ref{le:derivative-expectation-interchange:ass2}), and (\ref{le:derivative-expectation-interchange:eq1.2}) therefore imply that for all $\theta \in \R^d$ it holds that
\begin{equation}
\label{le:derivative-expectation-interchange:eq3}
\sup_{v \in [-q,q]^d} \EXP{\norm{(\nabla_{\theta}F)(\theta+v, X)}^p}<\infty.
\end{equation}
This and (\ref{le:derivative-expectation-interchange:eq:norm-equiv}) demonstrate that for all $i \in \{1,2,\dots,d\}$, $\theta \in \R^d$ it holds that
\begin{equation}
\label{le:derivative-expectation-interchange:eq4}
\begin{split}
\sup_{v \in [-q,q]^d}\EXP{|(\tfrac{\partial}{\partial\theta_i}F)(\theta+v,X)|^p
}
&\leq \sup_{v \in [-q,q]^d}\EXPP{\big(\textstyle\sum_{j=1}^d|(\tfrac{\partial}{\partial\theta_j}F)(\theta+v,X)|\big)^p
}
\\
& \leq c^p \bigg[\sup_{v \in [-q,q]^d}\EXP{\norm{(\nabla_{\theta}F)(\theta+v, X)}^p}\bigg] <\infty.
\end{split}
\end{equation}
Next observe that the hypothesis that for all $x \in S$ it holds that $(\R^d \ni \theta \mapsto F(\theta,x) \in \R) \in C^1(\R^d, \R)$
 and the fundamental theorem of calculus ensure that for all 
 $i \in \{1,2,\dots,d\}$, $\theta \in \R^d$, $h \in \R$ it holds that
\begin{equation}
\label{le:derivative-expectation-interchange:eq5}
\begin{split}
f(\theta+h e_i)-f(\theta)&= \Exp{F(\theta+h e_i,X)-F(\theta,X)}\\
&=\EXPP{ \big[F(\theta + u e_i,X)\big]_{u=0}^{u=h}}\\
&=\EXPPP{\int_0^h(\tfrac{\partial}{\partial \theta}F)(\theta + u e_i,X)e_i \ \mathrm{d}u}\\
&=\EXPPP{\int_0^h(\tfrac{\partial}{\partial \theta_i}F)(\theta + u e_i,X)\,\mathrm{d}u}.
\end{split}
\end{equation}
Moreover, note that Fubini's theorem (see, e.g., Klenke~\cite[Theorem 14.16]{Klenke}), (\ref{le:derivative-expectation-interchange:eq4}), the hypothesis that $p>1$, and Jensen's inequality assure that for all $i \in \{1,2,\dots,d\}$, 
$\theta \in \R^d$, 
$h \in [-q,q]$ it holds that
\begin{equation}
\begin{split}
\EXPPP{\int_{\min\{h,0\}}^{\max\{h,0\}}|(\tfrac{\partial}{\partial \theta_i}F)(\theta + u e_i,X)|\,\mathrm{d}u} 
&=\int_{\min\{h,0\}}^{\max\{h,0\}}\EXP{|(\tfrac{\partial}{\partial \theta_i}F)(\theta + u e_i,X)|}\,\mathrm{d}u
\\&\leq
|h| \bigg[\sup_{v \in [-q,q]^d} \EXP{|(\tfrac{\partial}{\partial \theta_i}F)(\theta + v,X)|} \bigg]
\\ & \leq |h| \bigg[\sup_{v \in [-q,q]^d}\big( \EXP{|(\tfrac{\partial}{\partial \theta_i}F)(\theta + v,X)|^p}\big)^{\nicefrac{1}{p}} \bigg]
\\ & < \infty.
\end{split}
\end{equation}
This, (\ref{le:derivative-expectation-interchange:eq5}), and again  Fubini's theorem (see, e.g., Klenke~\cite[Theorem 14.16]{Klenke}) imply
that for all 
$i \in \{1,2,\dots,d\}$, $\theta \in \R^d$,
$h \in [-q,q]$ it holds that
\begin{equation}
\label{le:derivative-expectation-interchange:eq6}
\begin{split}
f(\theta+h e_i)-f(\theta)
&=\EXPPP{\int_0^h(\tfrac{\partial}{\partial \theta_i}F)(\theta + u e_i,X)\,\mathrm{d}u}\\
&=
\int_0^h \EXP{(\tfrac{\partial}{\partial \theta_i}F)(\theta + u e_i,X)}\,\mathrm{d}u.
\end{split}
\end{equation}
In addition, note that 
(\ref{le:derivative-expectation-interchange:eq4}), the hypothesis that $p>1$,   
 and the de la Vall\'ee Poussin theorem (cf., e.g., Klenke~\cite[Corollary~6.21]{Klenke}) ensure that for all $i \in \{1,2,\dots,d\}$, $\theta \in \R^d$ it holds that the family of random variables
\begin{equation}
\label{le:derivative-expectation-interchange:eqUI}
\big(\Omega \ni\omega \mapsto (\tfrac{\partial}{\partial \theta_i}F)(\theta + v,X(\omega)) \in  \R \big),\qquad v \in  [-q,q]^d,
\end{equation}
is uniformly integrable. 
The fact that for all $i \in\{1,2,\dots,d\}$, $x\in S$ it holds that the function
$\R^d \ni \theta \mapsto (\frac{\partial}{\partial \theta_i} F)(\theta,x) \in \R$ is continuous and the Vitali  convergence theorem (cf., e.g., Klenke \cite[Theorem~6.25]{Klenke}) hence imply that  for all $i \in \{1,2,\dots,d\}$, $\theta \in \R^d$ and all functions $v=(v_n)_{n \in \N} \colon \N \to \R$ with $\limsup_{n \to \infty} |v_n| =0$ it holds that
\begin{equation}
\limsup_{n \to \infty} \EXPP{\big|(\tfrac{\partial}{\partial \theta_i}F)(\theta+v_n e_i,X)-(\tfrac{\partial}{\partial \theta_i}F)(\theta,X) \big|}=0.
\end{equation}
Hence, we obtain that for all $i \in\{1,2,\dots,d\}$, $\theta  \in \R^d$, $\varepsilon \in (0,\infty)$ there exists $\delta \in (0,\infty)$ such that
\begin{equation}
\sup_{u\in[-\delta,\delta]} \EXPP{\big|(\tfrac{\partial}{\partial \theta_i}F)(\theta+u e_i,X)-(\tfrac{\partial}{\partial \theta_i}F)(\theta,X) \big|}\leq \varepsilon.
\end{equation}
Therefore, we obtain that for all $i \in \{1,2,\dots,d\}$, $\theta \in \R^d$ and all functions $h=(h_n)_{n \in \N} \colon \N \to \R$ with $\limsup_{n \to \infty} |h_n|=0$ it holds that
\begin{equation}
\limsup_{n \to \infty}\sup_{u\in[-|h_n|,|h_n|]} \EXPP{\big|(\tfrac{\partial}{\partial \theta_i}F)(\theta+u e_i,X)-(\tfrac{\partial}{\partial \theta_i}F)(\theta,X) \big|}=0.
\end{equation}
This and \eqref{le:derivative-expectation-interchange:eq6} demonstrate that for all $i\in\{1,2,\dots,d\}$, $\theta \in \R^d$ and all functions $h=(h_n)_{n\in \N} \colon \N \to \R\setminus\{0\}$ with $\limsup_{n \to \infty} |h_n|=0$ it holds that
\begin{equation}
\label{V0}
\begin{split}
& \limsup_{n\to \infty} \left| \tfrac{f(\theta+h_n e_i)-f(\theta)}{h_n}- \EXP{(\tfrac{\partial}{\partial \theta_i}F)(\theta,X)} \right|\\
&=\limsup_{n\to \infty} \Bigg[ \frac{1}{|h_n|}\left|f(\theta+h_n e_i)-f(\theta)- h_n\,\EXP{(\tfrac{\partial}{\partial \theta_i}F)(\theta,X)} \right|\Bigg]\\
&=\limsup_{n\to \infty} \Bigg[ \frac{1}{|h_n|}\left|f(\theta+h_n e_i)-f(\theta)- \int_{0}^{h_n}\EXP{(\tfrac{\partial}{\partial \theta_i}F)(\theta,X)}\,\mathrm{d}u \right|\Bigg]\\
&=\limsup_{n\to \infty} \Bigg[ \frac{1}{|h_n|}\left|\int_0^{h_n} \EXP{(\tfrac{\partial}{\partial \theta_i}F)(\theta + u e_i,X)}\,\mathrm{d}u- \int_{0}^{h_n}\EXP{(\tfrac{\partial}{\partial \theta_i}F)(\theta,X)}\,\mathrm{d}u \right|\Bigg]\\
&=\limsup_{n\to \infty} \Bigg[ \frac{1}{|h_n|}\left|\int_0^{h_n} \EXP{(\tfrac{\partial}{\partial \theta_i}F)(\theta + u e_i,X)-(\tfrac{\partial}{\partial \theta_i}F)(\theta,X)}\,\mathrm{d}u \right|\Bigg]\\
&\leq\limsup_{n\to \infty} \Bigg[ \frac{1}{|h_n|}\int_{\min\{h_n,0\}}^{\max\{h_n,0\}} \EXPP{\big|(\tfrac{\partial}{\partial \theta_i}F)(\theta + u e_i,X)-(\tfrac{\partial}{\partial \theta_i}F)(\theta,X)\big|}\,\mathrm{d}u\Bigg]\\
&\leq\limsup_{n\to \infty} \Bigg[ \sup_{u \in[-|h_n|,|h_n|]} \EXPP{\big|(\tfrac{\partial}{\partial \theta_i}F)(\theta + u e_i,X)-(\tfrac{\partial}{\partial \theta_i}F)(\theta,X)\big|}\Bigg]
=0.
\end{split}
\end{equation}
Next observe that \eqref{le:derivative-expectation-interchange:eqUI}, the fact that for all $i \in \{1,2,\dots,d\}$, $x\in S$ it holds that the function  $\R^d \ni \theta \mapsto (\frac{\partial}{\partial \theta_i}F)(\theta, x) \in \R$ is continuous, and the Vitali  convergence theorem (cf., e.g, Klenke \cite[Theorem~6.25]{Klenke}) assure that for all $i\in\{1,2,\dots,d\}$, $\theta \in \R^d$ and all sequences $v=(v_n)_{n\in\N} \colon \N \to \R^d$ with $\limsup_{n \to \infty} \norm{v_n}=0$ it holds that 
\begin{equation}
\begin{split}
& \limsup_{n \to \infty} \big| \EXP{(\tfrac{\partial}{\partial \theta_i} F) (\theta+v_n,X)}-\EXP{(\tfrac{\partial}{\partial \theta_i} F) (\theta,X)}\!\big|\\
&\leq 
\limsup_{n \to \infty}  \EXPP{\big|(\tfrac{\partial}{\partial \theta_i} F) (\theta+v_n,X)-(\tfrac{\partial}{\partial \theta_i} F) (\theta,X)\big|}=0.
\end{split}
\end{equation} 
Combining this and \eqref{V0} establishes  items~(\ref{le:derivative-expectation-interchange:item1}) and (\ref{le:derivative-expectation-interchange:item2}). The proof of Lemma~\ref{le:derivative-expectation-interchange} is thus completed.
\end{proof}
%
%
%
%
\begin{cor} 
\label{SGD}
Let $d \in \N$, $p \in \{2,4,6,\ldots \}$, $\alpha,\kappa,c \in (0,\infty)$, $ \nu \in (0,1)$, $\xi, \vartheta \in \R^d$, let 
$\lll \cdot,\cdot \rrr \colon \R^d \times \R^d \to \R$
 be the $d$-dimensional Euclidean scalar product, let 
 $\left \| \cdot \right \| \! \colon \R^d \to [0,\infty)$ 
 be the function which satisfies for all $\theta \in \R^d$ that $\norm{\theta} = \sqrt{\lll \theta,\theta \rrr}$, let $ ( \Omega , \mathcal{F}, \P) $ be a probability space, let $(S, \mathcal{S})$ be a measurable space, let $X_{n} \colon \Omega \to S$, $n \in \N$, be i.i.d.\ random variables,
let $F = ( F(\theta,x) )_{\theta \in \R^d, x \in S} \colon \R^d \times S \to \R$ be $(\mathcal{B}(\R^{d}) \otimes \mathcal{S}) / \mathcal{B}(\R)$-measurable, assume for all $x \in S$ that $(\R^d \ni \theta \mapsto F(\theta,x) \in \R) \in C^1(\R^d, \R)$, assume for all $\theta \in \R^d$ that
\begin{equation}
\label{SGD:assumption2}
\EXP{|F(\theta,X_1)| + \norm{(\nabla_{\theta}F)(\theta, X_1)}}<\infty, 
\end{equation}
\begin{equation}
\label{SGD:assumption3}
\lll \theta - \vartheta, \Exp{(\nabla_{\theta}F)(\theta, X_1)} \rrr   \geq c \max\!\big\{\norm{\theta - \vartheta}^2, \norm{\Exp{(\nabla_{\theta}F)(\theta, X_1)}\!}^2\big\},  
\end{equation}
\begin{equation}
\label{SGD:assumption4}
\EXP{\norm{(\nabla_{\theta}F)(\theta, X_1) - \Exp{(\nabla_{\theta}F)(\theta, X_1)}\!}^p } \leq \kappa \big(1+ \norm{\theta}^p\big),
\end{equation}
and let $\Theta \colon \N_0 \times \Omega \to \R^d$ be the stochastic process which satisfies for all $n \in \N$ that 
\begin{equation}
\label{SGD:assumption5}
\Theta_0 = \xi   \qquad \text{ and } 
\qquad \Theta_n = \Theta_{n-1} - \tfrac{\alpha}{n^\nu}(\nabla_{\theta}F)(\Theta_{n-1},X_n).
\end{equation}
Then 
\begin{enumerate}[(i)]
	\item \label{SGD:item1}
	it holds that $\big\{\theta \in \R^d \colon \big(\Exp{F(\theta, X_1)} = \inf\nolimits_{v \in \R^d}\Exp{F(v, X_1)}\!\big)\!\big\} = \{\vartheta\}$ and 
	\item \label{SGD:item2}
	there exists $C \in (0,\infty)$ such that for all $n \in \N$ it holds that
	\begin{equation}
	\left(\EXP{ \norm{\Theta_n-\vartheta}^p }\right)^{\nicefrac{1}{p}} \leq C n^{-\nicefrac{\nu}{2}}.
	\end{equation}
\end{enumerate} 
\end{cor}
\begin{proof}[Proof of Corollary \ref{SGD}]
Throughout this proof let $f\colon \R^d \to \R$ be the function which satisfies for all $\theta \in \R^d$ that
\begin{equation}
f(\theta)=\Exp{F(\theta,X_1)}.
\end{equation}
Lemma~\ref{le:derivative-expectation-interchange} assures that for all $\theta \in \R^d$ it holds that
\begin{equation}
\label{SGD:derivative-expectation-interchange}
f \in C^1(\R^d,\R) \qandq (\nabla f)(\theta)=\EXP{(\nabla_\theta F)(\theta,X_1)}.
\end{equation}
This and (\ref{SGD:assumption3}) imply that for all $\theta \in \R^d$ it holds that
\begin{equation}
\lll \theta - \vartheta, (\nabla f)(\theta) \rrr   \geq c \max\!\big\{\norm{\theta - \vartheta}^2, \norm{(\nabla f)(\theta)}^2\!\big\}.
\end{equation}
Hence, we obtain that for all $\theta \in \R^d$ it holds that
\begin{equation}
\lll \theta - \vartheta, (\nabla f)(\theta) \rrr   \geq c \norm{\theta - \vartheta}^2.
\end{equation}
This proves that for all $v \in \R^d$ it holds that
\begin{equation}
\lll v, (\nabla f)(\vartheta+v) \rrr   \geq c \norm{v}^2.
\end{equation}
The fundamental theorem of calculus therefore ensures that for all $\theta \in \R^d$ it holds that
\begin{equation}
\begin{split}
f(\theta) &= f(\vartheta) + \big[f(\vartheta+t(\theta-\vartheta))\big]_{t=0}^{t=1} \\
&= f(\vartheta) + \int_0^1 f'(\vartheta+s(\theta-\vartheta))(\theta-\vartheta)\,\mathrm{d}s\\
& = f(\vartheta) + \int_0^1 \llll (\nabla f)\big(\vartheta+s(\theta-\vartheta)\big),\theta-\vartheta \rrrr\,\mathrm{d}s \\
& = f(\vartheta) + \int_0^1 \frac{1}{s}\,\llll (\nabla f)\big(\vartheta+s(\theta-\vartheta)\big),s(\theta-\vartheta) \rrrr\,\mathrm{d}s \\
&\geq  f(\vartheta) + \int_0^1 \frac{c}{s}\,\norm{s(\theta-\vartheta)}^2 \,\mathrm{d}s \\
&= f(\vartheta) + \frac{c}{2}\,\norm{\theta-\vartheta}^2.
\end{split}
\end{equation}
The hypothesis that $c\in (0,\infty)$ hence demonstrates that for all $\theta \in \R^d\setminus\{\vartheta\}$ it holds that
\begin{equation}
f(\theta)\geq f(\vartheta)+\frac{c}{2}\norm{\theta-\vartheta}^2>f(\vartheta).
\end{equation}
This establishes item (\ref{SGD:item1}). 
Moreover, observe that  Corollary \ref{convergence_of_SA-Vers-G-F-2} (with $F = -F$ 
in the notation of Corollary \ref{convergence_of_SA-Vers-G-F-2}) establishes item (\ref{SGD:item2}).
The proof of Corollary \ref{SGD} is thus completed.
\end{proof}
\subsection{Stochastic approximation for linear regression}
\label{subsection:LinReg}
%
\begin{lemma}
	\label{lem:symmatrixbound}
	Let $d\in\N$, let $\lll \cdot,\cdot\rrr\colon\R^d\times\R^d\to\R$ be
	the $d$-dimensional Euclidean scalar product, let
	$\left \| \cdot \right \| \! \colon\R^d\to [0,\infty)$ be the function which satisfies for all
	$\theta\in\R^d$ that $\norm{\theta}=\sqrt{\lll\theta,\theta\rrr}$,
	let $A\in\R^{d\times d}$,
	assume that $A$ is invertible and symmetric, and
	assume for all $\theta\in\R^d$ that $\lll\theta,A\theta\rrr\geq 0$.
	Then  there exists $c\in(0,\infty)$ such that
	for all $\theta\in\R^d$ it holds that
	\begin{equation}
	\label{symmatrixbound:claim}
	\lll\theta,A\theta\rrr\geq c\max\{\norm{\theta}^2,\norm{A\theta}^2\}.
	\end{equation}
\end{lemma}
\begin{proof}[Proof of Lemma \ref{lem:symmatrixbound}]
	Throughout this proof 
	for every $v\in\R^d$ let $v^*\in\R^{1\times d}$ be the transpose of $v$,
	for every $M\in\R^{d\times d}$ let $M^*\in\R^{d\times d}$ be the transpose
	of $M$,
	let 
	$e_1=(1,0,\dots,0)$, $e_2=(0,1,0,\dots,0)$, $\dots$, $e_d=(0,\dots,0,1)\in\R^d$,
	 let $E\in\R^{d\times d}$ be the $(d\times d)$-identity matrix,
	let 
	$T \in \R^{d \times d}$ and  $D=(\delta_{i,j})_{(i,j)\in\{1,\dots,d\}^2}\in\R^{d\times d}$ be
	$(d\times d)$-matrices
	such that $D$ is a diagonal matrix and such that
	\begin{equation}
	\label{symmatrixbound:diagA}
	TT^*=E\qquad\text{and}\qquad A=TDT^*,
	\end{equation}
	let $c_0=\min_{i\in\{1,\dots,d\}} \delta_{i,i}\in\R$, and let $c_1=\max_{i\in\{1,\dots,d\}} \delta_{i,i}\in\R$.
	Note that \eqref{symmatrixbound:diagA} implies that for all $\theta\in\R^d$ it holds that
	\begin{equation}
	\label{symmatrixbound:orthogTstar}
	 \norm{T^* \theta}^2=\lll T^*\theta,T^*\theta\rrr=(T^*\theta)^*(T^*\theta)=\theta^*TT^*\theta
	=\theta^*E\theta=\lll\theta,\theta \rrr = \norm{\theta}^2.
	\end{equation}
	Furthermore, observe that \eqref{symmatrixbound:diagA} demonstrates that $T$ is invertible with $T^{-1}=T^*$. Hence, we obtain that for all $\theta\in\R^d$ it holds that
	\begin{equation}
	\label{symmatrixbound:orthogT}
	\norm{T \theta}^2=\lll T\theta,T\theta \rrr=(T\theta)^*(T\theta)=\theta^*T^*T\theta
	=\theta^*T^{-1}T\theta=\theta^*E\theta=\lll\theta,\theta \rrr=\norm{\theta}^2.
	\end{equation}
	Combining this with  \eqref{symmatrixbound:diagA} and the hypothesis that for all $\theta \in \R^d$ it holds that $\lll \theta, A\theta\rrr\geq 0$ ensures
	that for all $i\in\{1,\dots,d\}$ it holds that
	\begin{equation}
	\label{symmatrixbound:diigeq0}
	\begin{aligned}
	\delta_{i,i}&=\delta_{i,i}\lll e_i,e_i\rrr=\lll \delta_{i,i}e_i,e_i \rrr
	=\lll De_i,e_i \rrr=\lll TDe_i,Te_i \rrr\\
	&=\lll TDT^{-1}Te_i,Te_i \rrr
	=\lll A(Te_i),Te_i \rrr\geq 0.
	\end{aligned}
	\end{equation}
	Furthermore, observe that \eqref{symmatrixbound:diagA}, the fact that $A$ is invertible, and the fact
	that $T$ is invertible prove that
	for all $i\in\{1,\dots,d\}$ it holds that
	\begin{equation}
	\delta_{i,i}e_i=De_i=T^{-1}ATe_i\neq 0.
	\end{equation}
	This assures that for all $i\in\{1,\dots,d\}$ it holds that
	\begin{equation}
	\delta_{i,i}\neq 0.
	\end{equation}
	Combining this with \eqref{symmatrixbound:diigeq0} shows that for all $i\in\{1,\dots,d\}$ it holds that
	\begin{equation}
	\delta_{i,i}>0.
	\end{equation}
	Hence, we obtain that
	\begin{equation}
	\label{symmatrixbound:infpos}
	0<c_0=\min_{i\in\{1,2,\dots,d\}} \delta_{i,i}\leq \max_{i \in\{1,2,\dots,d\}} \delta_{i,i}= c_1<\infty.
	\end{equation}
	Next note that for all $\theta=(\theta_1,\dots,\theta_d)\in\R^d$ it holds that
	\begin{equation}
	\label{symmatrixbound:diaglower}
	\lll\theta,D\theta \rrr=\sum_{i=1}^d \big(\theta_i(\delta_{i,i}\theta_i)\big)
	=\sum_{i=1}^d \big(\delta_{i,i} (\theta_i)^2\big)
	\geq c_0\left[\sum_{i=1}^d (\theta_i)^2\right]
	=c_0\lll\theta,\theta \rrr=c_0 \norm{\theta}^2.
	\end{equation}
	Moreover, observe that \eqref{symmatrixbound:infpos} ensures that for all $\theta=(\theta_1,\dots,\theta_d)\in\R^d$ it holds that
	\begin{equation}
	\label{symmatrixbound:diagupper}
	\begin{split}
	\norm{D\theta}^2&=\lll D\theta,D\theta \rrr=\sum_{i=1}^d \big((\delta_{i,i}\theta_i)(\delta_{i,i}\theta_i)\big)
	=\sum_{i=1}^d \big((\delta_{i,i})^2 (\theta_i)^2\big)\\
	&\leq (c_1)^2\left[\sum_{i=1}^d (\theta_i)^2\right]
	=(c_1)^2\lll\theta,\theta \rrr= (c_1)^2 \norm{\theta}^2.
	\end{split}
	\end{equation}
	Furthermore, note that \eqref{symmatrixbound:diagA} implies that for all $\theta\in\R^d$ it holds that
	\begin{equation}
	\lll\theta,A\theta \rrr=\lll\theta,TDT^*\theta\rrr
	=\theta^*TDT^*\theta
	=(T^*\theta)^*D(T^*\theta)
	=\lll T^*\theta,D(T^*\theta) \rrr.
	\end{equation}
	This, \eqref{symmatrixbound:orthogTstar}, and \eqref{symmatrixbound:diaglower} ensure that
	for all $\theta\in\R^d$ it holds that
	\begin{equation}
	\label{symmatrixbound:Alower}
	\begin{aligned}
	&\lll\theta,A\theta \rrr
	\geq c_0 \norm{T^*\theta}^2
	=c_0\norm{\theta}^2.
	\end{aligned}
	\end{equation}
	%
	Furthermore, observe that  for all $\theta \in \R^d$ it holds that
	\begin{equation}
	\norm{A \theta} \leq \bigg[ \sup_{v \in \R^d\setminus\{0\}} \frac{\norm{Av}}{\norm{v}} \bigg] \! \norm{\theta}.
	\end{equation}
	Hence, we obtain that for all $\theta \in \R^d$ it holds that
	\begin{equation}
	\begin{split}
	\left[ \inf_{v \in \R^d\setminus\{0\}} \frac{\norm{A^{-1}v}}{\norm{v}} \right] \! \norm{A\theta}
	=\left[ \inf_{v \in \R^d\setminus\{0\}} \frac{\norm{v}}{\norm{Av}} \right] \! \norm{A\theta}
		= \frac{\norm{A\theta}}{\left[ \sup_{v \in \R^d\setminus\{0\}} \frac{\norm{Av}}{\norm{v}}
	 \right]}
 \leq \norm{\theta}.
	\end{split}
	\end{equation}
	This shows that for all $\theta \in \R^d$ it holds that
	\begin{equation}
	\norm{\theta}^2 \geq \left[ \inf_{v \in \R^d \setminus \{0\}} \frac{ \norm{A^{-1} v}}{\norm{v}} \right]^2 \! \norm{A \theta}^2\geq \left[ \min\!\Big\{1,\left[\inf\nolimits_{v \in \R^d \setminus\{0\}}\tfrac{ \norm{A^{-1} v}}{\norm{v}} \right]^2\Big\} \right] \! \norm{A \theta}^2.
	\end{equation}
	Hence, we obtain that for all $\theta \in \R^d$ it holds that
	\begin{equation}
	\norm{\theta}^2 
	\geq
	 \left[\min\!\Big\{1,\inf\nolimits_{v \in \R^d \setminus\{0\}}\tfrac{ \norm{A^{-1} v}^2}{\norm{v}^2}\Big\} \right]  \max\!\big\{\norm{\theta}^2,\norm{A \theta}^2\big\}.
	\end{equation}
	Combining this with  \eqref{symmatrixbound:Alower} demonstrates that for all $\theta \in \R^d$ it holds that 
	\begin{equation}
	\lll \theta, A \theta \rrr 
	\geq
	\left[c_0\min\!\Big\{1,\inf\nolimits_{v \in \R^d \setminus\{0\}}\tfrac{ \norm{A^{-1} v}^2}{\norm{v}^2}\Big\} \right]  \max\!\big\{\norm{\theta}^2,\norm{A \theta}^2\big\}.
	\end{equation} 
The proof of Lemma \ref{lem:symmatrixbound} is thus completed.
\end{proof}
%
%
\begin{cor}[Stochastic approximation for linear regression]
\label{SA_for_linear_regression}
Let $d \in \N$, $p \in \{2,4,6,\ldots\}$, $\alpha \in (0,\infty)$, $\nu \in (0,1)$, $\xi \in \R^d$, let 
$\lll \cdot,\cdot \rrr \colon \R^d \times \R^d \to \R$
be the $d$-dimensional Euclidean scalar product, let
$\left \| \cdot \right \| \! \colon \R^d \to [0,\infty)$ 
be the function which satisfies for all $\theta \in \R^d$ that $\norm{\theta} = \sqrt{\lll \theta,\theta \rrr}$,  let $ ( \Omega , \mathcal{F}, \P) $ be a probability space, let $X_n \colon \Omega \to \R^d$, 
 $n \in \N$, be i.i.d.\ random variables, let $h \colon \R^d \to \R$ be $\mathcal{B}(\R^{d}) / \mathcal{B}(\R)$-measurable,  assume that $\EXP{\norm{h(X_1)X_1}^p + \abs{h(X_1)}^2 +\norm{X_1}^{2p}} < \infty$,
for every  $v \in \R^d$ let $v^* \in \R^{1\times d}$ be the transpose of $v$,
assume that $\Exp{X_1(X_1)^*} \in \R^{d\times d}$ is invertible, let $\vartheta \in \R^d$ satisfy
\begin{equation}
\label{SA_for_linear_regression:assumption2}
\vartheta = \left(\E\big[X_1(X_1)^*\big]\right)^{-1}\E\big[h(X_1)X_1\big],
\end{equation} 
let $F = ( F(\theta,x) )_{\theta \in \R^d, x \in \R^d}\colon \R^d \times \R^d \to \R$ 
be the function which satisfies for all $\theta,x \in \R^d$ that $F(\theta, x) =  [\lll \theta, x \rrr- h(x)]^2$,
and let $\Theta \colon \N_0 \times \Omega \to \R^d$ be the stochastic process which satisfies for all $n \in \N$ that 
\begin{equation}
\label{SA_for_linear_regression:assumption3}
\begin{split}
\Theta_0 = \xi \qandq \Theta_{n} &= \Theta_{n-1} - \tfrac{\alpha}{n^\nu}(\nabla_{\theta}F)(\Theta_{n-1},X_n). 
 \end{split}
\end{equation}
Then 
\begin{enumerate}[(i)]
\item \label{SA_for_linear_regression:item0}
it holds for all $\theta \in \R^d$ that  $\EXP{\abs{F(\theta,X_1)}}<\infty$,
\item \label{SA_for_linear_regression:item1}
it holds that $\big\{\theta \in \R^d \colon \big(\Exp{F(\theta, X_1)} = \inf\nolimits_{v \in \R^d}\Exp{F(v, X_1)}\!\big)\!\big\} = \{\vartheta\}$, and
\item \label{SA_for_linear_regression:item2}
there exists $C \in (0,\infty)$ such that for all $n \in \N$ it holds that
\begin{equation}
\left(\EXP{ \norm{\Theta_n-\vartheta}^p }\right)^{\nicefrac{1}{p}} \leq C n^{-\nicefrac{\nu}{2}}.
\end{equation}
\end{enumerate} 
\end{cor}

\begin{proof}[Proof of Corollary \ref{SA_for_linear_regression}]
%
Observe that the chain rule and the hypothesis that for all $\theta,x \in \R^d$ it holds that $F(\theta, x) = [ \lll \theta, x \rrr- h(x)]^2$ ensure that for all $\theta,x,v \in \R^d$ it holds that
\begin{equation}
\label{SA_for_linear_regression:eqWWW}
(\tfrac{\partial}{\partial\theta}F)(\theta,x)(v)
= 2[ \lll \theta, x \rrr- h(x)] \lll v, x \rrr.
\end{equation}
Hence, we obtain that for all $\theta,x \in \R^d$ it holds that
\begin{equation}
\label{SA_for_linear_regression:eq0}
(\nabla_{\theta}F)(\theta, x) 
= 2[ \lll \theta, x \rrr- h(x)] x
= 2 x \lll x,\theta \rrr - 2 h(x) x
=2xx^* \theta - 2 h(x)x.
\end{equation}
This, the hypothesis that $\EXP{\norm{h(X_1)X_1}^p + \abs{h(X_1)}^2 +\norm{X_1}^{2p}} < \infty$, the Cauchy-Schwarz inequality, and Jensen's inequality assure that for all $\theta \in \R^d$ it holds that
\begin{equation}
\label{SA_for_linear_regression:eqInteg2}
\begin{split}
\EXP{\norm{(\nabla_\theta F)(\theta, X_1)}}
&=
2\, \EXP{\norm{X_1(X_1)^*\theta - h(X_1)X_1}}\\
& \leq 2\, \EXP{\norm{X_1(X_1)^*\theta}}  +2\, \EXP{\norm{h(X_1) X_1}}\\
&=2\, \EXP{\norm{X_1 \lll X_1,\theta \rrr}}+2\, \EXP{\norm{h(X_1) X_1}}\\
&=2\, \EXP{|\lll X_1,\theta \rrr|\norm{X_1}}+2\, \EXP{\norm{h(X_1) X_1}}\\
&\leq 2\, \EXP{\norm{X_1} \norm{\theta}\norm{X_1}}+2\, \EXP{\norm{h(X_1) X_1}}\\
&= 2\, \norm{\theta} \,\Exp{\norm{X_1}^2} +2\, \EXP{\norm{h(X_1) X_1}}<\infty.
\end{split}
\end{equation}
Next observe that the hypothesis that for all $\theta,x \in \R^d$ it holds that $F(\theta, x) = [\lll \theta, x \rrr- h(x)]^2$ implies that for all $\theta,x \in \R^d$ it holds that
\begin{equation}
\begin{split}
F(\theta,x)
&=
[\lll \theta, x\rrr -h(x)][\lll \theta, x\rrr -h(x)]\\
&= 
|\lll \theta, x\rrr|^2 -2 \lll \theta, x\rrr h(x) +|h(x)|^2\\
&=\lll \theta, x \rrr \lll x,\theta \rrr -2 \lll \theta, h(x) x \rrr + |h(x)|^2\\
&=\theta^*xx^*\theta -2 \theta^*h(x)x  +|h(x)|^2.
\end{split}
\end{equation}
This, the triangle inequality, the Cauchy-Schwarz inequality, Jensen's inequality, 
and the hypothesis that $\EXP{\norm{h(X_1)X_1}^p + \abs{h(X_1)}^2 +\norm{X_1}^{2p}} < \infty$ ensure that for all $\theta \in \R^d$ it holds that 
\begin{equation}
\label{SA_for_linear_regression:eqInteg1}
\begin{split}
&\EXP{|F(\theta, X_1)|}\\
&=\EXPP{\big| [\lll\theta,X_1\rrr]^2 - 2 \lll \theta, h(X_1) X_1 \rrr + [h(X_1)]^2 \big|}\\
&\leq \EXP{ |\lll\theta,X_1\rrr |^2 + 2 | \lll \theta, h(X_1) X_1 \rrr | + |h(X_1)|^2}\\
&\leq \EXP{ \norm{\theta}^2 \norm{X_1}^2 + 2 \norm{\theta} \norm{h(X_1) X_1} + |h(X_1)|^2}\\
&\leq \big(1 + 2 \norm{\theta} + \norm{\theta}^2\big)\Big(\EXP{\norm{X_1}^{2}} + \EXP{\norm{h(X_1) X_1}} + \EXP{|h(X_1)|^2}\!\Big)<\infty.
\end{split}
\end{equation}
Next note that (\ref{SA_for_linear_regression:assumption2}), (\ref{SA_for_linear_regression:eq0}), and  \eqref{SA_for_linear_regression:eqInteg2} imply that for all $\theta \in \R^d$ it holds that
\begin{equation}
\label{SA_for_linear_regression:eq1}
\begin{split}
\EXP{(\nabla_{\theta}F)(\theta, X_1)}&= 2\,\EXP{X_1(X_1)^*}\theta -2\,\EXP{h(X_1)X_1}\\
&=2\,\EXP{X_1(X_1)^*}\theta-2\,\EXP{X_1 (X_1^*)} \big(\EXP{X_1 (X_1^*)}\big)^{-1} \EXP{h(X_1) X_1}\\
&=  2\,\E\big[X_1(X_1)^*\big]  \theta - 2\,\E\big[X_1(X_1)^*\big] \vartheta \\
&=  2\,\E\big[X_1(X_1)^*\big] (\theta-\vartheta).
\end{split}
\end{equation}
In addition, observe that for all $ \theta \in \R^d$ it holds that
\begin{equation}
\big\langle\theta, 2\,\E\big[X_1(X_1)^*\big]\theta \big\rangle = 2\, \EXP{\theta^*X_1(X_1)^*\theta} = 2\, \EXP{\abs{\theta^*X_1}^2 } \geq 0.
\end{equation}
The fact that $2\,\EXp{X_1(X_1)^*}$ is symmetric and invertible, \eqref{SA_for_linear_regression:eq1}, and Lemma~\ref{lem:symmatrixbound} hence ensure that there exists $c\in(0,\infty)$ such that for all $\theta \in \R^d$ it holds that
\begin{equation}
\label{SA_for_linear_regression:eq2}
\begin{split}
\llll \theta- \vartheta, \Exp{(\nabla_\theta F)(\theta,X_1)} \!\rrrr 
& = \llll \theta- \vartheta, (2\,\EXp{X_1(X_1)^*})(\theta-\vartheta) \!\rrrr \\
& \geq c\,\max\!\big\{\norm{\theta-\vartheta}^2, \norm{(2\,\EXp{X_1(X_1)^*})(\theta-\vartheta)}^2 \big\}\\
& = c\,\max\!\big\{\norm{\theta-\vartheta}^2, \norm{\Exp{(\nabla_\theta F)(\theta,X_1)}\!}^2 \big\}.
\end{split}
\end{equation}
This and item (\ref{fcond:item2}) in Lemma \ref{fcond} (with $g(\theta) = -\Exp{(\nabla_{\theta}F)(\theta, X_1)}$ for $\theta \in \R^d$ in the notation of Lemma \ref{fcond}) ensure that there exists $\mathfrak{C}\in (0,\infty)$ such that for all $\theta \in \R^d$ it holds that
\begin{equation}
\norm{\Exp{(\nabla_{\theta}F)(\theta, X_1)}\!} \leq \mathfrak{C} \norm{\theta-\vartheta} \leq \mathfrak{C} \norm{\vartheta} +\mathfrak{C} \norm{\theta}.
\end{equation}
 Lemma \ref{weak_triangle} and (\ref{SA_for_linear_regression:eq0}) hence imply that for all $\theta \in \R^d$ it holds that
 \begin{equation}
 \begin{split}
 &\E\bigl[\norm{(\nabla_\theta F)(\theta,X_1)-\E[(\nabla_\theta F)(\theta,X_1)]}^p\bigr]\\
 &\leq \E\bigl[2^p\bigl(\norm{(\nabla_\theta F)(\theta,X_1)}^p+\norm{\E[(\nabla_\theta F)(\theta,X_1)]}^p\bigr)\bigr]\\
 &\leq2^p\bigl(\E\big[\norm{2X_1(X_1)^*\theta-2h(X_1)X_1}^p\big]+(\mathfrak C\norm\vartheta+\mathfrak C\norm\theta)^p\bigr)\\
 &\leq2^p\bigl(2^p\,\E\big[\norm{X_1(X_1)^*\theta-h(X_1)X_1}^p\big]+2^p(\mathfrak C^p\norm\vartheta^p+\mathfrak C^p\norm\theta^p)\bigr)\\
 &=2^{2p}\bigl(\E\big[\norm{X_1(X_1)^*\theta-h(X_1)X_1}^p\big]+\mathfrak C^p\norm\vartheta^p+\mathfrak C^p\norm\theta^p\bigr).
 \end{split}
 \end{equation}
 The triangle inequality, Lemma~\ref{weak_triangle}, and the Cauchy-Schwarz inequality therefore demonstrate that for all $\theta\in\R^d$ it holds that 
 \begin{equation}
 \begin{split}
 &\E\bigl[\norm{(\nabla_\theta F)(\theta,X_1)-\E[(\nabla_\theta F)(\theta,X_1)]}^p\bigr]\\
 &\leq2^{2p}\Bigl(\E\bigl[2^p\bigl(\norm{X_1(X_1)^*\theta}^p+\norm{h(X_1)X_1}^p\bigr)\bigr]+\mathfrak C^p\norm\vartheta^p+\mathfrak C^p\norm\theta^p\Bigr)\\
 &\leq2^{2p}\Bigl(2^p\,\E\bigl[\norm\theta^p\norm{X_1}^{2p}+\norm{h(X_1)X_1}^p\bigr]+\mathfrak C^p\norm\vartheta^p+\mathfrak C^p\norm\theta^p\Bigr)\\
 &\leq2^{3p}\Bigl(\E\bigl[\norm{X_1}^{2p}\bigr]\norm\theta^p+\E\bigl[\norm{h(X_1)X_1}^p\bigr]+\mathfrak C^p\norm\vartheta^p+\mathfrak C^p\norm\theta^p\Bigr)\\
 &\leq2^{3p}\Bigl(\E\bigl[\norm{X_1}^{2p}\bigr]+\E\bigl[\norm{h(X_1)X_1}^p\bigr]+\mathfrak C^p\norm\vartheta^p+\mathfrak C^p\Bigr)(1+\norm\theta^p)\\
 &\leq \bigg[2^{3p+2}\max\Bigl\{\E\bigl[\norm{X_1}^{2p}\bigr],\E\bigl[\norm{h(X_1)X_1}^p\bigr],\mathfrak C^p\norm\vartheta^p,\mathfrak C^p\Bigr\}\bigg](1+\norm\theta^p).
 \label{cor49:new1}
 \end{split}
 \end{equation}
Combining this, the fact that for all $x \in \R^d$ it holds that $(\R^d \ni \theta \mapsto F(\theta,x) \in \R) \in C^1(\R^d, \R)$, \eqref{SA_for_linear_regression:assumption3}, \eqref{SA_for_linear_regression:eqInteg2}, \eqref{SA_for_linear_regression:eqInteg1},
and (\ref{SA_for_linear_regression:eq2}) with  Corollary~\ref{SGD} (with $\kappa=2^{3p+2}\max\bigl\{$ $\E[\norm{X_1}^{2p}], \E[\norm{h(X_1)X_1}^p],\mathfrak C^p\norm\vartheta^p,\mathfrak C^p\bigr\}$, $c=c$ 
in the notation of Corollary~\ref{SGD}) establishes items (\ref{SA_for_linear_regression:item0}), (\ref{SA_for_linear_regression:item1}), and (\ref{SA_for_linear_regression:item2}).
The proof of Corollary \ref{SA_for_linear_regression} is thus completed.
\end{proof}

\nocite{*} 

\bibliographystyle{acm}
\bibliography{bibfile_semesterpaper}

\end{document}